\documentclass[11pt, oneside, reqno]{amsart}

\usepackage{graphicx}
\usepackage{enumitem}
\usepackage{relsize}
\usepackage{amsmath,amssymb,amsthm,amsfonts, cite}
\usepackage{mathtools,accents}
\usepackage{mathrsfs}
\usepackage{aliascnt}
\usepackage{braket}
\usepackage{bm}
\usepackage{color}
\usepackage{textcomp}
\usepackage{xcolor,graphicx, txfonts}
\usepackage{abstract}
\usepackage{hyperref, cleveref}
\usepackage{setspace}
\usepackage{fancyhdr, changepage}
\usepackage{breqn}

\theoremstyle{plain}
\newtheorem{thm}{Theorem}[section]

\newtheorem{lem}[thm]{Lemma}

\newtheorem{cor}[thm]{Corollary}

\newtheorem{prop}[thm]{Proposition}

\theoremstyle{definition}
\newtheorem{de}[thm]{Definition}

\theoremstyle{remark}
\newtheorem{rem}[thm]{Remark}

\DeclareMathOperator{\Hess}{\rm Hess}
\DeclareMathOperator{\diver}{\rm div}
\DeclareMathOperator{\supp}{supp}
\DeclareMathOperator{\diam}{diam}
\DeclareMathOperator{\lip}{\rm lip}
\DeclareMathOperator{\RCD}{\sf RCD}
\DeclareMathOperator{\CD}{\sf CD}
\DeclareMathOperator{\MCP}{\sf MCP}
\DeclareMathOperator{\BE}{\rm BE}
\DeclareMathOperator{\Ric}{\rm Ric}
\DeclareMathOperator{\TestF}{\rm TestF}
\DeclareMathOperator{\TestV}{\rm TestV}
\DeclareMathOperator{\Tan}{\rm Tan}
\DeclareMathOperator{\reg}{reg}
\DeclareMathOperator{\Ch}{\sf Ch}

\DeclareMathOperator{\Ent}{\sf Ent}
\DeclareMathOperator{\Span}{Span}
\DeclareMathOperator{\HS}{HS}
\DeclareMathOperator*{\esssup}{ess\,sup}
\DeclareMathOperator*{\essinf}{ess\,inf}
\DeclareMathOperator*{\Geo}{\rm Geo}
\DeclareMathOperator*{\OptGeo}{\rm OptGeo}

\newcommand\inner[2]{\langle #1, #2 \rangle}
\newcommand\braces[1]{\left\{ #1 \right\}}
\newcommand\norm[1]{\| \, #1 \, \|}

\newcommand{\tPsi}{\tilde{\Psi}}
\newcommand{\gamd}{\gamma^{\delta}}
\newcommand{\gamdp}{\gamma^{\delta'}}
\newcommand{\lgam}{\bar{\gamma}}
\newcommand{\tgam}{\tilde{\gamma}}

\setlist[enumerate,1]{label={\arabic*.}, ref=\arabic*}

\makeatletter
\def\l@subsection{\@tocline{2}{0pt}{2.5pc}{5pc}{}}
\makeatother

\allowdisplaybreaks
\hypersetup{colorlinks}
\hypersetup{citecolor=blue}
\hypersetup{urlcolor=blue}

\begin{document}

\title{H\"older continuity of tangent cones in RCD(K,N) spaces and applications to non-branching}
\author{Qin Deng}\thanks{University of Toronto, qin.deng@mail.utoronto.ca} 

\begin{adjustwidth}{0.5cm}{0.5cm}
\maketitle
\end{adjustwidth}

\begin{abstract}
	In this paper we prove that a metric measure space $(X,d,m)$ satisfying the finite Riemannian curvature-dimension condition $\RCD(K,N)$ is non-branching and that tangent cones from the same sequence of rescalings are H\"older continuous along the interior of every geodesic in $X$. More precisely, we show that the geometry of balls of small radius centred in the interior of any geodesic changes in at most a H\"older continuous way along the geodesic in pointed Gromov-Hausdorff distance. This improves a result in the Ricci limit setting by Colding-Naber where the existence of at least one geodesic with such properties between any two points is shown. As in the Ricci limit case, this implies that the regular set of an $\RCD(K,N)$ space has $m$-a.e. constant dimension, a result already established by Bru\`e-Semola, and is $m$-a.e convex. It also implies that the top dimension regular set is weakly convex and, therefore, connected. In proving the main theorems, we develop in the $\RCD(K,N)$ setting the expected second order interpolation formula for the distance function along the Regular Lagrangian flow of some vector field using its covariant derivative. 
\end{abstract}

\renewcommand*\contentsname{\normalsize \bfseries Contents \normalfont}
\tableofcontents

\section{Introduction}
	
	In this paper, we prove that $\RCD(K,N)$ spaces are non-branching and generalize to the $\RCD(K,N)$ setting an improved version of the main result from Colding-Naber \cite{CN12}.  We begin by stating two formulations of the latter, which is be the main technical result of this paper. 
	
\begin{thm}\label{main theorem 1}(H\"older continuity of geometry of small balls with same radius)
Let $(X,d,m)$ be an $\RCD(K,N)$ space for some $K \in \mathbb{R}$ and $N \in (1, \infty)$. Let $p, q \in X$ and $d(p,q)=\ell$. Define $K' = (\frac{K}{-(N-1)} \vee 1)^{1 \slash 2}$ and $\ell' = \ell \wedge 1$. For any unit speed geodesic $\gamma:[0,\ell] \to X$ between $p$ and $q$, there exist constants $C(N)$, $\alpha(N)$ and $r_0(N) > 0$ so that for any $\delta > 0$ with $0 < r < r_0\frac{\delta\ell'}{K'}$ and $\delta \ell < s < t < \ell - \delta \ell$, 
\begin{equation}
	d_{pGH}\bigg((B_{r}(\gamma(s)), \gamma(s)), (B_{r}(\gamma(t)), \gamma(t))\bigg) < \frac{CK'}{\delta \ell'}r|s-t|^{\alpha}.
\end{equation}
\end{thm}

In order to pass the result to tangents, we use the following terminology: Let $x_1,x_2 \in X$, $(Y, d_{Y}, m_{Y}, y)$ $\in \Tan(X,d,m,x_1)$ and $(Z, d_{Z}, m_{Z}, z) \in \Tan(X,d,m,x_2)$. We say $Y$ and $Z$ come from the \textit{same sequence of rescalings} if there exists $s_j \downarrow 0$ so that
\begin{equation}
	(X, s_j^{-1}d, m^{x_1}_{s_j}, x_1) \xrightarrow{pmGH} (Y, d_{Y}, m_{Y}, y) \; \; \text{and}\; \; (X, s_j^{-1}d, m^{x_2}_{s_j}, x_2) \xrightarrow{pmGH} (Z, d_{Z}, m_{Z}, z).
\end{equation} 
The following estimate on tangents from the same sequence of rescaling follows from Theorem \ref{main theorem 1}.
\begin{thm}\label{main theorem 2}(H\"older continuity of tangent cones)
In the notations of Theorem \ref{main theorem 1}, for any unit speed geodesic $\gamma:[0,\ell] \to X$ between $p$ and $q$, there exist constants $C(N)$, $\alpha(N) > 0$ so that if $(Y_s, d_{Y_s}, m_{Y_s}, y_s) \in \Tan(X, d, m, \gamma(s))$ and $(Y_t, d_{Y_t}, m_{Y_t}, y_t) \in \Tan(X,d,m,\gamma(t))$ come from the same squence of rescalings, then
\begin{equation}
	d_{pGH}\bigg((B_{r}(y_s), y_s), (B_{r}(y_t),y_t)\bigg) < \frac{CK'}{\delta \ell'}r|s-t|^{\alpha}
\end{equation}
for all $r > 0$. 
\end{thm}

To prove these we first construct at least one geodesic between any two points satisfying the conclusion of Theorem \ref{main theorem 1}, which is the main result of \cite{CN12}. We then use this construction to prove that $\RCD(K,N)$ spaces, and so, in particular, Ricci limit spaces, are non-braching in Subsection \ref{subsection 6.1}. 
\begin{thm}\label{no branch}
Let $(X,d,m)$ be an $\RCD(K,N)$ space for some $K \in \mathbb{R}$ and $N \in (1,\infty)$. $(X,d,m)$ is non-branching. 
\end{thm}
This has been a natural open problem for Ricci limits since the seminal work of Cheeger-Colding in \cite{CC96, CC97, CC00a, CC00b}. Potential branching which can come from some simple examples were ruled out in \cite[Section 5]{CC00a} and some partial results were obtained in \cite{CN12} for noncollapsed limits of manifolds with uniform two-sided Ricci curvature bounds. To compensate for the lack of non-branching, the weaker notion of essentially non-branching was introduced and shown to hold for $\RCD(K,\infty)$ spaces in \cite{RS14}. This has found wide use in $\RCD$ theory and serves as a suitable replacement for the non-branching condition in most measure theoretic arguments. We point out finally that non-branching would follow for Ricci limits directly from the results of \cite{CN12} and the argument we use if one were able to, for example, prove that all geodesics in any Ricci limit space are limit geodesics. Our intrinsic construction of a geodesic in Subsection \ref{subsection 5.2} which satisfies the conclusion of Theorem \ref{main theorem 1} offers a little more freedom than the extrinsic construction of limit geodesics and this was enough to prove Theorem \ref{no branch}. Theorem \ref{no branch} is then used to pass from the existence of a geodesic between any two points which satisfies the conclusion of Theorem \ref{main theorem 1} to the theorem for all geodesics. 

In the case of Ricci limits, the H\"older continuity of tangent cones had several key applications. 
\begin{thm}\label{Ricci limit app}(\cite[Theorems 1.18, 1.20 and 1.21]{CN12})
Let $(X,d)$ be the Ricci limit of $(M_i^{n}, g_i)_{i \in \mathbb{N}}$ and $m$ be its canonical limit measure. The following holds:
\begin{enumerate}
	\item There is a unique $k \in \mathbb{N}$, $0 \leq k \leq n$ so that $m(X \backslash \mathcal{R}_k) = 0$, where $\mathcal{R}_k$ is the $k$-dimensional regular set;
	\item $\mathcal{R}_k$ from statement 1 is $m$-a.e. convex and weakly convex. In particular, $\mathcal{R}_k$ is connected;
	\item The isometry group of $X$ is a Lie group. 
\end{enumerate}
\end{thm}
Statements 1 and 3 have since been proved by other means in the case of $\RCD(K,N)$ spaces, see \cite{BS20} for 1 and \cite{S18, SRG19} for 3. We prove statement 2 in Subsection \ref{subsection 6.2} following \cite{CN12}. Since the proofs of statements 1 and 2 are intricately related, we will prove statement 1 as well. 

\subsection{Outline of Paper and Proof} We begin this subsection by introducing the strategy of the proof in \cite{CN12}, which we will largely follow. We then discuss the issues that arise when extending this to the metric measure setting and give an outline of their solutions. 

The existence of at least one geodesic satisfying the conclusion of Theorem \ref{main theorem 1} was shown in \cite{CN12}. The proof there was extrinsic and obtained by proving the theorem for manifolds. As such, consider a Riemannian manifold with $(M^n, g)$ with $\Ric_{M} \geq -(n-1)$ and a unit speed minimizing geodesic $\gamma: [0,1] \to M$ from some $p$ to $q$. Fix some $\delta > 0$ and $\delta < s_0 < s_1 < 1-\delta$. The desired Gromov-Hausdorff approximations in the proof of Theorem \ref{main theorem 1} for $\gamma$ are constructed on a large subset of the ball $B_{r}(\gamma(s_1))$ from the gradient flow $\Psi$ of $-d_p$ (more on the definition of this later). This is not altogether surprising since in the interior of $\gamma$, $d$ is a smooth function and the Laplacian of $d$ has a two-sided bound. A simple argument applying the Bochner formula to $d$ gives  
\begin{equation}\label{infinitesimal hess}
	\int_{\delta}^{1-\delta} |\Hess d_p|^2(\gamma(t)) \, dt \leq \frac{c(n)}{\delta}.
\end{equation}
The fact that $\Hess d_p = \frac{1}{2}\mathscr{L}_{\nabla d_p}g$ shows that $D(\Psi_{s_1-s_0}): T_{\gamma(s_1)}M \to T_{\gamma(s_0)}M$ satisfies the estimates of Theorem \ref{main theorem 2}. Of course, Theorem \ref{main theorem 2} is completely trivial in the case of Riemannian manifolds but the point is that this map comes from a construction on the manifolds itself. Smoothness then allows one to use $\Psi_{s_1-s_0}$ to construct Gromov-Hausdorff aproximations from $B_{r}(\gamma(s_1))$ to $B_{r}(\gamma(s_0))$, where $r$ needs to be sufficiently small depending on $\gamma$ and $\delta$ instead of just $n$ and $\delta$. This property does not pass through Ricci limits and so the challenge, then, is to remove the dependence on $\gamma$. 

One might consider using $\Hess d_p$ to control the geometry of balls of uniform (i.e. independent of $\gamma$) radius under $\Psi$. However, we do not have estimates on $\Hess d_p$ for such balls along $\gamma$; we do not even have smoothness of $d_p$. We mention that due to this lack of smoothness, $\Psi$ is not globally defined. The integral curves of $\Psi$ starting at any $x \in M$ should be thought of simply as a (choice of) unit speed geodesic from $x$ to $p$. In this way $\Psi$ is defined locally away from $p$ for a definite amount of time. 

Nevertheless, it turns out that one can still control the geometry under $\Psi$ by utilizing the next formula, which follows from the standard first variation formula and second order interpolation formula using the Hessian.  
\begin{equation}\label{CN interpolation}
\bigg| \frac{d}{dt} d(\sigma_1(t), \sigma_2(t)) \bigg| \leq |\nabla h - \sigma_1'|(\sigma_1(t))+|\nabla h - \sigma_2'|(\sigma_2(t)) + \inf \int_{\gamma_{\sigma_1(t), \sigma_2(t)}}|\Hess h|
\end{equation}
for a.e. $t$, where $\sigma_1, \sigma_2$ are unit speed geodesics in $M$, $h: M \to \mathbb{R}$ is a smooth function, and $\inf$ is taken across all minimizing geodesics connecting $\sigma_1(t), \sigma_2(t)$. This means as long as we can closely approximate $d_p$ by a smooth function $h$ where we have reasonable control on $|\nabla h - \nabla d_p|$ along two geodesics and on $\Hess h$, then we can control the distance between those two geodesics.

We mention a similar strategy was used to prove the almost splitting theorem in \cite{CC96}. In that setting, a single ball $B_{r}(\gamma(s))$ of small radius for a fixed $s \in (\delta,1-\delta)$ is considered. It turns out that the correct approximation to take for $d_p$ is the harmonic replacement $b$ of $d_p$ on $B_{2r}(\gamma(s))$. One is able to obtain the better than scale invariant estimate 
\begin{equation}\label{harmonic hess estimate}
	\fint_{B_{r}(\gamma(s))}|\Hess b|^2 \, dV \leq c(n,\delta)r^{-2+\alpha(n)}
\end{equation}
using the Bochner formula, which is enough to prove almost splitting. The almost splitting theorem has since been proved through other means for $\RCD$ spaces in \cite{G13}, see also \cite{MN19}. 

For our purposes, \eqref{harmonic hess estimate} and the resulting almost splitting theorem is not good enough because it only allows one to compare two balls of radius $r$ that are distance $r$ away from each other. As discussed in detail in \cite[Section 2]{CN12}, this estimate blows up as $r \to 0$ if one iterates along $\gamma$ in $r$-length intervals. The crucial idea in \cite{CN12} was then to use the heat flow approximation $h$ to (some cutoff of) $d_p$ instead. For such an approximation, they were able to obtain the estimate
\begin{equation}\label{heat flow hess estimate}
	\int_{\delta}^{1-\delta} \fint_{B_{r}(\gamma(t))}|\Hess h|^2 \, dV \, dt \leq c(n,\delta),
\end{equation}
where $h$ is the heat flow taken to some time on the scale of $r^2$ (see Theorem \ref{hess h integral bound}, statement 4). Moreover, $\int_{\delta}^{1-\delta} |\nabla d_p - \nabla h|$ can also be bounded to the correct order (see Theorem \ref{grad h geodesic bound}) for most geodesics. These estimates can then be used along with the segment inequality of Cheeger-Colding \cite[Theorem 2.11]{CC96} and \eqref{CN interpolation} to control the total integral change in distances between elements of two sets of large measure in $B_r(\gamma(s_1))$ under the flow $\Psi$. This is ultimately good enough to construct a Gromov-Hausdorff approximation using $\Psi_{s_1-s_0}$. We mention that since we are using segment inequality and integral bounds, the smaller the relative measure of the sets compared to the region where we have Hessian estimates, the worse the control we have on total distance change for those sets under $\Psi$.  


A crucial detail in this is that in order to make use of estimate \eqref{heat flow hess estimate}, it is important that most of $B_r(\gamma(t))$ stays close, on the scale of $r$, to $\gamma$ under the gradient flow for an amount of time independent of $r$ and $\gamma$. Since the control one has over distance is for sets of large relative measure and $\gamma$ is trivial in measure, one cannot guarantee using the argument outlined in the previous paragraph that most of $B_r(\gamma(t))$ does not simply drift away from $\gamma$ quickly. In \cite{CN12}, this was overcome by using \eqref{infinitesimal hess}. As mentioned previously, by smoothness, \eqref{infinitesimal hess} implies balls of sufficiently small radius depending on $\gamma$ stays close to $\gamma$ under $\Psi$ for some fixed amount of time depending only on $n$ and $\delta$. Induction with geometrically increasing radii along with the argument from the previous paragraph can then be used to guarantee large proportions of balls up to some radius independent of $\gamma$ also stay close to $\gamma$ under the flow $\Psi$.

We now outline the issues with extending this argument to the metric measure setting and the ideas we will use to resolve them:

\begin{itemize}[leftmargin=0.4cm]
	\item \eqref{CN interpolation} is essential in utilizing the Hessian and gradient estimates of the approximating function to control the geometry under the flow $\Psi$. In the smooth setting, it stems from the first variation formula along $\sigma_1$ and $\sigma_2$ and the following interpolation formula along a unit speed geodesic $\alpha$, which should be thought of as going between $\sigma_1(t)$ and $\sigma_2(t)$ for some $t$ in this application.
	\begin{equation}\label{Hessian interpolation}
		\inner{\alpha'(\tau_1)}{V} - \inner{\alpha'(\tau_0)}{V} = \int_{\tau_0}^{\tau_1} \inner{\nabla_{\alpha'(\tau)}V}{\alpha'(\tau)} \, d\tau,
	\end{equation}
	where $V$ is a vector field along $\alpha$ and $\nabla V$ is its covariant derivative. We will apply \eqref{CN interpolation} to control the integral distance change between all elements of two sets under the flow $\Psi$ and so an integral version of the formula suffices. The first variation formula for ``almost every" pair of $\sigma_1$ and $\sigma_2$ we are interested in follows easily from the first order differentiation formula for Wassersein geodesics (see \cite{G13}).

	In the direction of \eqref{Hessian interpolation}, the same formula (with obvious changes) was proved along Wasserstein geodesics with bounded density in \cite[Theorem 5.13]{GT18}. While this does most of the work, a suitable interpretation is required to obtain the integral interpolation formula \label{Hess interpolation} between two sets $S_1$ and $S_2$. To see the difficulty, one might try to decompose the set of all geodesics between $S_1$ and $S_2$ by grouping together all geodesics that start at the same $x \in S_1$. In this way, one obtains a family of Wasserstein geodesics parameterized by $x \in S_1$. However, these end at a $\delta$ measure and therefore the interpolation formula of  \cite[Theorem 5.13]{GT18} does not apply. This is, of course, expected because $\inner{\nabla d_x}{V}$ is not well-defined at $x$. The correct decomposition then is to break all the geodesics between $S_1$ and $S_2$ down the middle, parameterize the half that start in $S_2$ by the elements of $S_1$ they each go toward and vice versa. We point out that the same decomposition is used in the proof of the segment inequality. Some work then needs to be done to check the boundary terms that arise in interpolating between each of the halves match correctly. These are the contents of Section \ref{section 3}.
	\item The Hessian estimate \eqref{heat flow hess estimate} and several other estimates on the heat flow approximation of the distance function need to be shown in the $\RCD(K,N)$ setting. This simply comes down to verifying the proofs of \cite{CN12} all translate to the metric measure setting with minor adjustments. These are the contents of Section \ref{section 4}.
	\item Lastly, the argument in \cite{CN12} relies on \eqref{infinitesimal hess} to obtain estimates for small balls centred along in the interior $\gamma$ under $\Psi$ in order to start an induction process. Such an inequality is not available in the $\RCD$ setting since the Hessian is a measure-theoretic object, although progress has been made in this direction, see \cite{BC13, C14, CM17b, CM18}. Even if it were well-defined, one does not have Jacobi fields or smoothness arguments  to translate such an inequality to a statement about tangent cones or small balls along $\gamma$. This is, in many ways, the main obstruction to extending the arguments of \cite{CN12}. We will not attempt to develop all this theory. The key observation is that in fact we can do without the start of induction in radius.

	Recall that the need for this start of induction argument stems from the failure of \eqref{heat flow hess estimate} to control distance between a set of small measure and another set under $\Psi$. As such, it is possible for most of $\Psi_{t}(B_{r}(\gamma(s_1)))$ to distance from $\gamma(s_1-t)$ quickly, after which we can no longer apply \eqref{heat flow hess estimate} to control the geometry of $\Psi_{t}(B_{r}(\gamma(s_1)))$. To deal with this, consider for each $x \in B_{r}(\gamma(s_1))$ a piecewise geodesic which goes from $p$ to $x$ and then $x$ to $q$.  It was shown in \cite{CN12} that for a significant amount of $x$ (those with relatively low excess) and their corresponding piecewise geodesics, the estimate \eqref{heat flow hess estimate} still holds on the scale of $r$. The same is true for $\RCD(-(N-1),N)$ spaces, see Theorem \ref{hess h integral bound}. Using this, one can make an induction argument in time instead. Suppose for some small time $t$ most of $\Psi_{t}(B_{r}(\gamma(s_1)))$ stays close to $\gamma(s_1-t)$, after which it leaves. Due to the control we had on the geometry of $\Psi_{t}(B_{r}(\gamma(s_1)))$ in that time, we can guarantee that $\Psi_{t}(B_{r}(\gamma(s_1)))$ is still very close to one of (in fact, much of) these other piecewise geodesics with a good estimate \eqref{heat flow hess estimate}. Therefore, we can use the estimate for that piecewise geodesic for a little longer. The start of induction is trivial since the integral curves of $\Psi$ are 1-Lipschitz. In this way, we arrive at an $x \in B_{r}(\gamma(s_1))$ whose trajectory under $\Psi$ well represents the behaviour of $B_{r}(\gamma(s_1))$ under $\Psi$, in the sense that most of $B_{r}(\gamma(s_1))$ stays close to $x$ on the scale of $r$ under $\Psi$ for a definite amount of time. Multiple limiting and gluing arguments then allow for the selection of a geodesic from $p$ to $q$, perhaps different from $\gamma$, which well represents the behaviours of small balls centred in its interior under $\Psi$. The original argument of \cite{CN12} gives the required Gromov-Hausdorff approximations in the interior of such a geodesic. Notice that, analogous to \cite{CN12}, we have at this point only shown the existence of a geodesic between $p$ and $q$ which satisfies the main theorem. These are the contents of Section \ref{section 5}. 
\end{itemize}

The ideas outlined above overcome the difficulties of generalizing the arguments of \cite{CN12} to the $\RCD$ setting. To finish, we will first show that $\RCD$ spaces are non-branching before proving Theorem \ref{main theorem 1}. In order to prove non-branching, first notice that any two geodesics having the property above cannot branch. To see this, let $\gamma_1$ and $\gamma_2$ be two branching geodesics starting at some $p \in X$ which can be constructed by the methods of Section \ref{section 5}. In the interior, most of an arbitrarily small ball centred around $\gamma_1$ (resp. $\gamma_2$) must stay close to $\gamma_1$ (resp. $\gamma_2$) for some definite amount of time under the flow of $\Psi$, where the closeness is H\"older dependent on time. Moreover, it is possible to control how the volumes of balls changes along each geodesic. Combining these obesrvations with the essentially non-branching property of $\RCD(K,N)$ spaces show that there cannot be any splitting because there is simply not enough room to flow disjoint small balls around $\gamma_1$ and $\gamma_2$ into a small ball around a branching point. While we do not initially claim all geodesics can be constructed with the methods of Section \ref{section 5}, our construction does give a certain amount of freedom. For any $\delta>0$, it allows us to construct a geodesic $\gamd$ with nice properties on $[\delta,1-\delta]$ which agrees with the initial geodesic $\gamma$ at $\delta$. As it turns out, combining this with the previous observation is enough to show that in fact no pair of geodesics can branch. Theorem \ref{main theorem 1} follows easily from the results of Section \ref{section 5} and non-branching. These are the contents of Subsection \ref{subsection 6.1}. In Subsection \ref{subsection 6.2}, we generalize to the $\RCD$ setting the applications of the main result for Ricci limits outlined in \cite{CN12} using verbatim arguments. 


\subsection{Acknowledgements}
I thank Vitali Kapovitch for introducing me to this problem and sharing his many insights. The many hours of his time spent in discussing ideas, answering my questions, and editing my work were indispensible to the writing of this paper. I thank Christian Ketterer for helpful discussions and reviewing an early draft of this paper. I thank Nicola Gigli for answering several technical questions in the early stages of the paper and, along with Andrea Mondino, for many helpful comments. 

\section{Preliminaries} \label{section 2}

\subsection{Curvature-dimension condition preliminaries}

A \textit{metric measure space (m.m.s.)} is a triple $(X,d,m)$ where $(X,d)$ is a complete, separable metric space and $m$ is a nonnegative, locally finite Borel measure. As a matter of convention, $m$-measurable in this paper means measurable with respect to the completion of $(X, \mathscr{B}(X), m)$. We take the same convention for all other Borel measures as well. 

Given a complete and separable metric space $(X,d)$, we denote by $\mathcal{P}(X)$ the set of Borel probability measures and by $\mathcal{P}_2(X)$ the set of Borel probability measure with finite second moment, that is, the set of $\mu \in \mathcal{P}(X)$ where $\int_{X} d(x,x_0) d\mu(x) < \infty$ for some $x_0 \in X$. Given $\mu_1, \mu_2 \in \mathcal{P}_{2}(X)$, the \textit{$L^2$-Wasserstein distance} $W_2$ between them is defined as
\begin{equation*}
W_2^2 (\mu_1, \mu_2) := \inf\limits_{\gamma} \int_{X \times X} d^2(x,y) d\gamma(x,y),
\end{equation*}
where the infimum is taken over all $\gamma \in \mathcal{P}(X \times X)$ with $(\pi_1)_{*}(\gamma) = \mu_1$ and $(\pi_2)_{*}(\gamma) = \mu_2$. Such measures $\gamma$ are called \textit{admissible plans} for the pair $(\mu_1, \mu_2)$. $(\mathcal{P}_{2}(X), W_2)$ is called the $L^2$-Wasserstein space of $(X,d)$ and has been well-studied in the theory of optimal transportation. A $W_2$-geodesic between $\mu_0, \mu_1 \in \mathcal{P}_2(X)$ is any path $(\mu_t)_{t \in [0,1]}$ in $\mathcal{P}_2(X)$ satisfying $W_2(\mu_s,\mu_t) = |s-t|W_2(\mu_0, \mu_1)$ for any $s, t \in [0,1]$. If $(X,d)$ is a geodesic space then $(\mathcal{P}_2(X), W_2)$ is as well. A $c$-concave solution $\varphi$ to the coresponding dual problem of maximizing $\int \varphi d\mu_0+ \int \varphi^{c} d\mu_1$ is called a \textit{Kantorovich potential}. We refer to \cite{AG13} and \cite{V09} for definitions and details.   

The various notions of the classical curvature-dimension condition were first proposed independently in \cite{LV09} and \cite{S06a, S06b} and are defined as certain convexity conditions on the $L^2$-Wasserstein space of a metric measure space. We follow closely the formulations of \cite{BS10}.

Given a m.m.s. $(X,d,m)$, for any $\mu \in \mathcal{P}(X)$, the \textit{Shannon-Boltzmann entropy} is defined as 
\begin{equation*}
	\Ent_m(\cdot):\mathcal{P}(X) \to (-\infty,\infty], \; \Ent_{m}(\mu) := \int\log{\rho} \, d\mu \,\text{ , if $\mu = \rho m$ and $(\rho \log \rho)_-$ is $m$-integrable}
\end{equation*}
and $\infty$ otherwise.

\begin{de}\label{CDinf def} ($\CD(K,\infty)$ condition)
Let $K \in \mathbb{R}$. A m.m.s. $(X,d,m)$ is a $\CD(K, \infty)$ space iff for any two absolutely continuous measures $\mu_0, \mu_1 \in \mathcal{P}(X)$ with bounded support, there exists a $W_2$-geodsic $\braces{\mu_t}_{t \in [0,1]}$ such that for any $t \in [0,1]$, 
\begin{equation*}
\Ent_m(\mu_t) \leq (1-t)\Ent_m(\mu_0) + t\Ent_m(\mu_1) - \frac{K}{2}t(1-t)W^2_2(\mu_0, \mu_1).
\end{equation*}
\end{de}

The \textit{$N$-Renyi entropy} is defined as
\begin{equation*}
	S_N(\cdot|m):\mathcal{P}(X) \to (-\infty, 0], \; S_N(\mu|m):=-\int \rho^{1-\frac{1}{N}} \, dm,
\end{equation*}
if $\rho^{1-\frac{1}{N}} \in L^1(m)$, where $\mu = \rho m$, and $0$ otherwise.

Let $K \in \mathbb{R}$ and $N \in [1, \infty)$, the distortion coefficents $\sigma_{K,N}^{(t)}$ and $\tau_{K,N}^{(t)}$ are defined as follows:
\begin{align*}
 (t, \theta) \in [0,1] \times \mathbb{R}^{+} \to \sigma^{(t)}_{K,N}(\theta):=\begin{cases}
	\infty &\text{ if } K\theta^2 \geq N\pi^{2}\\
	\frac{\sin(t\theta\sqrt{K\slash N })}{\sin(\theta\sqrt{K \slash N})} &\text{ if } 0 < K\theta^2 < N\pi^2,\\
	t &\text{ if } K\theta^2 = 0,\\
	\frac{\sinh(t\theta\sqrt{K\slash N })}{\sinh(\theta\sqrt{K \slash N})} &\text{ if } 0 < K\theta^2 < 0,\\
	\end{cases}
\end{align*}
and
\begin{equation*}
\tau_{K,N}^{(t)}(\theta) := t^{\frac{1}{N}}\sigma_{K,N}^{(t)}(\theta)^{1-\frac{1}{N}}.
\end{equation*}

The standard finite dimensional \textit{curvature-dimension} condition was introduced in \cite{S06b, LV09}. 

\begin{de}\label{CD def} ($\CD(K,N)$ condition)
	Let $K \in \mathbb{R}$ and $N \in [1,\infty)$. We say that a m.m.s. $(X,d,m)$ is a $\CD(K,N)$ space if for any two absolutely continuous measures $\mu_0 = \rho_0m, \mu_1 = \rho_1m \in \mathcal{P}(X)$ with bounded support there exists a $W_2$-geodesic $\braces{\mu_t}_{t \in [0,1]}$ and an associated optimal coupling $\pi$ between $\mu_0$ and $\mu_1$ such that for any $t \in [0,1]$ and $N' \geq N$, 
\begin{equation*}
	S_N(\mu_t | m) \leq -\int \bigg( \tau^{(1-t)}_{K,N'}(d(x,y))\rho_0(x)^{-\frac{1}{N'}}+\tau^{(t)}_{K,N'}(d(x,y))\rho_1(y)^{-\frac{1}{N'}} \bigg)\, d\pi(x,y).
\end{equation*}
\end{de}

The \textit{reduced curvature-dimension} condition $\CD^{*}(K,N)$ was introduced in \cite{BS10} for its seemingly better tensorization and globalization properties. It is defined by replacing $\tau$ with $\sigma$ in Definition \ref{CD def}. The $\CD(K,N)$ and $\CD^{*}(K,N)$ conditions generalize to the metric measure setting the notion of Ricci curvature bounded below by $K$ and dimension bounded above by $N$. Examples include (possibily weighted) Riemannian manifolds \cite{S06b}, Finsler manifolds \cite{O09} and Alexandrov spaces \cite{P09}.

\begin{rem}
$\CD(K,N)$ implies $\CD(K',N')$ and $\CD^{*}(K',N')$ for all $K' \leq K$ and $N' \geq N$ as well as $\CD(K', \infty)$. A host of results that we cite were shown in the $\RCD^{*}(K,N)$ and $\RCD(K,\infty)$ setting, see \ref{RCD def} for definitions, and therefore apply in the $\RCD(K,N)$ setting. Going in the other direction, it was shown in \cite{CM16} that $\RCD^{*}(K,N)$ is equivalent to $\RCD(K,N)$ when $m(X) < \infty$. It is believed that this argument can be taken to the noncompact case. We mention that the proofs of this paper carry forward without modication to the $\RCD^*(K,N)$ setting. However, since several papers we cite use the stronger $\RCD$ assumption (though it can be checked this is not needed for the particular results we cite from them), we will do so as well to ease the burden of exposition. 
\end{rem}

It is known that if $(X,d,m)$ is $\CD(K,N)$ then $\supp(m)$ is a geodesic space which also satisfies the $\CD(K,N)$ condition. Due to this, we will always assume $X = \supp(m)$. One can check that for any $\lambda, c > 0$, if $(X,d,m)$ is $\CD(K,N)$, then $(X, \lambda d, cm)$ is $\CD(\frac{K}{\lambda^2}, N)$. $\CD(K,N)$ spaces, like their smooth counterparts, satisfy the standard Bishop-Gromov volume comparison.

\begin{thm}\label{BG volume comparison}(Bishop-Gromov volume comparison \cite[Theorem 2.3]{S06b})
Let $(X,d,m)$ be a $\CD(K,$ $N)$ space for some $K \in \mathbb{R}$ and $N \in (1, \infty)$. Then for all $x_{0}\in X$ and all $0< r< R\leq \pi \sqrt{N-1/(K \vee 0)}$ it holds:
\begin{equation}
\frac{m(B_r(x_0))}{m(B_{R}(x_0))} \geq { V_{K,N}(r, R):=}
\begin{cases}
\frac{\int_{0}^{r} \left(\sin(t \sqrt{K/(N-1)})\right)^{N-1} dt}{ \int_{0}^{R} \left(\sin(t \sqrt{K/(N-1)}) \right)^{N-1} dt} & \quad \text{if } K>0, \medskip \\
\left(\frac r  R  \right)^{N}& \quad \text{if } K=0, \medskip \\
\frac{\int_{0}^{r} \left(\sinh(t \sqrt{K/(N-1)})\right)^{N-1} dt}{ \int_{0}^{R} \left(\sinh(t \sqrt{K/(N-1)}) \right)^{N-1} dt} & \quad \text{if } K<0. 
\end{cases}
\end{equation}
\end{thm}
In contexts where $K$ and $N$ are clear, we will simply write $V(r,R)$ for \label{V}$V_{K,N}(r,R)$. 

For $N < \infty$, $\CD(K,N)$ spaces are locally doubling, by Theorem \ref{BG volume comparison}, and are therefore proper. They satisfy a 1-1 Poincaré inequality by \cite[Theorem 1.1]{R12}.

\subsection{First order calculus on metric measure spaces} We follow the framework for calculus on metric measure space developed by Ambrosio, Gigli and Savar\'e in \cite{AGS13, AGS14a, AGS14b, G15, G18}. Let $(X,d,m)$ be a metric measure space. Let $\lip(X)$, $\lip_{loc}(X)$, $\lip_{b}(X)$ be its class of Lipschitz, locally Lipschitz, and bounded Lipschitz functions respectively. Given $f \in \lip_{loc}(X)$, the \textit{local Lipschitz constant} (or \textit{local slope}) $\lip(f): X \to \mathbb{R}$ is defined by
\begin{equation}\label{Lipschitz constant}
	\lip(f)(x) := \limsup\limits_{y \to x} \frac{|f(y) - f(x)|}{d(y,x)}.
\end{equation}
By convention, $\lip(f)(x) := 0$ at any isolated point $x$. Given $f \in L^2(m)$, a function $g \in L^2(m)$ is called a \textit{relaxed gradient} if there exists a sequence $f_n \in \lip(X)$ and $\tilde{g} \in L^2(m)$ so that
\begin{enumerate}
	\item $f_n \to f$ in $L^2(m)$ and $\lip(f_n)$ converges weakly to $\tilde{g}$ in $L^2(m)$
	\item $g \geq \tilde{g}$ $m$-a.e..
\end{enumerate}
A \textit{minimal relaxed gradient} is a relaxed gradient that is minimal in $L^2$-norm in the family of relaxed gradients of $f$. If this family is non-empty, one can check that such a function exists and is unique $m$-a.e.. The minimal relaxed gradient is denoted by $|Df|$. The \textit{domain of the Cheeger energy} $D(\Ch ) \subseteq L^2(m)$ is the subset of $L^2$ functions with a minimal relaxed gradient. For $f \in L^2(m)$, the \text{Cheeger energy} is defined as 
   \[ \Ch(f) = \begin{cases} 
       \frac{1}{2}  \int |Df|^2 \, dm &  \text {if } f \in D(\Ch),\\
	\infty & \text{otherwise.}\\
       \end{cases}
    \]
$\Ch$ is a convex and lower semicontinuous functional on $L^2(m)$. The Cheeger energy first introduced in \cite{C99} was defined using a slightly different relaxation procedure. It is also possible to define a similar functional using the idea of \textit{minimal weak upper gradients}, see \cite[Section 5.1]{AGS14a}. It is shown in \cite[Section 6]{AGS14a} that under mild assumptions on the metric measure space, satisfied, for example, by the various curvature-dimension conditions, all these notions are equivalent.
\begin{rem}\label{lip=grad}
Let $(X,d,m)$ be an $\CD(K,N)$ space with $N < \infty$. For any Lipschitz function $f$ on $X, \lip (f) = |D f|$ $m$-a.e. This follows from \cite{C99}, where it is shown that a metric measure space satisfying a Poincaré inequality and a doubling inequality has $\lip (f) = |D f|$ $m$-a.e.. 
\end{rem}
$W^{1,2}(X) := D(\Ch)$ is a Banach space endowed with the norm $\norm{f}_{W^{1,2}(X)}^2 := \norm{f}_{L^2(m)}^2+\norm{|Df|}_{L^2(m)}^2$. We define $W^{1,2}_{Loc}(X)$ as the space of all function $f \in L^2(X, m)$ so that $gf \in W^{1,2}(X)$ for every compactly supported, Lipschitz $g$. By the strong locality property of the minimal relaxed gradient (ie. $|Dg| = |Dh| \ m$-a.e. in $\braces{g=h}$ for any $g,h \in W^{1,2}(X)$), any $f \in W^{1,2}_{Loc}(X)$ has an associated differential $|Df| \in L^2_{loc}(m)$. 

$(X,d,m)$ is said to be \textit{infinitesimally Hilbertian} if $W^{1,2}(X)$ is a Hilbert space. In this case, for $f, g \in W^{1,2}(X)$, one may define $\inner{Df}{Dg}$ using polarization: $\inner{Df}{Dg} := \frac{1}{2}(|D(f+g)|^2 - |Df|^2 - |Dg|^2)\in L^1(m)$. 

From here one can define the \textit{Laplacian}. $f \in W^{1,2}(X)$ is said to be in the \textit{domain of the Laplacian} ($f \in D(\Delta)$) if there exists $\Delta f \in L^2(m)$ so that
\begin{equation*}
	\int g \Delta f \, dm + \int \inner{Dg}{Df} \, dm = 0 \; \; \text{ for any $g \in W^{1,2}(X)$}.
\end{equation*}
Given a subspace $V \in L^2(m)$, we denote $D_{V}(\Delta) :=  \braces{f \in D(\Delta): \Delta(f) \in V}$. More generally, one may define the \textit{measure valued Laplacian}. 
\begin{de}\label{measure lap}(Measure valued Laplacian \cite[Definition 3.1.2]{G18}) The space $D(\boldsymbol{\Delta}) \subset W^{1,2}(X)$ is the space of $f \in W^{1,2}(X)$ such that there is a signed Radon measure $\mu$ satisfying
\begin{equation*}
	\int gd\mu = -\int\inner{Dg}{Df}dm \hspace{0.5cm} \forall g:X \to \mathbb{R}\text{ Lipschitz with bounded support.}
\end{equation*}
In this case the measure $\mu$ is unique and is denoted by $\boldsymbol{\Delta} f$.
\end{de}

\subsection{Tangent, cotangent, and tensor modules}

A technical framework for describing first order calculus on metric measure spaces and second order calculus on $\RCD(K,N)$ spaces was developed  by Gigli in \cite{G18}. While aspects of second order calculus can be effectively developed without this framework (see for example \cite{S14, AGS15}), \cite{G18} crucially gives constructions which generalize the notion of tensor fields. In the next few subsections, we will quickly introduce, sometimes informally, the necessary definitions given in \cite{G18} and refer to the original article for details and insights. 

Let $(X,d,m)$ be a metric measure space. The various collections of tensor fields of interest will be objects in the category of \textit{$L^p(m)$-normed $L^{\infty}(m)$-modules}. 
\begin{de}\label{LpLinfpremod}($L^p$-normed $L^\infty$-premodules \cite[Definition 1.2.1 1.2.10]{G18})
Let $p \in [0, \infty]$. Let $(\mathcal{M}, \norm{\cdot}_{\mathcal{M}}$) be a Banach space endowed with a bilinear map $L^{\infty}(m) \times \mathcal{M} \ni (f, v) \mapsto f \cdot v \in \mathcal{M}$ and a function $| \cdot | : \mathcal{M} \to L^p_{+}(m)$. We say $(\mathcal{M}, \norm{\cdot}_{\mathcal{M}}, \cdot, |\cdot|)$ is an $L^p$-normed $L^\infty$-premodule iff the following holds
\begin{enumerate}
	\item $(fg)\cdot v = f \cdot(g \cdot v)$ for all $f, g \in L^{\infty}(m)$ and $v \in \mathcal{M}$
	\item $\bold{1} \cdot v = v$ for all $v \in \mathcal{M}$ where $\bold{1}$ is the constant function equal to $1$
	\item $\norm{|v|}_{L^p(m)} = \norm{v}_{\mathcal{M}}$  for all $v \in \mathcal{M}$
	\item $|f \cdot v| = |f||v| m$-a.e. for all $f \in L^{\infty}(m)$ and $v \in \mathcal{M}$. 
\end{enumerate} 
\end{de}

We will often simply write $fv$ for $f \cdot v$ and call $|\cdot|$ the pointwise norm. If an \textit{$L^p(m)$-normed $L^\infty (m)$-premodule} satisfies additional locality and gluing properties (\cite[Definition 1.2.1]{G18}), we say it is an \textit{$L^p(m)$-normed $L^\infty(m)$-module}. One may localize such an object to some $A \in \mathscr{B}(X)$ by defining $\mathcal{M}\rvert_A := \braces{v \in \mathcal{M}\, : \, |v|= 0 \: m\text{-a.e. on}\, A^c}$, which is again canonically an $L^p(m)$-normed $L^\infty(m)$-module. 

An $L^2$-normed $L^\infty$-module which is a Hilbert space under $\norm{\cdot}_{\mathcal{M}}$ is called a \textit{Hilbert module}. In this case one can define a pointwise inner product by polarizing the pointwise norm $|\cdot|$. The prototypical example of a Hilbert module one has in mind is the collection of $L^2$ vector fields on a Riemannian manifold, where $|\cdot|$ is the Riemannian pointwise norm.

Given $L^\infty$-modules $\mathcal{M}$ and $\mathcal{N}$, we say a map $T: \mathcal{M} \to \mathcal{N}$ is a \textit{module morphism} if it is a bounded linear map between $\mathcal{M}$ and $\mathcal{N}$ as Banach spaces satisfying in addition $T(fv) = fT(v)$ for all $f \in L^{\infty}(m)$ and $v \in \mathcal{M}$. The \textit{dual module $\mathcal{M}^{*}$} is the space of all module morphisms between $\mathcal{M}$ and $L^{1}(m)$ and is an $L^{p_*}(m)$-normed $L^{\infty}(m)$-module, where $\frac{1}{p}+\frac{1}{p_{*}}=1$. A Hilbert module $\mathcal{H}$ is canonically isomoprhic to its dual.

Given two Hilbert modules $\mathcal{H}_1$ and $\mathcal{H}_2$, one can construct the Hilbert modules: the tensor product $\mathcal{H}_1 \otimes \mathcal{H}_2$ (\cite[Definition 1.5.1]{G18}) and the exterior product $\mathcal{H}_1 \Lambda \mathcal{H}_2$ (\cite[Definition 1.5.4]{G18}). 

\begin{de}\label{modulegenerator}(\cite[Definition 2.19]{G18})
Let $\mathcal{M}$ be an $L^2(m)$-normed and $V \subseteq \mathcal{M}$. $\Span(V)$ is defined as the collection of $v \in \mathcal{M}$ for which there is a Borel decomposition $(X_n)_{n \in \mathbb{N}}$ of $X$ and, for each $n \in \mathbb{N}$, collections $v_{1,n}, ... , v_{k_n, n} \in V$ and $f_{1,n}, ... ,f_{k_n, n} \in L^{\infty}(m)$ so that
\begin{equation*}
	\bold{1}_{X_n}v = \sum\limits_{i=1}^{k_n} f_{i,n}v_{i,n} \text{ , for each $n \in \mathbb{N}$,}
\end{equation*}
where $\bold{1}_{X_n}$ is the characteristic function of $X_n$. We say that $V$ genereates $\mathcal{M}$ iff $\overline{\Span{V}} = X$.
\end{de}

From this point we will assume $(X,d,m)$ is a infinitesimally Hilbertian metric measure space. We now define the tangent and cotangent modules of $(X,d,m)$.

\begin{thm}\label{tan cotan modules}(\cite[Proposition 2.2.5]{G18})
There exists a unique, up to isomorphism, Hilbert module $\mathcal{H}$ endowed with a linear map $d:W^{1,2}(X) \to \mathcal{H}$ satisfying:
\begin{enumerate}
	\item $|df| = |Df|$ $m$-a.e. for all $f \in W^{1,2}(X)$
	\item $d(W^{1,2}(X))$ generates $\mathcal{H}$.
\end{enumerate}
Such an $\mathcal{H}$ is called the \textit{cotagent module} of $(X,d,m)$ and denoted by $L^{2}(T^{*}X)$.\\
The dual of $L^{2}(T^{*}X)$ is called the \textit{tangent module} of $(X,d,m)$ and denoted by $L^{2}(TX)$. Elements of $L^{2}(TX)$ are called \textit{vector fields}.\\
We denote by $\nabla f$ the dual of $df$ (i.e. the unique element $\nabla f \in L^{2}(TX)$ so that $v(\nabla f) = \inner{v}{df}$ for all $v \in L^{2}(T^{*}X)$).
\end{thm}

Notice $\inner{df}{dg} = df(\nabla g) = \inner{\nabla f}{\nabla g} = \inner{Df}{Dg} \ m$-a.e. and we will use these interchangeably in the rest of the paper depending on the convention of the theorems we are quoting. For a discussion of the philosophical differences of the objects involved, see \cite[Section 2.2]{AGS14a}.

\begin{de}\label{div def}{\cite[Definition 2.3.11]{G18}}
$D(\diver) \subseteq L^2(TX)$ is the space of all vector fields $v \in L^2(TX)$ for which there exists $f \in L^2(m)$  so that for any $g \in W^{1,2}(X)$ the equality
\begin{equation*}
	\int fg \,dm = -\int dg(v)\, dm
\end{equation*}
holds. In this case $f$ is called the \textit{divergence of v} and denoted by $\diver(v)$. In particular, if $f \in D(\Delta)$, then $\nabla f \in D(\diver)$ and $\diver(\nabla f) = \Delta f$. 
\end{de}

\begin{de}\label{tensor modules}
$L^{2}((T^*)^{\otimes2}(X))$ denotes the tensor product of $L^{2}(T^{*}X)$ with itself (see \cite[Definition 1.5.1]{G18})). Similarly, $L^{2}(T^{\otimes2}(X))$ denotes the tensor product of $L^2(TX)$ with itself. We will use $|\cdot|_{\HS}$ and $\cdot : \cdot$ to denote the pointwise norm (Hilbert-Schmidt norm) and the pointwise inner product of Hilbert modules which arise from tensors. 
\end{de}
 
$L^{2}((T^*)^{\otimes2}(X))$ and $L^{2}(T^{\otimes2}(X))$ are Hilbert module duals of each other. We mention that for any element $A \in L^{2}((T^*)^{\otimes2}(X))$, we will often write $A(V,W) = A(V\otimes W)$. We will sometimes write this even when $V$ and $W$ are so that $V \otimes W$ is not in $L^{2}(T^{\otimes2}(X))$. In all these cases, $V \otimes W$ when multipled by the characteristic function of a compact set will be in $L^{2}(T^{\otimes2}(X))$ and so $A(V,W)$ is well-defined as a measurable function by locality and satisfies
\begin{equation}\label{tensor operator norm}
	A(V,W) \leq |A|_{\HS}|V||W|.
\end{equation}
We will usually have additional assumptions on $|V|$ and $|W|$ so that $A(V,W) \in L_{loc}^1(m)$.

\subsection{RCD(K,N) and Bakry-\'Emery conditions}
We now introduce the notion of $\RCD$ spaces, which are the main objects of interest for this paper. These were proposed and carefully analyzed in a series of papers including \cite{G15, EKS15, AMS19} in the finite dimensional case and \cite{AGS14b, AGMR15} in the infinite dimensional case. 

\begin{de}\label{RCD def}(\cite{AGS14b, G15})
Let $(X,d,m)$ be a metric measure space, $K \in \mathbb{R}$, and $N \in [1, \infty)$. We say $(X,d,m)$ satisfies the \textit{Riemmanian curvature-dimension condition} $\RCD(K,N)$ iff $(X,d,m)$ satisfies the $\CD(K,N)$ condition and is infinitesimally Hilbertian. Similarly one defines the \text{Riemannian curvature-dimension} conditions $\RCD^{*}(K,N)$ and $\RCD(K,\infty)$ using $\CD^{*}(K,N)$ and $\CD(K,\infty)$ respectively. 
\end{de}

The $\RCD$ condition is stable under measured Gromov-Hausdorff convergence and tensorization. Examples of $\RCD$ spaces include Ricci limits and Alexandrov spaces but non-Riemannian Finsler geometries are ruled out. We now state some equivalent formulations of the $\RCD(K,N)$ property. We will in general assume $(X,d,m)$ is infinitesimally Hilbertian in this subsection.

As in \cite{AGS15}, we define the \textit{Carr\'e du champ operator} for $f \in D_{W^{1,2}(X)}(\Delta)$ and $\varphi \in D_{L^{\infty}(m) \cap L^2(m)}(\Delta) \cap L^{\infty}(m)$ by
\begin{equation*}
	\Gamma_2(f;\varphi) := \int \frac{1}{2}|\nabla f|^2\Delta \varphi dm - \int \inner{\nabla f}{\nabla \Delta f}\varphi dm.
\end{equation*}

This enables us to state the non-smooth Bakry-\'Emery condition $\BE(K,N)$.
\begin{de}\label{BE}(Bakry-\'Emery condition \cite{AGS15, EKS15})
Let $K \in \mathbb{R}$ and $N \in [1, \infty]$. We say $(X,d,m)$ satisfies the $\BE(K,N)$ condition iff 
\begin{equation*}
	\Gamma_2(f;\varphi) \geq \frac{1}{N} \int (\Delta f)^2 \varphi dm + K \int |\nabla f|^2 \varphi dm.
\end{equation*}
\end{de}

$\BE(K,N)$ is closely related to $\CD(K,N)$. We say $(X,d,m)$ satisfies the \textit{Sobolev-to-Lipschitz property}, \cite[Definition 3.15]{GH18}, if any function $f \in W^{1,2}(X)$ with $|\nabla f| \in L^{\infty}(m)$ has a Lipschitz representatitive $\tilde{f} = f \, m$-a.e. with Lipschitz constant equal to $\esssup(|\nabla f|)$.

\begin{thm}\label{RCD BE equivalence}(\cite{AGS15,EKS15,AMS19})
Let $(X,d,m)$ be a metric measure space satisfying an exponential growth condition (see \cite[secion 3]{AGS15}), $K \in \mathbb{R}$, and $N \in (1,\infty)$. $(X,d,m)$ is $\RCD(K,N)$ iff $(X,d,m)$ is infinitesimally Hilbertian, satisfies the Sobolev-to-Lipschitz property and the $\BE(K,N)$ condition. 
\end{thm}

\subsection{Heat flow and Bakry-Ledoux estimates}

By applying the theory of the gradient flow of convex functionals on Hilbert spaces to $\Ch$ as in \cite{AGS14a}, one obtains for each $f \in L^2(m)$ a unique continuous curve $(H_t(f))_{t \in [0,\infty)}$ in $L^2(m)$ which is locally absolutely continuous in $(0,\infty)$ with $H_0(f) = f$ so that 
\begin{equation}\label{heat flow def}
	\frac{d}{dt} \, H_t(f) = \Delta' H_t(f) \, \text{ for a.e. $t \in (0,\infty)$,}
\end{equation}
where $\Delta' g$ is defined as the minimizer in $L^2$ energy in $\partial^-(\Ch)$ at $g$ provided it is non-empty (see \cite[Section 4.2]{AGS14a}). 

If $(X,d,m)$ is infinitesimally Hilbertian, then for any $t> 0$, $H_t(f) \in D(\Delta)$ and one has the a priori estimates
\begin{equation*}
\norm{H_t(f)}_{L^2} \leq \norm{f}_{L^2}, \; \; \norm{|DH_t(f)|}_{L^2}^2 \leq \frac{\norm{f}_{L^2}^2}{2t^2}, \; \; \norm{\Delta H_t(f)}_{L^2} \leq \frac{\norm{f}_{L^2}}{t}.
\end{equation*}
 $H_t(f)$ is linear and satisfies $\Delta(H_t(f)) = H_t(\Delta(f))$ for any $t > 0$. In particular, \eqref{heat flow def} is true for all $t \in (0, \infty)$ and 
\begin{equation*}
	H_t(f) = f + \int_{0}^{t} \Delta(H_s(f)) \, ds.
\end{equation*}

If $(X,d,m)$ is $\RCD(K,\infty)$, $H_t$ can be identified with $\mathscr{H}_t$, the gradient flow of $\Ent_m$ on $\mathcal{P}_2(X)$. Due to contraction properties coming from the $\RCD$ condition, $\mathscr{H}_t$ can be extended from $D(\Ent_m)$ to all of $\mathcal{P}_2(X)$. 

\begin{de}\label{heat kernel def}
For $t > 0$ and $x \in X$, $\mathscr{H}_t(\delta_x)$ is absolutely continuous with respect to $m$. Then $\mathscr{H}_t(\delta_x) = H_t(x,\cdot)m$, where $H_t(\cdot,\cdot)$ is the \textit{heat kernel}. 
\end{de}

$H_t(x,y)$ is symmetric and continuous in both variables. For each $f \in L^2(m)$, one has the representation formula, \cite[Theorem 6.1]{AGS14b},
\begin{equation}\label{heat kernl rep}
	H_t(f)(x) = \int f(y)H_t(x,y) \, dm.
\end{equation}

The $\RCD(K,N)$ condition implies the Bakry-Ledoux estimate, which is a finite dimensional analogue of the Bakry-\'Emery contraction estimate (\cite[Theorem 6.2]{AGS14b}). Moreover, it was shown in \cite{EKS15} that one has equivalence  in Theorem \ref{RCD BE equivalence} with the $\BE(K,N)$ condition replaced by the $(K,N)$ Bakry-Ledoux estimate.

\begin{thm}\label{Bakry-Ledoux} (Dimensional Bakry-Ledoux $L^2$ gradient-Laplacian estimate \cite[Theorem 4.3]{EKS15})
Let $(X,d,m)$ be an $\RCD(K,N)$ space for some $K \in \mathbb{R}$ and $N \in [1, \infty)$. For any $f \in W^{1,2}(X)$ and $t > 0$, 
\begin{equation*}
|\nabla(H_t(f))|^2+\frac{4Kt^2}{N(e^{2Kt}-1)}|\Delta H_t(f)|^2 \leq e^{-2Kt}H_t(|\nabla(f)|^2) \hspace{0.5cm} m\text{-a.e.}.
\end{equation*}
\end{thm}

\begin{rem}\label{pw Bakry-Ledoux}
If $|\nabla f| \in L^{\infty}$, one can take continuous representatitves of $\Delta H_t(f)$ and $H_t(|\nabla(f)|^2)$ and identify $|\nabla(H_t(f))|$ canonically with the local Lipschitz constant of $H_t(f)$  to obtain a pointwise Bakry-Ledoux bound, see \cite[Proposition 4.4]{EKS15}).
\end{rem}

This implies the Sobolev-to-Lipschitz property by \cite[Theorem 6.2]{AGS14b}. $H_t$ also has a $L^{\infty}$-to-Lipschitz property by \cite[Theorem 6.8]{AGS14b}, where it was shown for $t > 0$ and $f \in L^2(m)$ that
\begin{equation}\label{Linf-to-Lip}
2I_{2K}(t)|\nabla H_t(f)|^2 \leq H_t(f^2) \, \text{ , $m$-a.e.},
\end{equation}
where $I_{2K}(t) := \int_0^t e^{2Ks}\, ds$. 

\subsection{Second order calculus and improved Bochner inequality}

The class of \textit{test functions} was introduced in \cite{S14} as 
\begin{equation*}
	\TestF(X):=\braces{f \in D(\Delta) \cap L^{\infty}\, :\, |Df| \in L^{\infty} \, \text{ and } \, \Delta f \in W^{1,2}(X)}.
\end{equation*}
It is known that $\TestF(X)$ is an algebra and, on $\RCD(K,\infty)$ spaces, it was shown by the results of \cite{AGS14b} mentioned in the previous subsection that the heat flow approximations of an $L^{\infty} \cap L^2$ function are test functions and so $\TestF(X)$ is dense in $W^{1,2}(X)$.

In \cite{S14, H18}, it was shown that under the $\BE(K,N)$ condition, $|\nabla f|^2 \in D(\boldsymbol{\Delta})$ for any $f \in \TestF(X)$ and so one may define $\boldsymbol{\Gamma}_2(f) := \frac{1}{2}\boldsymbol{\Delta}|\nabla f|^2-\inner{\nabla f}{\nabla \Delta f}$. One can then define a Hessian for $f \in \TestF(X)$ and show the improved Bochner inequality (see \ref{improved Bochner inequality}). In section 3 of \cite{G18}, the same calculations were carried out in the framework proposed therein. We outline the main definitions and results from there. In this subsection, we assume $(X,d,m)$ is an $\RCD(K,N)$ space. 

\begin{de}\label{W22 def}(\cite[Definition 3.3.1]{G18})
$W^{2,2}(X) \subseteq W^{1,2}(X)$ is the space of all functions $f \in W^{1,2}(X)$ for which there exists $A \in L^2((T^{*})^{\otimes 2}(X))$ so that for any $g_1, g_2, h \in \TestF(X)$ the equality
\begin{align*}
2 \int h&A(\nabla g_1, \nabla g_2) \, dm \\
	&= \int -\inner{\nabla f}{\nabla g_1}\diver(h \nabla g_2) - \inner{\nabla f}{\nabla g_2}\diver(h \nabla g_1)-h\inner{\nabla f}{\nabla \inner{\nabla g_1}{\nabla g_2}} \, dm
\end{align*}
holds. In this case $A$ is called the \textit{Hessian of $f$} and denoted by $\Hess f$. $W^{2,2}(X)$ is a Hilbert space under the norm
\begin{equation*}
	\norm{f}_{W^{2,2}(X)}^2 := \norm{f}_{L^2(m)}^2 + \norm{df}_{L^2(T^{*}X)}^2 + \norm{\Hess f}_{L^2((T^{*})^{\otimes 2}(X))}.
\end{equation*} 
\end{de}

It turns out that the test functions are contained in $W^{2,2}(X)$ and one has the improved Bochner inequality as in \cite{S14, H18}.
\begin{thm}\label{improved Bochner inequality}(Improved Bochner inequality \cite[Theorem 3.3.8]{G18})
Let $(X,d,m)$ be an $\RCD(K,N)$ space for $K \in \mathbb{R}$ and $N \in [1, \infty)$ and $f \in \TestF(X)$. Then $f \in W^{2,2}(X)$ and
\begin{equation*}
	\boldsymbol{\Gamma}_2(f)\geq \big[K|\nabla f|^2 + |\Hess (f)|_{\HS}^2 \big]m.
\end{equation*}
\end{thm}

$H^{2,2}(X)$ is then defined as the closure of $\TestF(X)$ in $W^{2,2}(X)$. An approximation argument gives the following.
\begin{cor}\label{improved Bochner inequality cor}(\cite[Corollary 3.3.9]{G18})
$D(\Delta) \subseteq W^{2,2}(X)$ and for $f \in D(\Delta)$,
\begin{equation*}
	\int|\Hess f|_{\HS}^2 \, dm \leq \int \bigg[(\Delta f)^2 -K|\nabla f|^2\bigg] \, dm.
\end{equation*}
\end{cor}

Finally, we introduce the analogue of vector fields which have a first order (covariant) derivative.
\begin{de}\label{W12(TX) def}(\cite[Definition 3.4.1]{G18})
$W^{1,2}_{C}(TX) \subseteq L^2(TX)$ is the space of all $v \in L^2(TX)$ for which there exists $T \in L^2(T^{\otimes 2}(X))$ so that for any $g_1. g_2, h \in \TestF(X)$ the equality 
\begin{equation*}
	\int hT\,:\,(\nabla g_1 \otimes \nabla g_2) \, = \int -\inner{v}{\nabla g_2}\diver{(h \nabla g_1)}-h\Hess(g_2)(v,\nabla g_2)\,dm 
\end{equation*}
holds. In this case $T$ is called the \textit{covariant derivative} of $v$ and denoted by $\nabla v$. $W^{1,2}_C(TX)$ is a Hilbert space under the norm
\begin{equation*}
	\norm{v}_{W^{1,2}_C(TX)}^2 := \norm{v}_{L^2(TX)}^2+\norm{\nabla v}_{L^2(T^{\otimes 2}(X))}^2.
\end{equation*} 
\end{de}

The class of \textit{test vector fields} is defined as 
\begin{equation*}
	\TestV(X) := \braces{\sum\limits_{i=1}^n g_i\nabla f_i \, : \, n \in \mathbb{N}, \, f_i, g_i \in \TestF(X)}.
\end{equation*}
By \cite[Theorem 3.4.2]{G18}, $\TestV(X) \subseteq W^{1,2}_C(TX)$ and so $H^{1,2}_C(TX)$ is defined to be the closure of $\TestV(X)$ in $W^{1,2}_C(TX)$. For any $f \in W^{2,2}(X)$,  $\Hess f$ and $\nabla(\nabla f)$ are dual under the duality of $L^{2}((T^*)^{\otimes2}(X))$ and $L^{2}(T^{\otimes2}(X))$.


\subsection{Non-branching and essentially non-branching spaces}

Given a geodesic metric space $(X,d)$, we define the space of constant speed geodesics
\begin{equation*}
	\Geo(X) := \braces{\gamma \in C([0,1],X)\; : \; d(\gamma(s), \gamma(t)) = |s-t|d(\gamma(0), \gamma(1))\; \forall s, t\in[0,1]}.
\end{equation*}

For each $t \in [0,1]$, $e_t: \Geo(X) \to X$ defined by $e_t(\gamma) := \gamma(t)$ denotes the evaluation map at time $t$. On a complete and separable metric space $(X,d)$, any $W_2$-geodesic has a \textit{lifting} to a measure on the space of geodesics in the following sense.

\begin{thm}\label{W2 lift}\cite[Theorem 3.2]{L07}
	Let $(\mu_t)_{t \in [0,1]}$ be a $W_2$-geodesic. Then there exists $\boldsymbol{\pi} \in \mathcal{P}(\Geo(X))$ so that
	\begin{equation*}
		(e_t)_{*}(\boldsymbol{\pi}) = \mu_t \; \; \forall t \in [0,1],
	\end{equation*}
	\begin{equation*}
		|\dot{\mu}_t|^2= \int |\dot{\gamma}_t|^2 \, d\boldsymbol{\pi}(\gamma), \; \; \text{for a.e. } t \in[0.1],
	\end{equation*}
	where $e_t(\gamma) := \gamma(t)$ is the evaluation map at time $t$.
\end{thm}

These are called \textit{optimal dynamical plans}. This motivates the following definition: for any $\mu_0, \mu_1 \in \mathcal{P}_2(X)$, we denote by $\OptGeo(\mu_0, \mu_1)$ the space of all optimal dynamical plans from $\mu_0$ to $\mu_1$. 

\begin{de}\label{nonbranching} 
Given two geodesics $\gamma^1 \neq \gamma^2$ on a geodesic metric space $(X,d)$. Assume $\gamma^1, \gamma^2$ are constant speed and parameterized on the unit interval. we say $\gamma^1$ and $\gamma^2$ \textit{branch} if there exists $0 < t < 1$ such that $\gamma^1_s = \gamma^2_s$ for all $s \in [0,t]$. A subset $S \subseteq \Geo(X)$ is called a \textit{set of non-branching geodesics} if there are no branching pairs in $S$. A geodesic metric space for which $\Geo(X)$ is itself a set of non-branching geodesics is called \textit{non-branching}. 
\end{de}
Many results were shown for various types of $CD$ spaces under the additional non-branching assumption. These include the local-to-global property, tensorization property and local Poincar\'e inequality, see \cite{S06a, BS10, LV07}. A weaker assumption was introduced in \cite{RS14} for which these results generalize. 

\begin{de}
A metric measure space $(X,d,m)$ is called \textit{essentially non-branching} if for any $\mu_0, \mu_1 \in \mathcal{P}_2(X)$ absolutely continuous with respect to $m$, any element of $\OptGeo(\mu_0, \mu_1)$ is concentrated on a set of non-branching geodesics. 
\end{de}

$\RCD(K,\infty)$ spaces are shown to be essentially non-branching by the results of \cite{DS08, AGS14b} and \cite{RS14}. We will frequently refer to the following theorem from \cite{GRS16} shown for finite dimensional $\RCD(K,N)$ spaces, see also \cite{R12, RS14} for related results and \cite{CM17a} for the same result in the case of essentially non-branching $\MCP(K,N)$ spaces.

\begin{thm}\label{directional BG}(\cite{GRS16}, \cite[Theorem 1.1]{CM17a})
	Let $(X,d,m)$ be an $\RCD(K,N)$ space for some $K \in \mathbb{R}$ and $N \in [1,\infty)$. If $\mu_0, \mu_1 \in \mathcal{P}_2(X)$ with $\mu_0 = \rho_0m \ll m$, then there exists a unique $\nu \in \OptGeo(\mu_0, \mu_1)$. $(e_t)_{*}(\nu) \ll m$ for any $t \in [0,1)$ and such $\nu$ is given by a unique map $S:\supp(\mu_0) \to \Geo(X)$ in the sense that $\nu = S_{*}(\mu)$. Moreover, if $\mu_0$, $\mu_1$ have bounded support and $\norm{\rho_0}_{L^{\infty}(m)} < \infty$, then 
\begin{equation}\label{directional BG eq1}
	\norm{\rho_t}_{L^{\infty}(m)} \leq \frac{1}{(1-t)^N}e^{Dt\sqrt{(N-1)K^-}}\norm{\rho_0}_{L^\infty(m)}, \; \; \forall t \in [0,1),
\end{equation} 
where $D := \diam(\supp(\mu_0) \cup \supp(\mu_1))$ and $K^- := \max\braces{-K,0}$
\end{thm}

\begin{rem}\label{a.e. unique geodesic}
In particular, this implies for any $p \in X$, there is a unique geodesic between $p$ and $x$ for $m$-a.e. $x \in X$. Using the Kuratowski and Ryll-Nardzewski measurable selection theorem, one may select constant speed geodesics $\gamma_{x,p}$ from all $x \in X$ to $p$ so that the map $X \times [0,1] \ni (x, t) \mapsto \gamma_{x,p}(t)$ is Borel. This then guarantees such a choice is unique up to a set of measure 0. Similarly, one may select constant speed geodesics $\gamma_{x,y}$ for all $x,y \in X$ so that the map $X \times X \times [0,1] \ni (x,y,t) \mapsto \gamma_{x,y}(t)$ is Borel. Using, for example, the arguments in \cite[Section 4]{C14}, the set of points $(x,y) \in X \times X$ connected by non-unique geodesics is analytic. $(m \times m)$-almost everywhere uniquness of the Borel selection then follows using Fubini's theorem. In cases where we fix geodesics in this manner, we say we take a \textit{Borel selection} of geodesics $\gamma_{x,p}$ from all $x \in X$ to $p$ (or $\gamma_{x,y}$ from all $x \in X$ to all $y \in X$).
\end{rem}

\subsection {RCD(K,N) structure theory}

We review of the structure theory of $\RCD$ spaces in this subsection. We will assume basic familiarity with pointed Gromov-Hausdorff (pGH) and pointed measure Gromov-Hausdorff (pmGH) convergence and refer to \cite{BBI14, V09, GMS15} for details. 

A notion of considerable interest for $\RCD(K,N)$ spaces is that of measured tangents. Similar objects have been well-studied in the setting of Alexandrov spaces and Ricci limits (see, for example, \cite{BGP92} and \cite{CN13} for an overview). Given a m.m.s. $(X,d,m)$, $\bar{x} \in X$ and $r \in (0,1)$, consider the \textit{normalized rescaled pointed m.m.s. (p.m.m.s.)} $(X, r^{-1}d, m^{\bar{x}}_r, \bar{x})$ where
\begin{equation*}
	m^{\bar{x}}_{r} := \bigg(\int_{B_r(\bar{x})} 1 - \frac{d(x,y)}{r} dm(y)\bigg)^{-1}m.
\end{equation*}
In what follows let $(X,d,m)$ be an $\RCD(K,N)$ space for some $K \in \mathbb{R}$ and $N \in (1, \infty)$. We define
\begin{de}\label{tangent def}(The collection of tangent spaces $\Tan(X,d,m,\bar{x})$)  Let $\bar{x} \in X$. A p.m.m.s. $(Y, d_Y, m_Y, \bar{y})$ is called a \textit{tangent (cone)} of $(X,d,m)$ at $\bar{x}$ if there exists a sequence of radii $r_i \downarrow 0$ so that $(X, r_i^{-1}d, m^{\bar{x}}_{r_i}, \bar{x})$ $\to (Y, d_Y, m_Y, \bar{y})$ as $i \to \infty$ in the pmGH topology. The collection of all tangents of $(X,d,m)$ at $\bar{x}$ is denoted $\Tan(X,d,m,\bar{x})$.
\end{de}
A standard compactness argument by Gromov shows that $\Tan(X,d,m,\bar{x})$ is non-empty for any $\bar{x} \in X$. The rescaling and stability properties of the $\RCD(K,N)$ condition under pmGH convergence (see \cite[Theorem 7.2]{GMS15} and \cite[Theorem 6.11]{AGS14b}) show that every element of $\Tan(X,d,m,\bar{x})$ is an $\RCD(0,N)$ space. 

Let $c_k = \int_{B_1(0)} 1-|x| \, d\mathscr{L^k}(x)$ be the normalization constant of the $k$-dimensional Lebesgue measure and define the \textit{$k$-dimensional regular set $\mathcal{R}_k$} by
\begin{equation*}
	\mathcal{R}_k := \{x \in X \, : \, \Tan(X,d,m,x) = \{(\mathbb{R}^k, d_{E}, c_k\mathscr{H}^k,0^k)\} \}.
\end{equation*}
Define $\mathcal{R}_{\reg} := \bigcup\limits_{k=1}^{\lfloor N \rfloor} \mathcal{R}_k$ the \textit{regular set} of $X$ and $\mathcal{S} := X \backslash \mathcal{R}_{\reg}$ the \textit{singular set}. In \cite{GMR15}, it was shown that, for $m$-a.e. $x \in X$, there exists some $k \in \mathbb{N}$, $1 \leq k \leq N$ so that $(\mathbb{R}^k, d_{E}, c_k\mathscr{H}^k, 0^k) \in \Tan(X, d, m, x)$. This was improved in \cite{MN19}, where it was shown that each $\mathcal{R}_k$ is $m$-measurable and $m(\mathcal{S})=0$. A final improvement was made in \cite{BS20} with the following theorem. 

\begin{thm}\label{unique local dimension 1} (Constancy of the dimension \cite[Theorem 3.8]{BS20}) Let $(X,d,m)$ be an $\RCD(K,N)$ m.m.s. for some $K \in \mathbb{R}$ and $1 \leq N < \infty$. Assume $X$ is not a point. There exists a unique $n \in \mathbb{N}$, $1 \leq n \leq N$ so that $m(X\backslash \mathcal{R}_{n})=0$.
\end{thm}

The same theorem was proved in the case of Ricci limits in \cite{CN12} from the main result there and as such also follows from Theorem \ref{main theorem 2}, see Theorem \ref{unique local dimension 2}. By \cite[Theorem 1.2]{K18}, see also \cite[Theorem 1.9]{KL18} in the case of Ricci limits, it is known that the unique $n$ in Theorem \ref{unique local dimension 1} is also the largest integer $n$ for which $\mathcal{R}_n$ is non-empty.

\subsection{Additional RCD(K,N) theory}
In this subsection we record several theorems for $\RCD(K,N)$ spaces which will be of use later.

Define the coefficients $\tilde{\sigma}_{K,N}(\cdot):[0,\infty)\to \mathbb{R}$ by
 \[
\tilde{\sigma}_{K,N}(\theta):=\left\{
\begin{array}{ll}
\theta \sqrt{\frac{K}{N}} \, {\rm cotan} \left(\theta \sqrt{\frac{K}{N}} \right),&\qquad\textrm{ if }K>0,\\
1 &\qquad\textrm{ if }K=0 ,\\
\theta \sqrt{-\frac{K}{N}} \, {\rm cotanh} \left(\theta \sqrt{-\frac{K}{N}} \right),&\qquad\textrm{ if }K<0,
\end{array}
\right.
\]
then one has the following sharp bound on the measure valued Laplacian of distance functions. 
\begin{thm}\label{lap d bound}(Laplacian comparison for the distance function \cite[Corollary 5.15]{G15}) Let $(X,d,$ $m)$ be a compact $\RCD(K,N)$ space for some $K \in \mathbb{R}$ and $N \in (1, \infty)$. For $x_0 \in
 X$ denote by $d_{x_0}:X \to [0,\infty)$ the function $x \mapsto d(x,x_0)$. Then 
\begin{equation*}
\frac{d_{x_0}^2}{2} \in D(\boldsymbol{\Delta}) \hspace{0.5cm}\text{with}\hspace{0.5cm} \boldsymbol{\Delta}\frac{d_{x_0}^2}{2} \leq N\tilde{\sigma}_{K,N}(d_{x_0})m \hspace{0.5cm} \forall x_0 \in X
\end{equation*}
and
\begin{equation*}
d_{x_0} \in D(\boldsymbol{\Delta}, X\backslash\left\{x_0\right\}) \hspace{0.5cm}\text{with}\hspace{0.5cm} \boldsymbol{\Delta}d_{x_0}\rvert_{X\backslash \left\{x_0\right\}} \leq \frac{N\tilde{\sigma}_{K,N}(d_{x_0})-1}{d_{x_0}}m \hspace{0.5cm} \forall x_0 \in X.
\end{equation*}
\end{thm}
$\boldsymbol{\Delta}d_{x_0}\rvert_{X\backslash \left\{x_0\right\}}$ is defined similar to Definition \ref{measure lap}, the difference being that the test functions $g$ must be compactly supported in $X\backslash \left\{x_0\right\}$.

\begin{rem}\label{lap d bound radon}
We mention, and will use the fact, that in the case where $X$ is noncompact one can make essentially the same statement. Some small adjustments are needed since the Laplacians of $d_{x_0}$ on $X \backslash \{x_0\}$ and $d_{x_0}^2$ on $X$ are not naturally guaranteed to be signed Radon measures. To accomodate this, a weakening of the definition of the measure valued Laplacian was given in \cite[Definition 2.11, 2.12]{CM18}.  This definition allows the Laplacian to be, more generally, a Radon functional (i.e. in $(C_C(X))'$). The difference is that a Radon functional has a representation as the difference of two possibly infinite positive Radon measures by the Riesz-Markov-Kakutani representation theorem, whereas in the case of a signed Radon measure, at least one of these must be finite. It was shown in \cite[Corollary 4.17, 4.19]{CM18} that the Laplacian comparison for the distance function holds as stated with this weaker definition.  As such, we will, by a slight abuse of notation, treat the Laplacian of the distance function $d_{x_0}$ on $X \backslash \{x_0\}$ as a signed Radon measure in the few instances where we use Theorem \ref{lap d bound} in integration against compactly supported functions. Note that due to the comparison theorem, this Laplacian is locally a signed Radon measure, having at most an infinite negative part. 
\end{rem}

We will also need the Li-Yau Harnack inequality \cite{LY86} and the Li-Yau gradient inequality \cite{BG11, LY86}. These were proved for the $\RCD$ setting, in the finite measure case in \cite{GM14}, and in general in \cite{J15}.
\begin{thm}\label{LY Harnack} (Li-Yau Harnack inequality \cite{GM14, J15})
Let $(X, d, m)$ be an $\RCD(K,N)$ space for some $K \in \mathbb{R}$ and $N \in [1, \infty)$. Let $f \in L^p(m)$ for some $p \in [1,\infty)$ be non-negative. If $K \geq 0$, then for every $x, y \in X$ and $0 < s < t$ it holds that
\begin{equation*}
	(H_t f)(y) \geq (H_s f)(x) e^{\frac{d^2(x,y)}{4(t-s)e^{\frac{2Ks}{3}}}}\bigg(\frac{1-e^{\frac{2K}{3}s}}{1-e^{\frac{2K}{3}t}}\bigg)^{\frac{N}{2}}.
\end{equation*}
If instead $K < 0$, then
\begin{equation*}
	(H_t f)(y) \geq (H_s f)(x) e^{\frac{d^2(x,y)}{4(t-s)e^{\frac{2Kt}{3}}}}\bigg(\frac{1-e^{\frac{2K}{3}s}}{1-e^{\frac{2K}{3}t}}\bigg)^{\frac{N}{2}}.
\end{equation*}
\end{thm}

\begin{thm}\label{LY grad} (Li-Yau gradient inequality \cite{GM14, J15})
Let $(X,d,m)$ be an $\RCD(K,N)$ space for some $K \in \mathbb{R}$ and $N \in [1, \infty)$. Let $f \in L^p(m)$ for some $p \in [1,\infty)$ be non-negative. Then for every $t > 0$ it holds that
\begin{equation*}
	|\nabla H_tf|^2 \leq e^{-\frac{2Kt}{3}}(\Delta H_tf)H_tf + \frac{NK}{3}\frac{e^{-\frac{4Kt}{3}}}{1-e^{-\frac{2Kt}{3}}}(H_tf)^2 \hspace{0.5cm} m\text{-a.e.}.
\end{equation*}
\end{thm}

$\RCD(K,N)$ spaces also satisfy the parabolic maximum principle, see \cite[Section 3]{L18} and \cite[Section 4.1]{GH08} for full details.
\begin{de}\label{weak solutions} (\cite[Definition 3.1]{L18})
Let $(X,d,m)$ be an $\RCD(K,N)$ space with $K\in \mathbb{R}$ and $N\in (1,\infty)$. Let $I$ be an open interval in $\mathbb{R}$, $\Omega$ be an open seubset of $X$, and $g\in L^2(\Omega)$. We say  that a function $u: I\rightarrow W^{1,2}(\Omega)$ satisfies the parabolic equation
\begin{equation*}
	\frac{\partial}{\partial t}u-\Delta u\leq g,\quad\mbox{weakly in }I\times \Omega,
\end{equation*}
if for every $t \in I$, the Fr\'{e}chet derivative of $u$, denoted by $\frac{\partial}{\partial t}u$, exists in $L^2(\Omega)$  and for any nonnegative function $\psi\in W^{1,2}(\Omega)$, it holds
\begin{equation*}
	\int_\Omega \frac{\partial}{\partial t}u(t,\cdot)\psi\,dm + \mathcal{E}\big(u(t,\cdot),\psi\big)\leq \int_\Omega g\psi\,dm.
\end{equation*}
\end{de} 

\begin{thm}\label{parabolic maximum principle}(Parabolic maximum principle \cite[Lemma 3.2]{L18})
Let $(X,d,m)$ be an $\RCD(K,N)$ space with $K\in \mathbb{R}$ and $N\in (1,\infty)$. Fix $T\in (0,\infty]$ and open subset $\Omega \subseteq X$. Assume that a function $u: (0,T)\rightarrow W^{1,2}(\Omega)$, with $u_+(t,\cdot)=\max\{u(t,\cdot),0\}\in W^{1,2}(\Omega)$ for any $t\in (0,T)$, satisfies the following equation with initial value condition:
\begin{equation*}
\begin{cases}
\frac{\partial}{\partial t}u-\Delta u\leq 0,\quad &{\hbox{weakly in}}\,\ (0,T)\times\Omega,\\
u_+(t,\cdot)\rightarrow 0,\quad &{\hbox{in}}\,\ L^2(\Omega)\,\ \hbox{as}\,\ t\rightarrow0.
\end{cases}
\end{equation*}
Then $u(t,x)\leq0$ for any $t$ in $(0,T)$ and $m$-a.e. $x$ in $\Omega$.
\end{thm}

\subsection{Mean value and integral excess inequalities}\label{subsection 2.10}

We refer to \cite{CN12} in the smooth case, and \cite{MN19} in the RCD case, for the proofs of the statements in this subsection. We start with the existence of good cut off functions.

\begin{lem}\label{good ball cut off function} (Existence of good cut off functions \cite[Lemma 3.1]{MN19})
Let $(X,d,m)$ be an $\RCD(K,N)$ space for some $K \in \mathbb{R}$ and $N \in [1, \infty)$. Then for every $x \subset X$, for every $R > 0$ and $0 < r < R$ there exists a Lipschitz function $\psi^r: X \to \mathbb{R}$ satisfiying: 
\begin{enumerate}
	\item $0 \leq \psi^r \leq 1$ on $X$, $\psi^r \equiv 1$ on $B_{r}(x)$ and supp$(\psi) \subset B_{2r}(x)$;
	\item $r^2|\Delta \psi^r| + r|\nabla \psi^r| \leq C(K,N,R)$ $m\text{-a.e.}.$
\end{enumerate}
\end{lem}

For any subset $C$ in a metric space, we denote by $T_{r}(C)$ the $r$-tubular neighbourhood of $C$ and for $r_1 > r_0 > 0$, $A_{r_0, r_1}(C) := T_{r_1}(C) \backslash \overline{T_{r_0}(C)}$ the $(r_0, r_1)$-annular neighbourhood of $C$.

\begin{lem}\label{good cut off function} (Existence of good cut off functions on annular neighbourhoods \cite[Lemma 3.2]{MN19})
Let $(X,d,m)$ be an $\RCD(K,N)$ space for some $K \in \mathbb{R}$ and $N \in [1, \infty)$. Then for every closed subset $C \subset X$, for every $R > 0$ and $0 < 10r_0 < r_1 \leq R$ there exists a Lipschitz function $\psi: X \to \mathbb{R}$ satisfiying: 
\begin{enumerate}
	\item $0 \leq \psi \leq 1$ on $X$, $\psi \equiv 1$ on $A_{3r_0, \frac{r_1}{3}}(C)$ and supp$(\psi) \subset A_{2r_0,\frac{r_1}{2}}(C)$;
	\item $r_0^2|\Delta \psi| + r_0|\nabla \psi| \leq C(K,N,R)$ $m$-a.e. on $A_{2r_0,3r_0}(C)$;
	\item $r_1^2|\Delta \psi| + r_1|\nabla \psi| \leq C(K,N,R)$ $m$-a.e. on $A_{\frac{r_1}{3},\frac{r_1}{2}}(C)$.
\end{enumerate}
\end{lem}

\begin{rem}\label{pw good cut off function}
Note that the gradient bounds in~\ref{good ball cut off function} and~\ref{good cut off function} are naturally $m$-a.e.. However, since the proof  involves only~\ref{Bakry-Ledoux}, by Remark \ref{pw Bakry-Ledoux}, choosing continuous representatitves and using the local Lipschitz constant of $\psi$ for $|\nabla \psi|$, these statements can be made pointwise. 
\end{rem}

As demonstrated in \cite{CN12}, and later for the $\RCD$ setting in \cite{MN19}, several key estimates, including heat kernel bounds, the mean value, $L^1$-Harnack, and integral Abresch-Gromoll inequalities can be proved starting from the existence of good cut off functions and the Li-Yau Harnack inequality \ref{LY Harnack}. 

\begin{lem}\label{heat kernel bounds}(Heat kernel bounds \cite[lemma 3.3]{MN19})
Let $(X,d,m)$ be an $\RCD(K,N)$ space for some $K \in \mathbb{R}$, $N \in (1, \infty)$ and let $H_t(x,y)$ be the heat kernel for some $x \in X$. Then for every $R > 0$, for all $0 < r < R$ and $t \leq R^2$, 
\begin{enumerate}
	\item if $y \in B_{10\sqrt{t}}(x)$, then $\frac{C^{-1}(K,N,R)}{m(B_{10\sqrt{t}}(x))} \leq H_t(x,y) \leq \frac{C(K,N,R)}{m(B_{10\sqrt{t}}(x))}$
	\item $\int\limits_{X \backslash B_{r}(x)} H_t(x,y)dm(y) \leq C(K,N,R)r^{-2}t$.
\end{enumerate}
\end{lem}

\begin{lem}\label{MV L1H inequality}(Mean value and $L^1$-Harnack inequality \cite[Lemma 3.4]{MN19})
Let $(X,d,m)$ be an $\RCD(K,N)$ space for some $K \in \mathbb{R}$, $N \in (1, \infty)$ and let $0 < r < R$. If $u: X \times [0, r^2] \to \mathbb{R}$, $u(x,t) = u_t(x)$, is a nonnegative Borel function with compact support at each time $t$ and satisfiying $(\partial_t - \Delta)u \geq -c_0$ in the weak sense, then,
\begin{equation*}
	\fint\limits_{B_r(x)} u_0 \leq C(K, N, R)[u_{r^2}(x) + c_0r^2] \; \; \text{for $m$-a.e. }x.
\end{equation*}
More generally the following $L^1$-Harnack inequality holds
\begin{equation*}
	\fint\limits_{B_r(x)} u_0 \leq C(K,N,R)[\essinf_{y \in B_r(x)} u_{r^2}(y)+c_0r^2] \; \; \forall x \in X.
\end{equation*}
\end{lem}

\begin{rem}\label{MV L1H inequality pw}
Lemma 3.4 of \cite{MN19}, whose proof follows \cite[Lemma 2.1]{CN12}, treats the continuous case of $u$. However, in what follows  we will want to use this inequality for $u_{t} = |\nabla h_t(f)|$, which is not known to have a continuous representatitive. The proof of this statement for Borel $u$ follows exactly as in the continuous case with the obvious measure-theoretic adjustments. 
\end{rem}

Applying \ref{MV L1H inequality} to a function which is constant in time gives the following classical mean value inequality.

\begin{cor}\label{classical MV inequality}(Classical mean value inequality \cite[Corollary 3.5]{MN19})
Let $(X,d,m)$ be an $\RCD(K,$ $N)$ space for some $K \in \mathbb{R}$, $N \in (1, \infty)$ and let $0 < r < R$. If $u: X \to \mathbb{R}$ is a nonnegative Borel function with compact support with $u \in D(\boldsymbol{\Delta})$ and satisfies $\boldsymbol{\Delta} u \leq c_0m$ in the sense of measures, then for $0 < r \leq R$, 
\begin{equation*}
	\fint\limits_{B_r(x)} u \leq C(K, N, R)[u(x) + c_0r^2] \; \; \text{for $m$-a.e. }x.
\end{equation*}
\end{cor}

This, in combination with the existence of good cut off functions and Laplacian estimates on distance functions, allows one to prove an integral Abresch-Gromoll inequality. For points $p$ and $q$ in a metric space, we define the excess function $e_{p,q}(x) := d(p,x) + d(x,q) - d(p,q)$.

\begin{thm}\label{Abresch-Gromoll} (Integral Abresch-Gromoll inequality \cite[Theorem 3.6]{MN19})
Let $(X,d,m)$ be an $\RCD(K,N)$ space for some $K \in \mathbb{R}$, $N \in (1, \infty)$; let $p,q \in X$ with $d_{p,q} := d(p,q) \leq 1$ and fix $0 < \epsilon < 1$.\\
If $x \in A_{\epsilon d_{p,q}, 2d_{p,q}}(\left\{p,q\right\})$ satisfies $e_{p,q}(x) \leq r^2d_{p,q} \leq \bar{r}(K, N, \epsilon)^2d_{p,q}$, then 
\begin{equation*}
	\fint\limits_{B_{rd_{p,q}}(x)} e_{p,q}(y) dm(y) \leq C(K,N,\epsilon)r^2d_{p,q}.
\end{equation*}
\end{thm}

Combined with Bishop-Gromov volume comparison, this immediately implies the classical Abresch-Gromoll inequality \cite{AG90}.
\begin{cor}\label{classical Abresch-Gromoll} (Classical Abresch-Gromoll inequaltiy \cite[Corollary 3.7]{MN19})
Let $(X,d,m)$ be an $\RCD(K,N)$ space for some $K \in \mathbb{R}$, $N \in (1, \infty)$; let $p,q \in X$ with $d_{p,q} := d(p,q) \leq 1$ and fix $0 < \epsilon < 1$.\\
If $x \in A_{\epsilon d_{p,q}, 2d_{p,q}}(\left\{p,q\right\})$ satisfies $e_{p,q}(x) \leq r^2d_{p,q} \leq \bar{r}(K, N, \epsilon)^2d_{p,q}$, then there exists $\alpha(N) \in (0,1)$ such that
\begin{equation*}
	e_{p,q}(y) \leq C(K,N,\epsilon) r^{1+\alpha(N)}d_{p,q} \, , \, \forall y \in B_{rd_{p,q}}(x).
\end{equation*}
\end{cor}

\section{Differentiation formulas for Regular Lagrangian flows}\label{section 3}

\subsection{Regular Lagrangian flow}

In what follows, we will always be on some $\RCD(K,N)$ space $(X,d,m)$ for $K \in \mathbb{R}$ and $N \in [1,\infty)$. In \cite{CN12}, the crucial idea is to understand the geometric properties of the gradient flow with respect to heat flow approximations of the distance function. We will do the same with the Regular Lagrangian flow which was first introduced by Ambrosio in \cite{A04} on $\mathbb{R}^d$ and generalized to the metric measure setting by Ambrosio-Trevisan in \cite{AT14}. The setup is quite general and we refer to \cite{AT14} for full details. We will be interested in applying this theory specificially for vector fields in the $L^2$ tangent module of an RCD space.

\begin{de}\label{L^2 vector fields}(Time-dependent $L^2$ vector fields)
Let $T > 0$. $V: [0,T] \to L^{2}(TX)$ is a time-dependent $L^2$ vector field iff the map $[0,T]\ni t\mapsto V_t \in L^2(TX)$ is Borel.\\
$V$ is bounded iff 
\begin{equation*}
	\norm{V}_{L^{\infty}} := \norm{|V|}_{L^\infty([0,T]\times X)} < \infty.
\end{equation*}
$V \in L^1([0,T],L^2(TX))$ iff
\begin{equation*}
	 \int\limits_{0}^{T} \norm{V_t}_{L^2(TX)} dt < \infty.
\end{equation*}
\end{de}

\begin{de}\label{RLF}(Regular Lagrangian flow)
Given a time-dependent $L^2$ vector field $(V_t)$. A Borel map $F:[0,T] \times X \to X$ is a \textit{Regular Lagrangian flow (RLF)} to $V_t$ iff the following holds:
	\begin{enumerate}
	\item $F_0(x) = x$ and $[0,T] \ni t \mapsto F_t(x)$ is continuous for every $x \in X$;
	\item For every $f \in \TestF(X)$ and $m$-a.e. $x \in X$, $t \mapsto f(F_t(x))$ is in $W^{1,1}([0,T])$ and 
		\begin{equation}\label{RLF eq1} 
			\frac{d}{dt}f(F_t(x))= df(V_t)(F_t(x)) \; \; \text{ for a.e. } t \in [0,T];
		\end{equation}
	\item There exists a constant $C := C(V)$ so that $(F_t)_{*}m \leq Cm$ for all $t$ in $[0,T]$. 
	\end{enumerate}
\end{de}

\begin{rem}\label{RLF W12loc}
In the case where $V \in L^1([0,T],L^2(TX))$ and $F$ is an RLF of $V$,  using a standard Fubini's theorem argument and that $\TestF(X)$ is dense in $W^{1,2}(X)$, we have for every $f \in W^{1,2}_{loc}(X)$, $t \mapsto f(F_t(x))$ is in $W^{1,1}([0,T])$ and 
\begin{equation}
	\frac{d}{dt}f(F_t(x))= df(V_t)(F_t(x)) \; \; \text{ for a.e. } t \in [0,T].
\end{equation}
\end{rem}

\cite{AT14} gives the existence and uniqueness of RLFs to $(V_t)$ in a certain class of vector fields. We use the following weaker formulation of their result and note that only a bound on the symmetric part of $\nabla V_t$ is needed. 
\begin{thm}\label{RLF existence}(Existence and uniqueness of Regular Lagrangian flow \cite{AT14})
Let $(V_t) \in L^1([0,T],$ $L^2(TX))$ satisfy $V_t\in D(div)$ for a.e. $t \in [0,T]$ with
\begin{equation*}
	div(V_t) \in L^1([0,T],L^2(m)) \; \;  \;(div(V_t))^- \in L^1([0,T], L^{\infty}(m)) \; \; \; \nabla V_t \in L^1([0,T],L^2(T^{\otimes 2}X)).
\end{equation*}
There exists a unique, up to $m$-a.e. equality, RLF $(F_t)_{t \in [0,T]}$ for $(V_t)$. The bound
\begin{equation}\label{RLF existence eq1}
	(F_t)_{*}(m) \leq exp(\int\limits_{0}^{t} \norm{\diver(V_s)^-}_{L^{\infty}(m)}\,ds)m
\end{equation}
holds for every $t \in [0,T]$.
\end{thm} 

\begin{rem}\label{RLF local volume bound}
	It was pointed out to the author by Nicola Gigli that the estimate \eqref{RLF existence eq1} can be localized for any $S \in \mathscr{B}(X)$,
	\begin{equation*}
		(F_t)_{*}(m\rvert_S) \leq exp(\int\limits_{0}^{t} \norm{\diver(V_s)^-}_{L^{\infty}((F_s)_*(m \rvert_S))}\, ds)m.
	\end{equation*}
	This follows from \cite[(4-22)]{AT14} choosing $\beta(z) := z^p$ for $p \to \infty$. 
\end{rem}


For $(F_t)$ an RLF to some $(V_t)$, we will be interested in expressing  $\frac{d}{dt} d(F_{t}(x),F_{t}(y))$ in two ways: using $V_t$ in a first order variation formula and $\nabla V_t$ in a second order formula, see \eqref{Hessian interpolation}, which we show in Subection \ref{subsection 6.3}.

\begin{prop}\label{first order differentiation} (First order differentiation formula along RLFs)
Let $T > 0$ and $U, V \in L^1([0,T],$ $L^2(TX))$. If $(F_t),(G_t)$ are the Regular Lagrangian flows of $(U_t), (V_t)$ respectively, then for $m$-a.e. $x,y \in X$, $d(F_t(x),G_t(y)) \in W^{1,1}([0,T])$ and 
\begin{equation*}
	\frac{d}{dt} d(F_t(x),G_t(y))= \inner{\nabla d_{G_t(y)}}{U_t}(F_t(x)) + \inner{\nabla d_{F_t(x)}}{V_t}(G_t(y)) \; \; \text{ for a.e. $t \in [0,T].$}
\end{equation*}
\end{prop}
\begin{proof}
	It is known that $\RCD(K, \infty)$ spaces have the tensorization of Cheeger energy property from \cite[Theorem 6.17]{AGS14b} and the density of the product algebra property from \cite[Proposition A.1]{BS20}, see also \cite[Defintion 3.8, 3.9]{GR18} for definitions.  Consider the vector field $(W_t)$ defined by requiring, for all $f \in W^{1,2}(X \times X)$,
	\begin{equation*}
		\inner{W_t}{\nabla f}(x,y) = \inner{U_t}{\nabla f_y}(x) + \inner{V_t}{\nabla f_x}(y),
	\end{equation*}
for $(m \times m)$-a.e. $(x,y) \in X \times X$. The tensorization of Cheeger energy is used implicitly in this definition and the vector field is naturally in $L^1([0,T], L^2_{loc}(T(X \times X)))$. We refer to \cite{GR18} for a rigorous treatment of locally $L^2$ vector fields and the corresponding theory of RLFs. We mention a slightly more careful, alternative definition of $(W_t)$ was also given in \cite[Proposition 3.7, Theorem 3.13]{GR18}, where the expected decomposition of the module $L^0(T^*(X \times X))$ was shown for spaces with tensorization of Cheeger energy and density of the product algebra properties. By \cite[Proposition A.2]{BS20}, $(F_t, G_t)$ is an RLF of $(W_t)$, from which the proposition follows by definition of $(W_t)$.
\end{proof}

\subsection{Continuity equation}

We give a brief summary of the theory of continuity equations in this section. These are intimately related to Regular Lagrangian flows but provide a more convenient language for the discussion of local flows in cases where RLFs, which as defined are of a global nature, may not exist. 

\begin{de}\label{bounded compressions}{Curves of bounded compression \cite[Definition 2.3.21]{G18}}
We say a curve $(\mu_t)_{t \in [0,T]} \subseteq \mathcal{P}_2(X)$ is a curve of bounded compression iff 
\begin{enumerate}
	\item It is $W_2$-continuous;
	\item For some $C > 0$, $\mu_t \leq Cm$ for every $t \in [0,T]$. 
\end{enumerate}
\end{de}

\begin{de}\label{continuity equation}(Solutions of continuity equation \cite{GH15}, \cite[Definition 2.3.22]{G18})
Let $(\mu_t)_{t \in [0,T]}$ $\subseteq \mathcal{P}_2(X)$ be a curve of bounded compression and $V \in L^1([0,T],$$L^2(TX))$. We say that $(\mu_t, V_t)$ solves the continuity equation
\[
\frac{d}{d t}\mu_t+\diver(V_t\mu_t)=0
\]
iff, for every $f \in W^{1,2}(X)$, the map $t\mapsto\int f\,d\mu_t$ is absolutely continuous and satisfies
\begin{equation}\label{continuity equation eq1}
\frac{d}{d t}\int f\,d \mu_t=\int  df(V_t)\,d\mu_t \; \; \text{ for a.e. }t\in[0,T]. 
\end{equation}
\end{de}

\begin{rem}\label{continuity equation notation}
By abuse of notation we will sometimes say $(\mu_t)$ solves the continuity equation $\frac{d}{d t}\mu_t+\diver(V_t\mu_t)=0$ for some vector field $(V_t)$ which is only locally $L^2$, for example, $V_t = -\nabla d_p$ for some $p \in X$. In this case, $(\mu_t)$ is always compactly supported for every $t$ and it is understood that we cut off the vector field $V_t$ outside of this support. 
\end{rem}

As shown in \cite{AT14}, RLFs are very closely related to the solutions of continuity equations; they can be thought of as realizations of these solutions as maps on the space itself. 

\begin{thm}\label{continuity equation = RLF}(\cite{AT14})
Let $(V_t)$ satisfy the conditions of \ref{RLF existence} and $(F_t)$ be the corresponding unique Regular Lagrangian flow. If $\mu_0 \in \mathcal{P}_2(X)$ with bounded density, then $\mu_t := (F_t)_{*}(\mu_0)$ is a distributional solution of the continuity equation $\frac{d}{d t}\mu_t+\diver(V_t\mu_t)=0$. 
\end{thm}

\begin{rem}\label{continuity equation existence}
The existence and uniqueness of solutions to $\frac{d}{d t}\mu_t+\diver(V_t\mu_t)=0$ starting at some $\mu_0$ of bounded density is proved in \cite{AT14} for $(V_t)$ satisfying the conditions of Theorem \ref{RLF existence}. In fact, the existence and uniqueness of RLFs in \ref{RLF existence} is shown in part by using the existence and uniqueness on the level of continuity equations combined with a superposition principle.
\end{rem}

RLFs from vector fields with a two-sided divergence bound are $m$-a.e. invertible. To be precise,

\begin{prop}\label{RLF inverse}
Let $(V_t) \in L^1([0,T],L^2(TX))$ satisfy $V_t\in D(div)$ for a.e. $t \in [0,T]$ with
\begin{equation*}
	div(V_t) \in L^1([0,T],L^2(m))\; \; \; div(V_t) \in L^1([0,T], L^{\infty}(m))\; \; \; \nabla V_t \in L^1([0,T],L^2(T^{\otimes 2}X)).
\end{equation*}
Let $(F_t)$ be the unique RLF of $(V_t)_{t \in [0,T]}$ and $(G_t)$ be the unique RLF of $(-V_{T-t})_{t \in [0,T]}$. For $m$-a.e. $x \in X$ and any $0 \leq t \leq T$,
\begin{equation*}
	G_{t}(F_T(x)) = F_{T-t}(x).
\end{equation*}
\end{prop}

\begin{proof}
	We first show $G_{T}(F_T(x)) = x$ for $m$-a.e. $x \in X$. Define the time-dependent $L^2$ vector field $(W_t)_{t \in [0,2T]}$ by
\begin{align*}
W_t :=\begin{cases}
	V_t &\text{ if } 0\leq t \leq T\\
	-V_{2T-t} &\text{ if } T < t \leq 2T.
	\end{cases}
\end{align*}
For any $\mu$ with compact support and bounded density, $(\mu_t)_{t \in [0,2T]}$ defined by
\begin{align*}
\mu_t :=\begin{cases}
	(F_t)_{*}(\mu) &\text{ if } 0\leq t \leq T\\
	(G_{t-T})_{*}((F_T)_{*}(\mu)) &\text{ if } T < t \leq 2T
	\end{cases}
\end{align*}
solves the continuity equation $\frac{d}{d t}\mu_t+\diver(W_t\mu_t)=0$ by Theorem \ref{continuity equation = RLF}. This in particular means $(\mu_t)_{t \in [0,T]}$ solves the continuity equation $\frac{d}{d t}\mu_t+\diver(V_t\mu_t)=0$ on $[0,T]$. It is then easy to check by definition that
\begin{align*}
\mu'_t :=\begin{cases}
	\mu_t &\text{ if } 0\leq t \leq T\\
	\mu_{2T-t} &\text{ if } T < t \leq 2T
	\end{cases}
\end{align*}
solves the continuity equation $\frac{d}{d t}\mu'_t+\diver(W_t\mu'_t)=0$ as well. By uniqueness, see Remark \ref{continuity equation existence}, $(G_{T})_{*}((F_T)_{*}(\mu)) = \mu$. Since this is true for any $\mu$ with compact support and bounded density, we conclude $G_{T}(F_T(x)) = x$ for $m$-a.e. $x \in X$. 

By the same argument for each $t$ in the countable set $\mathbb{Q} \cap [0,T]$, we have for $m$-a.e. $x \in X$ and any $t \in \mathbb{Q} \cap [0,T]$, 
\begin{equation*}
	G_{t}(F_T(x)) = F_{T-t}(x).
\end{equation*}

The proposition follows by continuity of $G_t(x)$ and $F_t(x)$ in $t$ for all $x \in X$. 
\end{proof}

We recall the following result fom \cite{G13} which in particular implies $W_2$-geodesics with uniformly bounded densities are solutions of continuity equations.

\begin{thm}\label{W2 geodesic cont equation}(\cite[Theorem 1.1]{GT18})
Let $\mu_t$ be a $W_2$-geodesic with compact support and $\mu_t \leq Cm$ for every $t \in [0,1]$ and some $C>0$. If $f \in W^{1,2}(X)$ then the map $[0,1]\ni t \mapsto \int f d\mu_t$ is $C^1([0,1])$ and
\begin{equation*}
	\frac{d}{dt} \int f d\mu_t = \int \inner{\nabla f}{\nabla \phi_t} \, d\mu_t,
\end{equation*}
where $\phi_t$ is any function such that for some $s \neq t$, $s \in [0,1]$, the function $-(s-t)\phi_t$ is a Kantorovich potential from $\mu_t$ to $\mu_s$.
\end{thm}

The corollary below then follows by making the same type of arguments as in \cite[Section 7]{AT14}.

\begin{cor}\label{W2 geodesic cont equation cor}
Let $p \in X$ and $f \in W^{1,2}(X)$ fixing a representatitive. For $m$-a.e. $x \in X$, the map $t \mapsto f(\gamma_{x,p}(t))$ is in $W_{loc}^{1,1}([0,d_{x,p}))$ and 
\begin{equation*}
	\frac{d}{dt} f(\gamma_{x,p}(t)) = -df(\nabla d_p)(\gamma_{x,p}(t)) \; \; \text{for a.e. } t \in [0,d_{x,p}),
\end{equation*}
where $\gamma_{x,p}$ is a unit speed geodesic from $x$ to $p$. 
\end{cor}

\begin{proof}
By Remark \ref{a.e. unique geodesic}, we take a Borel selection of $\gamma_{x,p}$ which is unique for $m$-a.e. $x$. For each $x \in X$, let $\tilde{\gamma}_{x,p}:[0,1] \to X$ be the constant speed reparameterization of $\gamma$.

First consider a Lipschitz representative of $f \in \TestF(X)$. Clearly $f(\tilde{\gamma}_{x,p}(t))$ is continuous on $t \in [0,1]$ for each $x$. We show
\begin{enumerate}
	\item for $m$-a.e. $x$,  $\frac{-d(\tilde{\gamma}_{x,p}(t),p)}{1-t}\inner{\nabla f}{\nabla d_p}(\tilde{\gamma}_{x,p}(t)) \in L^1_{loc}([0,1));$
	\item for $m$-a.e. $x$, $f(\tilde{\gamma}_{x,p}(b))-f(\tilde{\gamma}_{x,p}(a))= \int_{a}^{b}  \frac{-d(\tilde{\gamma}_{x,p}(t),p)}{1-t}\inner{\nabla f}{\nabla d_p}(\tilde{\gamma}_{x,p}(t))\, dt$ for any $0 \leq a < b < 1$.
\end{enumerate}

For any $\mu$ with compact support and bounded density with respect to $m$, define $\mu_t := (\tilde{\gamma}_{\cdot,p}(t))_{*}(\mu)$ for $t \in [0,1]$. $(\mu_t)_{t \in [0,1]}$ is a $W_2$-geodesic. By \ref{directional BG}, for any $\delta > 0$, $(\mu_t)_{t \in [0, 1-\delta]}$ is of uniformly bounded density. By Theorem \ref{W2 geodesic cont equation}, the map $[0,1-\delta] \ni t \mapsto \int f d\mu_t$ is in  $C^1([0,1-\delta])$ and 
\begin{equation}\label{W2 geodeisc cont equation cor eq1}
	\frac{d}{dt} \int f d\mu_t = \int \frac{-d(x,p)}{1-t}\inner{\nabla f}{\nabla d_p}(x) d\mu_t(x). 
\end{equation}

Fix a representative of $\inner{\nabla f}{\nabla d_p} \in L^2(m)$.

Proof of 1:   The map $t \mapsto \int \frac{-d(x,p)}{1-t}\inner{\nabla f}{\nabla d_p}(x) \, d\mu_t(x)$ is in $L^1_{loc}([0,1))$ since $\inner{\nabla f}{\nabla d_p}(x) \in L^1_{loc}(m)$ and $\mu_t$ has uniformly bounded support and locally unifomly bounded density on $[0,1)$. By Fubini's theorem, this implies for $\mu$-a.e. $x$, $\frac{-d(\tilde{\gamma}_{x,p}(t),p)}{1-t}\inner{\nabla f}{\nabla d_p}(\tilde{\gamma}_{x,p}(t)) \in L^1_{loc}([0,1))$. Since this is true starting at any measure $\mu$ with compact support and bounded density with respect to $m$, statement 1 follows.

Proof of 2: By continuity of $f(\tilde{\gamma}_{x,p}(t))$ in $t$, it is enough to show that, for $m$-a.e. $x$, 
\begin{equation*}
	f(\tilde{\gamma}_{x,p}(\frac{k}{n}))-f(\tilde{\gamma}_{x,p}(\frac{k-1}{n}))= \int_{\frac{k-1}{n}}^{\frac{k}{n}}  \frac{-d(\tilde{\gamma}_{x,p}(t),p)}{1-t}\inner{\nabla f}{\nabla d_p}(\tilde{\gamma}_{x,p}(t))\, dt
\end{equation*}
for any $n \in \mathbb{N}$ and $1 \leq k \leq n-1$. Assume this is not the case, then there exists some $n$, $k$ and a bound set $S$ with $0 < m(S) < \infty$ so that for each $x \in S$, without loss of generality, $f(\tilde{\gamma}_{x,p}(\frac{k}{n}))-f(\tilde{\gamma}_{x,p}(\frac{k-1}{n}))> \int_{\frac{k-1}{n}}^{\frac{k}{n}}  \frac{-d(\tilde{\gamma}_{x,p}(t),p)}{1-t}\inner{\nabla f}{\nabla d_p}(\tilde{\gamma}_{x,p}(t))\, dt$. Applying \eqref{W2 geodeisc cont equation cor eq1} to a part of the Wasserstein geodesic from the normalization of $m\rvert_{S}$ to $\delta_p$ gives a contradiction. 

The general case of $f \in W^{1,2}(X)$ then follows by an approximation argument. Choose a sequence $f_i \in \TestF(X)$ converging to $f$ in $W^{1,2}(X)$. Using a diagonalization argument with Borel-Cantelli lemma and Fubini's theorem, there exists some subsequence $f_i$ so that for $m$-a.e. $x \in X$, $f_i(\tilde{\gamma}_{x,p}(t)) \to f(\tilde{\gamma}_{x,p}(t))$ in $L^1_{loc}([0,1))$ as functions of $t$. 

For any $\mu$ with compact support and bounded density, and $\mu_t$ defined as before, we also have
\begin{equation*}
	\int \frac{-d(x,p)}{1-t}\inner{\nabla f_i}{\nabla d_p}(x) \, d\mu_t(x) \to \int \frac{-d(x,p)}{1-t}\inner{\nabla f}{\nabla d_p}(x)\, d\mu_t(x) \; \; \text{ in } \; L^1_{loc}([0,1)).
\end{equation*}
Another diagonalization argument with Borel-Cantelli lemma and Fubini's theorem gives a further subsequence $f_i$ so that for $m$-a.e. $x \in X$,  
\begin{equation*}
	\frac{-d(\tilde{\gamma}_{x,p}(t),p)}{1-t}\inner{\nabla f_i}{\nabla d_p}(\tilde{\gamma}_{x,p}(t)) \to\frac{-d(\tilde{\gamma}_{x,p}(t),p)}{1-t}\inner{\nabla f}{\nabla d_p}(\tilde{\gamma}_{x,p}(t)) \; \; \text{ in } \; L^1_{loc}([0,1)).
\end{equation*}
Combining these with statements $1$ and $2$, we have for any $f  \in W^{1,2}(X)$, for $m$-a.e. $x \in X$, the map $t \mapsto f(\tilde{\gamma}_{x,p}(t))$ is in $W_{loc}^{1,1}([0,1))$ and 
\begin{equation*}
	\frac{d}{dt} f(\tilde{\gamma}_{x,p}(t)) = \frac{-d(\tilde{\gamma}_{x,p}(t),p)}{1-t}df(\nabla d_p)(\tilde{\gamma}_{x,p}(t)) \; \; \text{for a.e. } \; t \in [0,1).
\end{equation*}

The corollary then follows by a reparameterization of $\tilde{\gamma}_{x,p}$.
\end{proof}

We will be particularly interested in the following type of object: Let $p \in X$ and $\mu \in \mathcal{P}_2(X)$ be of bounded density with respect to $m$. Take a Borel selection (\ref{a.e. unique geodesic}) of unit speed geodesics $\gamma_{x, p}$ from all $x \in X$ to $p$ and define $T := d(\supp(\mu), p)$. For $0 \leq t \leq T$, define $\mu_t := (\gamma_{\cdot, p}(t))_{*}(\mu)$. $(\mu_t)$ defined this way are more naturally considered $L^1$-Wasserstein geodesics and are well-studied in the theory of needle decomposition of $\RCD$ spaces, see \cite{BC13, C14, CM17b}. We record some properties of these objects which will be needed later.

\begin{thm}\label{-dp grad flow properties}
Let $0 < \delta < T$ and $\mu \leq Am$ for some $A > 0$. Let $(\mu_t)_{t \in [0,T-\delta]}$ be as defined in the previous paragraph and $D := \sup_{x \in \supp(\mu)} d(x,p) \leq \bar{D}$. Then
	\begin{enumerate}
		\item $(\mu_t)_{t \in [0,T-\delta]}$ is a $W_2$-geodesic;
		\item There exists $C(K,N,\bar{D},\delta)$ so that $\mu_t \leq A(1+Ct)^Nm$ for all $t \in [0,T-\delta]$. In particular, the densities of $(\mu_t)_{t \in [0,T-\delta]}$ are uniformly bounded with respect to $m$;
		\item $(\mu_t)_{t \in [0,T-\delta]}$ solves the continuity equation 
			\begin{equation*}
				\frac{d}{d t}\mu_t+\diver(-\nabla d_p \mu_t)=0.
			\end{equation*}
	\end{enumerate}
\end{thm}

\begin{proof}
Statement 1 was proved in \cite[Lemma 4.4]{C14}, where $d$-monotonicity of $\left\{(x,\gamma_{x,p}(T-\delta)): x \in \supp{\mu} \right\}$ is used to show its $d^2$-montonicity.
Statement 2 was proved in \cite[Section 9]{BC13}, see also\cite[Section 3.2]{CM18} for a discussion. 
Statements 1 and 2 give that $(\mu_t)$ is of bounded compression.  Statement 3 then follows from Corollary \ref{W2 geodesic cont equation cor}. 
\end{proof}

In order to have terminology which includes the globally defined RLFs as well as the type of locally defined flows such as the example above, we will use the following definition.
\begin{de}\label{local flow def}
Let $S \in \mathscr{B}(X)$ with $m(S) > 0$ and $(V_t)_{t \in [0,T]} \in L^1([0,T],L^2(TX))$. A Borel map $F : [0,T] \times S \to X$ is a \textit{local flow of $V_t$ from $S$} if the following holds:
\begin{enumerate}
	\item $F_0(x) = x$ and $[0,T] \ni t \mapsto F_t(x)$ is continuous for every $x \in S$;
	\item For every $f \in \TestF(X)$ and $m$-a.e. $x \in S$, $t \mapsto f(F_t(x))$ is in $W^{1,1}([0,T])$ and 
		\begin{equation}\label{local flow eq1} 
			\frac{d}{dt}f(F_t(x))= df(V_t)(F_t(x)) \text{ for a.e. } t \in [0,T];
		\end{equation}
	\item There exists a constant $C := C(V,S)$ so that $(F_t)_{*}(m \rvert_S) \leq Cm$ for all $t$ in $[0,T]$. 
\end{enumerate}
\end{de}
As before, by abuse of notation we will often say $(F_t)$ is the local flow of $(V_t)$ from $S$ for some vector field $V$ which is only locally $L^2$. In this case, it is understood that $F_t(S)$ is essentially bounded for each $t$ and we cut off $V_t$ outside of this region. 

\begin{rem}\label{local flow examples}
We will be primarily intersted in the following examples:
\begin{enumerate}
	\item For any $p \in X$ and bounded $S \in \mathscr{B}(X)$, $F_t(x) := \gamma_{x,p}(t)$ defined on $(t,x) \in  [0,T-\delta] \times S$, where $T := \essinf_{x \in S}d(x,p)$, $\delta > 0$, and $\gamma_{x,p}$ is a unit speed geodesic from $x$ to $p$ is Borel selected (see \ref{a.e. unique geodesic}), is a local flow of $-\nabla d_p$ from $S$ by Corollary \ref{W2 geodesic cont equation cor} and Theorem \ref{-dp grad flow properties}. 
	\item The restriction of any RLF onto some $S \in \mathscr{B}(X)$ is a local flow of the corresponding $(V_t)$ from $S$ by definition.
\end{enumerate}
\end{rem}

The following differentiation formula follows by the same argument for RLFs in Proposition \ref{first order differentiation}.

\begin{prop}\label{first order differentiation flow} (First order differentiation formula for distance along local flows)
Let $T > 0$. If $(F_t)_{t \in [0,T]},$ $(G_t)_{t \in [0,T]}$ are local flows of $(U_t), (V_t)$ from $S_1$ and $S_2$ respectively, then for $(m \times m)$-a.e. $(x,y) \in S_1 \times S_2$, $d(F_t(x),G_t(y)) \in W^{1,1}([0,T])$ and 
\begin{equation*}
	\frac{d}{dt} d(F_t(x),G_t(y))= \inner{\nabla d_{G_t(y)}}{U_t}(F_t(x)) + \inner{\nabla d_{F_t(x)}}{V_t}(G_t(y)) \; \; \text{ for a.e. $t \in [0,T].$}
\end{equation*}
\end{prop}

We mention that if $(W_t)$ is defined as in proposition \ref{first order differentiation} from $(U_t)$ and $(V_t)$, then it is straightforward to check using the arguments of \cite{BS20} that $(F_t, G_t)$ is a local flow of $(W_t)$ from $S_1 \times S_2$. Again, $(W_t)$ here naturally belongs in $L^1([0,T],L^2_{loc}(T(X \times X)))$ so Definition \ref{local flow def} needs to be altered to allow for this. We refer to \cite{GR18} for relevant definitions. 

The next proposition gives control on the metric speeds of the curves $t \mapsto F_t(x)$ of a local flow $F$. As pointed out in \cite[(A.22)]{GT18}, it follows from a similar argument as in \cite[Theorem 2.3.18]{G18} after a small adjustment since we do not a priori assume the absolute continuity of the curves $F_\cdot(x)$.
\begin{prop}\label{RLF speed bound}
Let $V \in L^1([0,T],L^2(TX))$  and let $(F_t)$ be a local flow of $(V_t)$ from $S$. For $m$-a.e. $x \in S$ the curve $t \mapsto F_t(x)$ is absolutely continuous and its metric speed $ms_t(F_{\cdot}(x))$ at time $t$ satisfies
\begin{equation*}
	ms_t(F_{\cdot}(x))=|V_t|(F_t(x)) \; \; \text{ for a.e. } t \in [0,T].
\end{equation*}
\end{prop}

\subsection{Second order interpolation formula}\label{subsection 6.3} The proof of a second order interpolation formula for the distance function (see \eqref{Hessian interpolation} and \eqref{CN interpolation}) along flows requires the results of \cite{GT18}. The hard work is done there and their result immediately implies an analogous second order interpolation formula (Theorem \ref{Wasserstein second order}) for the Wasserstein distance. It is our goal to pass this fromula from Wasserstein distance to distance on the space itself.

For the rest of the subsection we will always be in the setting of some $\RCD(K,N)$ space $(X,d,m)$ with $K \in \mathbb{R}$ and $N \in [1,\infty)$. We fix a Borel selection (\ref{a.e. unique geodesic}) of constant speed geodesics $\tilde{\gamma}_{x,y}$ from all $x\in X$ to all $y\in X$ parameterized on the unit interval. We denote by $\gamma_{x,y}$ the unit speed reparameterization of $\tgam_{x,y}$ to $[0,d(x,y)]$. We start with the following formulation of the main result from \cite{GT18}.

\begin{thm}\label{Wasserstein second order}(\cite[Theorem 5.13]{GT18})
Let $\mu_0, \mu_1 \in \mathcal{P}_2(X)$ be compactly supported and satisfy $\mu_0, \mu_1 \leq Cm$ for some $C > 0$. Let $(\mu_t)$ be the unique $W_2$-geodesic connecting $\mu_0$ to $\mu_1$. For every $t \in [0,1]$, let $\phi_t$ be any function so that for some $s \neq t$, $s \in [0,1]$, the function $-(s-t)\phi_t$ is a Kantorovich potential from $\mu_t$ to $\mu_s$. For any $V \in H^{1,2}_{C}(TX)$, the map $[0,1] \ni t \mapsto \int \inner{V}{\nabla \phi_t} d\mu_t$ is in $C^1([0,1])$ and
\begin{equation}\label{Wasserstein second order eq1}
	\frac{d}{dt}  \bigg(\int \inner{V}{\nabla \phi_t} d\mu_t\bigg) = \int \big(\nabla V : (\nabla \phi_t \otimes \nabla \phi_t)\big) \, d\mu_t \; \; \text{  for all $t \in [0,1]$}.
\end{equation}
\end{thm}

The next lemma follows from the previous theorem.

\begin{lem}\label{second order lemma}
Let $p \in X$ and $\nu \leq Cm$ be a nonnegative, compactly supported measure. For any $V \in H^{1,2}_C(TX)$,
\begin{align}\label{second order lemma eq1}
\begin{split}
	\int \inner{V}{\nabla d_p}(x) \, d\nu(x) \, - \int \inner{V}{\nabla d_p}&(\tgam_{p,x}(\frac{1}{2})) \, d\nu(x) =\\ &\int_{\frac{1}{2}}^{1}\bigg( \int d(p,x) \big(\nabla V :(\nabla d_p \otimes \nabla d_p)\big)(\tgam_{p,x}(t)) \, d \nu(x) \bigg) \, dt,
\end{split}
\end{align}
where $\tgam_{x,y}$ is as defined in the beginning of this subsection.
\end{lem}

Note that although $\nabla d_p \otimes \nabla d_p$ is not in $L^2(T^{\otimes 2}(X))$, it is locally (i.e. it is after multiplication by the characteristic function of any compact set). Therefore, by the locality properties of the objects involved, $\nabla V : (\nabla d_p \otimes \nabla d_p)$ is well-defined and is in $L^2(m)$.

\begin{proof}
Fix reprsentatives for $\inner{V}{\nabla d_p}$ and $\nabla V:(\nabla d_p \otimes \nabla d_p) \in L^2(m)$. 

We claim for $m$-a.e. $x \in X$, 
\begin{equation}\label{second order lemma eq2}
	\inner{V}{\nabla d_p}(x)-\inner{V}{\nabla d_p}(\tgam_{p,x}(\frac{1}{2})) = \int_{\frac{1}{2}}^{1}d(p,x)(\nabla V :(\nabla d_p \otimes \nabla d_p)\big)(\tgam_{p,x}(t)) \, dt.
\end{equation}
The right side integral is finite for $m$-a.e. $x$ by using a Fubini's theorem argument along with Theorem \ref{directional BG} \eqref{directional BG eq1}.

Suppose \eqref{second order lemma eq2} does not hold for $m$-a.e. $x \in X$, then without loss of generality we may assume there exists a bounded set $S$ with $0<m(S)<\infty$ so that 
\begin{equation}\label{second order lemma eq3}
	\inner{V}{\nabla d_p}(x)-\inner{V}{\nabla d_p}(\tgam_{p,x}(\frac{1}{2})) > \int_{\frac{1}{2}}^{1}d(p,x)(\nabla V :(\nabla d_p \otimes \nabla d_p)\big)(\tgam_{p,x}(t)) \, dt
\end{equation}
for each $x \in S$. 
Let $\mu := \frac{1}{m(S)}m\rvert_{S}$. Multiplying both sides of \eqref{second order lemma eq3} by $d(p,x)$ and integrating with respect to $\mu$, we immediately contradict Theorem \ref{Wasserstein second order}. Therefore, \eqref{second order lemma eq2} holds and so the lemma follows for any $\nu$ with compact support and bounded density. 
\end{proof} 

To proceed we state the segment inequality for $L^1(m)$ functions, first introduced by Cheeger-Colding in \cite[Theorem 2.11]{CC96}. This has been established for the metric measure setting in \cite{R08} but we will give a self-contained proof since the decomposition procedure for the family of geodesics used in the proof will be used again in the proof of Proposition \ref{second order interpolation}.

\begin{thm}\label{segment inequality}(Segment inequality for $L^1$ functions on RCD spaces)
Let $(X,d,m)$ be an $\RCD(K,N)$ space with $K \in \mathbb{R}$ and $N \in [1,\infty)$. Let $\bar{R} > 0$. Let $f \in L_{loc}^{1}(X)$ be nonnegative and $\mu \leq A(m \times m)$ be a nonnegative measure on $X \times X$ supported on $B_{R}(p) \times B_{R}(p)$ for some $0 < R \leq \bar{R}$ and $p \in X$. Then 
\begin{align}\label{segment inequality eq1}
\begin{split}
\int_{0}^{1} \bigg(\int f(\tgam_{x,y}(t))&d(x,y) \, d\mu(x,y)\bigg) \, dt \\
	&\leq AC(K,N,\bar{R})R[m(\pi_1(\supp(\mu)))+m(\pi_2(\supp(\mu)))] \int_{B_{2R}(p)} f(z) \, dm(z),
\end{split}
\end{align}
where $\pi_1, \pi_2$ are projections onto the first and the second coordinate respectively and $\tgam_{x,y}$ is as defined in the beginning of this subsection.
\end{thm}

\begin{proof}
By Radon-Nikodym theorem, $\mu = g (m \times m)$ for some compactly supported $g \in L^{\infty}(X \times X, m \times m)$. We denote $g^{1}_{x}(\cdot) := g(x, \cdot)$ and $g^{2}_{y}(\cdot) := g(\cdot, y)$. By Fubini's Theorem, $\norm{g^{1}_x}_{L^{\infty}} \leq A$ for $m$-a.e. $x$ and similarly, $\norm{g^{2}_y}_{L^{\infty}} \leq A$ for $m$-a.e. $y$.

For each $t \in [\frac{1}{2}, 1]$ and $m$-a.e. $x \in \pi_1(\supp(\mu))$, by Theorem \ref{directional BG} \eqref{directional BG eq1}, 
\begin{equation}\label{segment inequality eq3}
	(\tgam_{x,\cdot}(t))_{*}(g^{1}_{x} m) \leq AC(K,N,\bar{R})m\rvert_{B_{2R}(p)}.
\end{equation}
Similarly for $t \in [0,\frac{1}{2}]$ and $m$-a.e. $y \in \pi_2(\supp(\mu))$,
\begin{equation}\label{segment inequality eq5}
	(\tgam_{\cdot,y}(t))_{*}(g^{2}_{y} m) \leq AC(K,N,\bar{R})m\rvert_{B_{2R}(p)}.
\end{equation}
We conclude 
\begin{align*}
&\hspace{0.5cm} \int_{0}^{1} \bigg(\int f(\tgam_{x,y}(t))d(x,y) \, d\mu(x,y)\bigg) \, dt\\
&\leq 2R\int_{0}^{1} \bigg(\int f(\tgam_{x,y}(t)) \, d\mu(x,y)\bigg) \, dt\\
&= 2R\bigg(\int_{\frac{1}{2}}^{1} \bigg(\int f(\tgam_{x,y}(t)) \, d\mu(x,y)\bigg) \, dt +\int_{0}^{\frac{1}{2}} \bigg(\int f(\tgam_{x,y}(t)) \, d\mu(x,y)\bigg) \, dt \bigg)  \\
&\leq AC(K,N,\bar{R})R[m(\pi_1(\supp(\mu)))+m(\pi_2(\supp(\mu)))] \int_{B_{2R}(p)} f(z) \, dm(z),
\end{align*}
where \eqref{segment inequality eq3}, \eqref{segment inequality eq5} and Fubini's theorem was used in the last line. 

\end{proof}

We now prove the main interpolation formula of this subsection. The main idea is to use Lemma \ref{second order lemma} along with the decomposition procedure from the proof of the semgent inequality. Notice that the right side of \eqref{second order interpolation eq1} makes sense due to the segment inequality.
\begin{prop}\label{second order interpolation}(Second order interpolation formula)
Let $\mu \leq C (m \times m)$ be a nonnegative and compactly supported measure on $X \times X$. Let $V \in H^{1,2}_C(TX)$. Then
\begin{equation}\label{second order interpolation eq1}
	\int \big(\inner{V}{\nabla d_x}(y) + \inner{V}{\nabla d_y}(x) \big)\, d\mu(x,y)= \int_{0}^{1} \int d(x,y) (\nabla V : \nabla d_x \otimes \nabla d_x)(\tgam_{x,y}(t))\, d\mu(x,y) \, dt,
\end{equation}
where $\tgam_{x,y}$ is as defined in the beginning of this subsection.
\end{prop}
\begin{proof}
Let $\mu = g(m \times m)$ for compactly supported $g \in L^{\infty}(X \times X, m \times m)$ with $g^{1}_{x}(\cdot) := g(x, \cdot)$ and $g^{2}_{y}(\cdot) := g(\cdot, y)$. $\norm{g^{1}_x}_{L^{\infty}} \leq C$ for $m$-a.e. $x$ and $\norm{g^{2}_y}_{L^{\infty}} \leq C$ for $m$-a.e. $y$. We have 
\begin{align}\label{second order interpolation eq2}
\begin{split}
	&\hspace{0.5cm} \int \inner{V}{\nabla d_x}(y) \, d\mu(x,y) - \int \inner{V}{\nabla d_x}(\tgam_{x,y}(\frac{1}{2})) \, d\mu(x,y)\\
	&= \int \int \bigg( \inner{V}{\nabla d_x}(y) - \inner{V}{\nabla d_x}(\tgam_{x,y}(\frac{1}{2})) \bigg) \, d(g^{1}_{x}m)(y) \, dm(x) \; \; \text{. by Fubini's theorem}\\
	&= \int \int_{\frac{1}{2}}^{1}\bigg( \int d(x,y) \big(\nabla V :(\nabla d_x \otimes \nabla d_x)\big)(\tgam_{x,y}(t)) \, d(g^{1}_{x}m)(y) \bigg) \, dt \, dm(x) \; \; \text{ , by \ref{second order lemma}}\\
	&= \int_{\frac{1}{2}}^{1} \bigg( \int \bigg( d(x,y)\big(\nabla V :(\nabla d_x \otimes \nabla d_x)\big)(\tgam_{x,y}(t)) \bigg) d\mu(x,y)\, \bigg) \,dt \; \; \text{, by Fubini's theorem.}
\end{split}
\end{align}
Making the analogous argument on $t \in [0,\frac{1}{2}]$, we obtain
\begin{align}\label{second order interpolation eq3}
\begin{split}
	&\hspace{0.5cm} \int \inner{V}{\nabla d_y}(x) \, d\mu(x,y) - \int \inner{V}{\nabla d_y}(\tgam_{x,y}(\frac{1}{2})) \, d\mu(x,y)\\
	&=\int_{0}^{\frac{1}{2}} \bigg( \int \bigg( d(x,y)\big(\nabla V :(\nabla d_y \otimes \nabla d_y)\big)(\tgam_{x,y}(t)) \bigg) d\mu(x,y)\, \bigg) \, dt.
\end{split} 
\end{align}
The claim then follows if
\begin{equation}\label{second order interpolation eq4}
	\int \inner{V}{\nabla d_x}(\tgam_{x,y}(\frac{1}{2})) \, d\mu(x,y) + \int \inner{V}{\nabla d_y}(\tgam_{x,y}(\frac{1}{2})) \, d\mu(x,y) = 0
\end{equation}
and
\begin{align}\label{second order interpolation eq5}
\begin{split}
	&\hspace{0.5cm} \int_{0}^{\frac{1}{2}} \bigg( \int \bigg( d(x,y)\big(\nabla V :(\nabla d_y \otimes \nabla d_y)\big)(\tgam_{x,y}(t)) \bigg) d\mu(x,y)\, \bigg) \, dt\\
	&= \int_{0}^{\frac{1}{2}} \bigg( \int \bigg( d(x,y)\big(\nabla V :(\nabla d_x \otimes \nabla d_x)\big)(\tgam_{x,y}(t)) \bigg) d\mu(x,y)\, \bigg) \, dt,\\
\end{split}
\end{align}
which we show in Lemma \ref{grad d_p = -grad d_q} and Remark \ref{grad d_p = -grad d_q cor}.
\end{proof}

The following two lemmas show that, in a measure-theoretic sense, $\nabla d_p$ and $\nabla d_q$ in the interior of a geodesic between $p$ and $q$ ``point in opposite directions".

\begin{lem}\label{Lipschitzconstantinteriorpoint} 
	Let $p,q \in X$. Let $\gamma_{p,q}:[0,1] \to X$ be a constant speed geodesic from $p$ to $q$ and let $z = \gamma(t_0)$ for some $t_0 \in (0,1)$. Let $f$ be a locally Lipschitz function, then $\lip(f+d_p)(z) = \lip(f-d_q)(z)$.
\end{lem}
\begin{proof}
By classical Abresch-Gromoll inequality \ref{classical Abresch-Gromoll}, for $x$ in a sufficiently small neighbourhood of $z$, $e_{p,q}(x) \leq Cd(x,z)^{1+\alpha}$. Therefore, for any such $x$, 
\begin{align*}
&\hspace{0.5cm}\frac{|(f(x)+d_p(x))-(f(z)+d_p(z))|}{d(x,z)}-\frac{|(f(x)-d_q(x))-(f(z)-d_q(z))|}{d(x,z)}\\  
&\leq \frac{Cd(x,z)^{1+\alpha}}{d(x,z)}=Cd(x,z)^{\alpha}.
\end{align*}
This shows that $\lip(f+d_p)(z) = \lip(f-d_q)(z)$ since $Cd(x,z)^{\alpha} \to 0$ as $d(x,z) \to 0$.
\end{proof}

\begin{lem}\label{grad d_p = -grad d_q}
In the notations of \ref{second order interpolation}, for any $t \in (0,1)$,
\begin{equation*}
	\int \inner{V}{\nabla d_x}(\tgam_{x,y}(t)) \, d\mu(x,y) + \int \inner{V}{\nabla d_y}(\tgam_{x,y}(t)) \, d\mu(x,y) = 0.
\end{equation*}
\end{lem}

\begin{proof}
Fix $t \in (0,1)$. Let $\mu = g(m \times m)$ for compactly supported $g \in L^{\infty}(X \times X, m \times m)$ with $g^{1}_{x}(\cdot) := g(x, \cdot)$ and $g^{2}_{y}(\cdot) := g(\cdot, y)$. $\norm{g^{1}_x}_{L^{\infty}} \leq C$ for $m$-a.e. $x$ and $\norm{g^{2}_y}_{L^{\infty}} \leq C$ for $m$-a.e. $y$.

We first prove the claim for $V = \nabla f$ where $f$ is locally Lipschitz. The general claim then follows by approximation. We observe the following:
\begin{enumerate}
	\item By Remark \ref{lip=grad}, for each $x \in X$,
	 \begin{align*}
	 \langle \nabla f, \nabla d_x \rangle &= \frac{|\nabla (f+d_x)|^2 - |\nabla f|^2 - |\nabla d_x|^2}{2}\\
					&= \frac{(\lip(f+d_x))^2 - (\lip (f))^2- (\lip (d_x))^2}{2} \; \; m-\text{a.e..}
	\end{align*}
\item Similarly, for each $y \in X$,
	\begin{align*}
		 \langle \nabla f, -\nabla d_y \rangle &= \frac{|\nabla (f-d_y)|^2 - |\nabla f|^2 - |-\nabla d_y|^2}{2}\\
					&= \frac{(\lip (f-d_y))^2 - (\lip (f))^2- (\lip (-d_y))^2}{2} \; \; m-\text{a.e..}
	\end{align*}
\item $(\tgam_{x,\cdot}(t))_{*}(g^{1}_{x}m)$ is of bounded density w.r.t. $m$ for $m$-a.e. $x$ by Theorem \ref{directional BG} \eqref{directional BG eq1}.
\end{enumerate}
Lemma \ref{Lipschitzconstantinteriorpoint} and Fubini's Theorem then gives the claim for $\nabla f$,
\begin{align*}
	&\hspace{0.5cm}  \int \inner{\nabla f}{\nabla d_x}(\tgam_{x,y}(t)) \, d\mu(x,y) \\
	&= \int \bigg(\int \inner{\nabla f}{\nabla d_x}(\tgam_{x,y}(t)) \, d(g^{1}_{x}m)(y) \, \bigg) \, dm(x)\\
	&= \int \bigg(\int \frac{(\lip(f+d_x))^2 - (\lip (f))^2- (\lip (d_x))^2}{2}(\tgam_{x,y}(t)) \, d(g^{1}_{x}m)(y) \, \bigg) \, dm(x) \text{ , by observations 1 and 3}\\
	&= \int \bigg(\int \frac{(\lip(f-d_y))^2 - (\lip (f))^2- (\lip (-d_y))^2}{2}(\tgam_{x,y}(t)) \, d(g^{1}_{x}m)(y) \, \bigg) \, dm(x) \text{ , by \ref{Lipschitzconstantinteriorpoint} and observ. 3}\\
	&=  \int \inner{\nabla f}{-\nabla d_y}(\tgam_{x,y}(t)) \, d\mu(x,y) \text{ , by observations 2 and 3.}
\end{align*}
\end{proof}

\begin{rem}\label{grad d_p = -grad d_q cor}
\ref{grad d_p = -grad d_q} implies \eqref{second order interpolation eq5} as well. This is because Lemma \ref{second order lemma} is true for any interval $[a,b] \subseteq (0,1]$ and so one can use the same argument as in Proposition \ref{second order interpolation} to say that, for $0 < s < \frac{1}{2}$, 
\begin{align*}
\begin{split}
	&\hspace{0.5cm} \int_{s}^{\frac{1}{2}} \bigg( \int \bigg( d(x,y)\big(\nabla V :(\nabla d_y \otimes \nabla d_y)\big)(\tgam_{x,y}(t)) \bigg) d\mu(x,y)\, \bigg) dt\\
	&= \int \inner{V}{\nabla d_y}(\tgam_{x,y}(\frac{1}{2})) \, d\mu(x,y) - \int \inner{V}{\nabla d_y}(\tgam_{x,y}(s)) \, d\mu(x,y)\\
	&= -\int \inner{V}{\nabla d_x}(\tgam_{x,y}(\frac{1}{2})) \, d\mu(x,y) + \int \inner{V}{\nabla d_y}(\tgam_{x,y}(s)) \, d\mu(x,y)\\
	&= \int_{s}^{\frac{1}{2}} \bigg( \int \bigg( d(x,y)\big(\nabla V :(\nabla d_x \otimes \nabla d_x)\big)(\tgam_{x,y}(t)) \bigg) d\mu(x,y)\, \bigg) dt.
\end{split}
\end{align*}
Taking a limit $s \to 0$ gives \eqref{second order interpolation eq5}.
\end{rem}

The second order interpolation formula \ref{second order interpolation} and first order differentiation formula for distance along local flows \ref{first order differentiation flow} immediately give an integral version of \eqref{CN interpolation}. They also give the following related estimate which will be used heavily in Section \ref{section 5}. Let $U$, $V \in L^{1}([0,T], L^{2}(TX))$ be bounded (see Definition \ref{L^2 vector fields}) and $S_1, S_2$ be bounded sets of positive measure. Let $(F_t)_{t \in [0,T]}$, $(G_t)_{t \in [0,T]}$ be local flows of $U$, $V$ from $S_1$, $S_2$ respectively. Let $r > 0$ and for each $t \in [0,T]$, define $dt^{F,G}_{r}(t): S_1 \times S_2 \to [0,r]$ the \textit{distance distortion on scale $r$ at $t$} by
\begin{equation}\label{distortion def}
	dt^{F,G}_r(t)(x,y) := \min\braces{r, \max\limits_{0 \leq \tau \leq t}|d(x,y)-d(F_\tau(x), G_\tau(y))|}.
\end{equation} 
Define $\Gamma^{F,G}_r(t) := \braces{(x,y) \in S_1 \times S_2 \, : \, dt^{F,G}_r(t)(x,y) < r}.$
The terminology and definition of the distance distortion function comes from \cite{KW11}, where it was used in a similar way as in this paper to analyze the geometry of gradient flows. 
\begin{prop}\label{distortion bound}
Let $W \in H^{1,2}_{C}(TX)$. The map 
$t \mapsto \int_{S_1 \times S_2} dt^{F,G}_r(t)(x,y) \, d(m \times m)(x,y)$ is Lipschitz on $[0,T]$ and satisfies
\begin{align*}
	&\hspace{0.5cm}\frac{d}{dt} \int\limits_{S_1 \times S_2} dt^{F,G}_r(t)(x,y) \, d(m \times m)(x,y) \\
	&\leq \int_{\Gamma^{F,G}_r(t)} \big(|U_t - W|(F_t(x)) + |V_t - W|(G_t(y)) \big)\, d(m \times m)(x,y)\\
	&\hspace{2cm}+ \int_{0}^{1} \int_{\Gamma^{F,G}_r(t)} d(F_t(x),G_t(y)) |\nabla W|_{\HS}(\tgam_{F_t(x),G_t(y)}(s)) \, d(m \times m)(x,y) \, ds
\end{align*}
for a.e. $t \in [0,T]$, where $\tgam_{\cdot,\cdot}$ is as defined in the beginning of this subsection.
\end{prop}

\begin{proof}
First fix representatives for all involved measure-theoretic objects. For any $(x, y) \in S_1 \times S_2$, $dt^{F,G}_r(t)(x,y)$ is continuous, monontone non-decreasing and bounded between $0$ and $r$ as a function of $t \in [0,T]$. Therefore, $dt^{F,G}_r(t)(x,y)$ is differentiable for a.e. $t \in [0,T]$ and $\frac{d^+}{dt}dt^{F,G}_r(t)(x,y) = 0$ for all $t$ where $(x,y) \notin \Gamma^{F,G}_r(t)$. 

Furthermore, by boundedness of $U$, $V$ and Proposition \ref{RLF speed bound}, $F_{\cdot}(x)$, $G_{\cdot}(y)$ are uniformly Lipschitz curves for $m$-a.e. $x \in S_1$ and $m$-a.e. $y \in S_2$ respectively. In particular, the functions $[0,T] \ni t \mapsto d(F_t(x), G_t(y))$ are uniformly Lipschitz for $(m \times m)$-a.e. $(x,y) \in S_1 \times S_2$.  The same is true for the functions $t \mapsto dt^{F,G}_r(t)(x,y)$. Therefore, $t \mapsto \int_{S_1 \times S_2} dt^{F,G}_r(t)(x,y) \, d(m \times m)(x,y)$ is Lipschitz on $[0,T]$ as well. 

By Proposition \ref{first order differentiation flow}, for $(m \times m)$-a.e.$(x,y) \in S_1 \times S_2$, $d(F_t(x),G_t(y)) \in W^{1,1}([0,T])$ and 
\begin{equation}\label{distortion bound eq1}
	\frac{d}{dt} d(F_t(x), G_t(y)) = \inner{\nabla d_{G_t(y)}}{U_t}(F_t(x)) + \inner{\nabla d_{F_t(x)}}{V_t}(G_t(y)).
\end{equation}
At any point of differentiability for both $d(F_t(x), G_t(y))$ and $dt^{F,G}_r(t)(x,y)$, it is clear from definition that $\frac{d}{dt} dt^{F,G}_r(t)$ $(x,y) \leq (\frac{d}{dt} d(F_t(x), G_t(y)))_+$

For any $t \in [0,T]$, let $\tilde{\Gamma}^{F,G}_r(t) := \braces{(x,y)\, : \, \inner{\nabla d_{G_t(y)}}{U_t}(F_t(x)) + \inner{\nabla d_{F_t(x)}}{V_t}(G_t(y)) \geq 0} \cap \Gamma^{F,G}_r(t)$. Then $\mu_t := (F_t, G_t)_*((m \times m)\rvert_{\tilde{\Gamma}^{F,G}_r(t)})$ is compactly supported by Proposition \ref{RLF speed bound} and has bounded density with respect to $(m \times m)$ by definition \ref{local flow def} of a local flow. By the second order interpolation formula \ref{second order interpolation}, 
\begin{equation}\label{distortion bound eq2}
	\int \big(\inner{W}{\nabla d_y}(x) + \inner{W}{\nabla d_x}(y) \big)\, d\mu_t(x,y)= \int_{0}^{1} \int d(x,y) (\nabla W : \nabla d_x \otimes \nabla d_x)(\tgam_{x,y}(s))\, d\mu_t(x,y) \, ds.
\end{equation}

Therefore, 
\begin{align}\label{distortion bound eq3}
\begin{split}
	&\hspace{0.5cm}\int \big(\inner{U_t}{\nabla d_y}(x) + \inner{V_t}{\nabla d_x}(y) \big)\, d\mu_t(x,y)\\
	&= \int \big(\inner{U_t - W}{\nabla d_y}(x) + \inner{V_t - W}{\nabla d_x}(y) \big)\, d\mu_t(x,y)\\
	&\hspace{3cm}+ \int_{0}^{1} \int d(x,y) (\nabla W : \nabla d_x \otimes \nabla d_x)(\tgam_{x,y}(s))\, d\mu_t(x,y) \, ds  \; \; \text{ , by \eqref{distortion bound eq2}}\\
	& \leq \int \big(|U_t - W|(x) + |V_t - W|(y) \big)\, d\mu_t(x,y)\\
	&\hspace{3cm}+ \int_{0}^{1}\int d(x,y) |\nabla W|_{\HS}(\tgam_{x,y}(s)) \, d\mu_t(x,y) \, ds \; \; \text{ , since $|\nabla d_x|,|\nabla d_y| = 1$ $m$-a.e.}\\
	&\leq \int_{\Gamma^{F,G}_r(t)} \big(|U_t - W|(F_t(x)) + |V_t - W|(G_t(y)) \big)\, d(m \times m)(x,y)\\
	&\hspace{3cm}+ \int_{0}^{1}  \int_{\Gamma^{F,G}_r(t)} d(F_t(x),G_t(y))|\nabla W|_{\HS}(\tgam_{F_t(x),G_t(y)}(s)) \, d(m \times m)(x,y) \, ds,
\end{split}
\end{align}
by definition of $\mu_t = (F_t, G_t)_*((m \times m)\rvert_{\tilde{\Gamma}^{F,G}_r(t)})$, $\tilde{\Gamma}^{F,G}_r(t) \subseteq \Gamma^{F,G}_r(t)$ and the fact that all integrands are positive.

To conlcude, we have for a.e. $t \in [0,T]$, for $(m \times m)$-a.e. $(x,y) \in S_1 \times S_2$, $d(F_t(x), G_t(y))$ and $dt^{F,G}_r(t)(x,y)$ are both differentiabile in $t$. For any such $t$,
\begin{align*}
	&\hspace{0.5cm}\frac{d}{dt} \int\limits_{S_1 \times S_2} dt^{F,G}_r(t)(x,y) \, d(m \times m)(x,y) \\
	&=  \int\limits_{S_1 \times S_2} \frac{d}{dt} dt^{F,G}_r(t)(x,y) \, d(m \times m)(x,y) \; \; \text{ , by DCT since $dt_r$ is uniformly Lipschitz for a.e. $(x,y)$}\\
	&= \int\limits_{\Gamma^{F,G}_r(t)} \frac{d}{dt} dt^{F,G}_r(t)(x,y) \, d(m \times m)(x,y)\\
	&\leq \int\limits_{\Gamma^{F,G}_r(t)} (\frac{d}{dt} d(F_t(x), G_t(y)))_+ \, d(m \times m)(x,y)\\
	&= \int\limits_{\tilde{\Gamma}^{F,G}_r(t)} \inner{\nabla d_{G_t(y)}}{U_t}(F_t(x)) + \inner{\nabla d_{F_t(x)}}{V_t}(G_t(y))  \, d(m \times m)(x,y) \; \; \text{ , by \eqref{distortion bound eq1} and definition of $\tilde{\Gamma}$}\\
	&\leq \int_{\Gamma^{F,G}_r(t)} \big(|U_t - W|(F_t(x)) + |V_t - W|(G_t(y)) \big)\, d(m \times m)(x,y)\\
	&\hspace{2cm}+ \int_{0}^{1}  \int_{\Gamma^{F,G}_r(t)} d(F_t(x),G_t(y))|\nabla W|_{\HS}(\tgam_{F_t(x),G_t(y)}(s)) \, d(m \times m)(x,y) \, ds \; \; \text{ , by \eqref{distortion bound eq3}.}
\end{align*}
\end{proof}

\section{Estimates on the heat flow approximations of distance and excess functions}\label{section 4}
Here we collect estimates on the heat flow approximations of distance and excess functions, all of which were established in \cite{CN12}. All their arguments translate directly to the $\RCD$ setting due to the availability of the improved Bochner inequality, the Li-Yau Harnack and gradient inequalities, and the various estimates of Subsection \ref{subsection 2.10}. We record their proofs for the sake of completeness, making minor regularity and measure-theoretic adjustments as needed.

In this section we fix $(X,d,m)$ an $\RCD(-(N-1),N)$ space for $N \in (1,\infty)$, $0 < \delta < \frac{1}{2}$, and two points $p, q \in X$ with $d(p,q) \leq 1$. Any time we use $c$ it is always a constant depending only on $N$ and $\delta$ unless specified otherwise. We fix the following notations:
\begin{enumerate}
	\item $d_{p,q} = d(p,q) \leq 1$ and for any $\epsilon > 0$, $d_{\epsilon} := \epsilon d_{p,q}$.
	\item $d^-(x) := d(p,x)$.
	\item $d^+(x) := d(p,q)-d(x,q)$.
	\item $e(x) := d(p,x)+d(x,q)-d(p,q) = d^-(x) - d^+(x).$
\end{enumerate}

We will consider these functions multiplied by some appropriate cut off functions. Let $\psi^{\pm}: X \to \mathbb{R}$ be the good cut off functions as in \ref{good cut off function} satisfying
\begin{align}
\psi^{-} = \left\{ \begin{array}{rl}
 1&\mbox{ on }A_{\frac{\delta}{8} d_{p,q},8 d_{p,q}}(p) \notag\\
  0 &\mbox{ on }X\setminus A_{\frac{\delta}{16} d_{p,q},16 d_{p,q}}(p)
       \end{array} \right. ,\,
\psi^{+} = \left\{ \begin{array}{rl}
 1&\mbox{ on }A_{\frac{\delta}{8} d_{p,q},8 d_{p,q}}(q) \notag\\
  0 &\mbox{ on }X\setminus A_{\frac{\delta}{16} d_{p,q},16 d_{p,q}}(q)
       \end{array} \right. \, .\notag
\end{align}

Let $\psi := \psi^{+}\psi^{-}$, $e_0 := \psi e$ and $h^{\pm}_0 := \psi d^{\pm}$. We denote
\begin{enumerate}
	\setcounter{enumi}{4} 
	\item $h_t^{\pm} := H_{t}(h_0^{\pm})$ and $e_t := H_{t}(e_0)$. 
	\item $X_{r,s} := A_{rd_{p,q}, sd_{p,q}}(p) \cap A_{rd_{p,q}, sd_{p,q}}(q)$.
\end{enumerate}
By definition $e_0 = e$, $h^{\pm}_0 = h^{\pm}$ on $X_{\frac{\delta}{8},8}$ and $e_t = h^-_t - h^+_t$ by uniqueness of heat flow.

We will always take the continuous representative whenever possible. This in particular applies to $e_t, h^{\pm}_t, \Delta e_t,$ and $\Delta h^{\pm}_t$ for $t > 0$. We remark that since  $h_t^\pm$ and $e_t$ are Lipschitz, one can also take the local Lipschitz constant as the representatives of $|\nabla h_t^\pm|$ and $|\nabla e_t|$ by \ref{lip=grad}. These have a sufficiently nice continuity property, see Lemma \ref{lip constant nice}, which makes most of our $m$-a.e. statements about $|\nabla h_t^\pm|$ and $|\nabla e_t|$ pointwise and ease certain measure-theoretic difficulties in the arguments for this section.  

\begin{lem}\label{lip constant nice}
Let $(X,d,m)$ be an $\RCD(K,N)$ space for $K \in \mathbb{R}$ and $N \in [1, \infty)$. Let $f: X \to \mathbb{R}$ be a Lipschitz function. Fix $U \subseteq X$ open and $x \in U$. Then 
\begin{equation*}
	\lip(f)(x) \leq \esssup_{U} \lip(f).
\end{equation*}
\end{lem}

\begin{proof}
	For any $\epsilon > 0$, there exists $y \in U$ so that $\frac{|f(y) - f(x)|}{d(y,x)} \geq \lip(f)(x) - \frac{\epsilon}{2}$. By continuity of $f$, there exists $r>0$ so that $B_{r}(y) \subseteq U$ and for any $z \in B_{r}(y)$, 
$\frac{|f(y) - f(x)|}{d(y,x)} \geq \lip(f)(x) - \epsilon$. Let $\gamma_{z,x}:[0,1] \to X$ be a constant speed geodesic from $z$ to $x$. The local Lipschitz constant is an upper gradient of $f$, see \cite[Remark 2.7]{AGS14a}, and therefore,
\begin{equation*}
	\fint_{B_r(y)}|f(z) - f(x)| \, dm(z) \leq \fint_{B_r(y)}\int_{0}^{1} d(z,x)\lip(f)(\gamma_{z,x}(s))\,ds \, dm(z). 
\end{equation*}
Since we know $\fint_{B_r(y)} |f(z) - f(x)| \geq \lip(f)(x) - \epsilon$ and for each $s < 1$, $(\gamma_{\cdot,x}(s))_{*}(m)$ is absolutely continuous w.r.t. $m$ by Theorem \ref{directional BG}, we conclude $\esssup_{U} \lip(f) \geq \lip(f)(x) - \epsilon$. 
\end{proof}

We proceed with our estimates for $e_t$ and $h^\pm_t$.

\begin{lem}\label{lap cutoffed bound} There exists a constant $c(N, \delta)$ such that for all $t>0$,
\begin{equation}
\Delta h_t^-, -\Delta h_t^+, \Delta e_t \leq \frac{c(N,\delta)}{d_{p,q}}.
\end{equation}
\end{lem}
\begin{proof}
We show the claim for $e_t$; the proof is analogous for others. By Laplacian comparison theorem for the distance function \ref{lap d bound}, see also Remark \ref{lap d bound radon}, and the definition $\psi$, $e_0 \in D(\boldsymbol{\Delta})$ with
\begin{equation}
\boldsymbol{\Delta}e_0 = \Delta\psi e m + \inner{\nabla \psi}{\nabla e}m+ \psi\boldsymbol{\Delta}e \leq \frac{c(N,\delta)}{d_{p,q}}m.
\end{equation}
We know $e_{t}(x) = \int H_t(x,y)e_0(y)dm(y)$. For $t > 0$, $e_t \in D(\Delta)$ and 
\begin{align}
\begin{split}
	\Delta e_t &= \int \Delta_x H_t(x,y)e_0(y) \, dm(y)\\
		&= \int \Delta_y H_t(x,y)e_0(y) \, dm(y) \; \; \text{ , by symmetry of $H_t$}\\
		&= \int H_t(x,y) d\boldsymbol{\Delta} e_0 (y) \; \; \text{ , since $e_0$ is compactly supported}\\
		&\leq \frac{c(N,\delta)}{d_{p,q}}.
\end{split}
\end{align}
\end{proof}

\begin{lem}\label{excess pw bound} There exists a constant $c(N, \delta)$ such that for all $0 < \epsilon \leq \bar{\epsilon}(N,\delta)$ and $x \in X_{\frac{\delta}{4},5}$ the following holds:
	\begin{enumerate}
		\item \label{excess pw bound 1}$|e_{d_{\epsilon}^2}(y)| \leq c\big(\epsilon^{2}d_{p,q} + e(x)\big)$ for every $y \in B_{10d_{\epsilon}}(x);$
		\item \label{excess pw bound 2}$|\nabla e_{d_{\epsilon}^2}|(y) \leq c\bigg(\epsilon+\frac{\epsilon^{-1}e(x)}{d_{p,q}}\bigg)$ for $m$-a.e. $y \in B_{10d_{\epsilon}}(x);$
		\item \label{excess pw bound 3}$ |\Delta e_{d_\epsilon^2}(y)| \leq c\bigg(\frac{1}{d_{p,q}} + \frac{\epsilon^{-2}e(x)}{d_{p,q}^2}\bigg)$ for every $y \in B_{10d_{\epsilon}}(x);$
		\item \label{excess pw bound 4}$\fint_{B_{d_\epsilon}(y)}|\Hess e_{d_\epsilon^2}|_{\HS}^2 \leq c\bigg(\frac{1}{d_{p,q}}+\frac{\epsilon^{-2}e(x)}{d_{p,q}^2}\bigg)^2$ for every $y \in B_{10d_{\epsilon}}(x)$.
	\end{enumerate}
\end{lem}

\begin{proof}
	$e_{t}(x) = e_0(x) + \int\limits_{0}^{t} \Delta e_s(x) ds$ pointwise by definition of the heat flow and the continuity $e_s$. \\
	By Lemma \ref{lap cutoffed bound}, 
	\begin{equation}\label{excess pw bound eq1}
	e_{t}(x) \leq e_0(x) +  \frac{c}{d_{p,q}}t = e(x) + \frac{c}{d_{p,q}}t. 
	\end{equation}
	Setting $s = d_{\epsilon}^2$, $t = 2d_{\epsilon}^2$ and $y \in B_{10d_{\epsilon}}(x)$ in the statement of the Li-Yau Harnack inequality, \ref{LY Harnack}, we conclude 
\begin{align}
\begin{split}
e_{d_\epsilon^2}(y) &\leq c(N)e_{2d_\epsilon^2}(x) \text{ , by \ref{LY Harnack} and $d_{p,q} \leq 1$}\\
	&\leq c(e(x)+\epsilon^2d_{p,q}) \text{ , by \eqref{excess pw bound eq1}}.
\end{split}
\end{align}
This proves statement \ref{excess pw bound 1} of the lemma.

To prove \ref{excess pw bound 3} of the lemma, first notice that we need only establish a lower bound on $\Delta e_{d_\epsilon^2}(y)$ since \ref{lap cutoffed bound} already gives us the desired upper bound. This is an application of the Li-Yau gradient inequality \ref{LY grad} and statement \ref{excess pw bound 1}. The bound holds pointwise even though \ref{LY grad} holds only a.e. due to the existence of a continuous representatitive of $\Delta e_t$.

Statement \ref{excess pw bound 2} of the lemma follows from another application of Li-Yau gradient inequality \ref{LY grad} along with the bounds from statements \ref{excess pw bound 1} and \ref{excess pw bound 3}.

For the last statement take good cut off function $\phi$ supported on $B_{2d_\epsilon}(y)$ with $\phi \equiv 1$ on $B_{d_\epsilon}(y)$. We have
\begin{align}\label{excess pw bound eq2}
\begin{split}
\int\limits_{X} (\Delta e_{d_\epsilon^2})^2\phi dm &= -\int\limits_{X} \inner{\nabla e_{d_\epsilon^2}}{\nabla (\Delta e_{d_\epsilon^2}\phi)} dm\\
 &= -\int\limits_{X} \inner{\nabla e_{d_\epsilon^2}}{\nabla\Delta e_{d_\epsilon^2}}\phi dm  - \int\limits_{X} \inner{\nabla e_{d_\epsilon^2}}{\nabla \phi}\Delta e_{d_\epsilon^2} dm.
\end{split}
\end{align}
Integrating $\phi$ with $|\Hess e_{d_\epsilon^2}|_{\HS}^2$ and applying the improved Bochner inequality \ref{improved Bochner inequality}, we get
\begin{align}
\begin{split}
	&\hspace{0.5cm} \int\limits_{B_{d_\epsilon}} |\Hess e_{d_\epsilon^2}|_{\HS}^2 \, dm \leq \int\limits_{X} |\Hess e_{d_\epsilon^2}|_{\HS}^2 \phi \, dm\\ &\leq \int\limits_{X} \frac{1}{2}\phi d\boldsymbol{\Delta}|\nabla e_{d_\epsilon^2}|^2- \int\limits_{X} \inner{\nabla e_{d_\epsilon^2}}{\nabla \Delta e_{d_\epsilon^2}}\phi \, dm  + \int\limits_{X} (N-1)|\nabla e_{d_\epsilon^2}|^2\phi\,  dm \; \;\text{, by \ref{improved Bochner inequality}}\\
	&\leq \int\limits_{X} \frac{1}{2}\Delta\phi |\nabla e_{d_\epsilon^2}|^2 dm +\int\limits_{X} (\Delta e_{d_\epsilon^2})^2\phi dm + \int\limits_{X} |\nabla e_{d_\epsilon^2}||\nabla \phi|\Delta e_{d_\epsilon^2} dm + \int\limits_{X} (N-1)|\nabla e_{d_\epsilon^2}|^2\phi dm \;\text{, by \eqref{excess pw bound eq2}.}
\end{split}
\end{align}
Applying to this computation properties \ref{excess pw bound 1} - \ref{excess pw bound 3} of the lemma, property 2 of good cut off functions \ref{good ball cut off function} and Bishop-Gromov volume comparison \ref{BG volume comparison}, we obtain statement \ref{excess pw bound 4} of the lemma.
\end{proof}

Next we prove estimates on the heat flow approximation of the distance functions.

\begin{lem}\label{h pw bound}
There exists $c(N,\delta)$ such that for every $\epsilon \leq \bar{\epsilon}(N,\delta)$ and $x \in X_{\frac{\delta}{4},5}$,
\begin{equation*}
	|h^{\pm}_{d_\epsilon^2} - d^{\pm}|(x) \leq c\big(\epsilon^2d_{p,q}+e(x)\big).
\end{equation*}
\end{lem}
\begin{proof}
From the Laplacian bounds in \ref{lap cutoffed bound}, for $x \in X_{\frac{\delta}{4},5}$,
\begin{equation}\label{h pw bound eq1}
	h^-_{d_\epsilon^2}(x) - d^-(x) =  \int_{0}^{d_{\epsilon}^2} \Delta h^- _t(x) \, dt \leq c\epsilon^2 d_{p,q},
\end{equation}
and
\begin{equation}\label{h pw bound eq2}
	d^+(x) - h^+_{d_\epsilon^2}(x) = \int_{0}^{d_{\epsilon}^2} -\Delta h^+ _t(x) \, dt \leq c\epsilon^2 d_{p,q}.
\end{equation}
To obtain bounds in the other direction, we note that
\begin{equation*}
	h^-_{d_\epsilon^2}-d^-(x) = h^+_{d_\epsilon^2}-d^+(x) + e_{d_\epsilon^2}(x) - e(x).
\end{equation*}
We conclude using this with the bound $|e_{d_{\epsilon}^2}(x)| \leq c\big(\epsilon^{2}d_{p,q} + e(x)\big)$ from statement 1 of Lemma \ref{excess pw bound} and bounds \eqref{h pw bound eq1} or \eqref{h pw bound eq2}.
\end{proof}

We will end up wanting to establish appropriate gradient and Hessian bounds along curves that are close to being a geodesic between $p$ and $q$. This requires the following definition.
\begin{de}\label{epsilon geodesic}
A unit speed, piecewise geodesic curve $\sigma$ between $p$ and $q$ is called an $\epsilon$-geodesic between $p$ and $q$ if $||\sigma|-d_{p,q}| \leq \epsilon^2d_{p,q}$, where $|\sigma|$ is the length of $\sigma$. 
\end{de}
\begin{rem}\label{epsilon geodesic excess}
Notice that $x$ lies on an $\epsilon$-geodesic iff $e(x) \leq \epsilon^2d_{p,q}$.
\end{rem}
The previous lemma \ref{h pw bound} can now be restated in terms of $\epsilon$-geodesics.
\begin{cor}\label{h pw bound cor}
There exists $c(N,\delta)$ such that for every $\epsilon$-geodesic between $p$ and $q$ with $\epsilon \leq \bar{\epsilon}(N,\delta)$, and $\frac{\delta}{3}\leq t \leq 1 - \frac{\delta}{3}$,
\begin{equation*}
	|h^{\pm}_{d_\epsilon^2} - d^{\pm}|(\sigma(t)) \leq c(\epsilon^2d_{p,q}).
\end{equation*}
\end{cor}
\begin{proof}
This follows from the \ref{h pw bound} since
\begin{enumerate}
	\item $e(\sigma(t)) \leq \epsilon^2d_{p,q}$
	\item Since $\sigma$ is unit speed and an $\epsilon$-geodesic, as long as $\epsilon < \sqrt{\frac{\delta}{12}}$, we will have $\sigma(t) \in  X_{\frac{\delta}{4},5}.$
\end{enumerate}
\end{proof}

We establish an upper bound on the norm of the gradient of $h^{\pm}_t$ for $x \in X_{\frac{\delta}{5},6}$. 

\begin{lem}\label{grad h pw bound}
There exists $c(N,\delta)$ such that for $\epsilon \leq \bar{\epsilon}(N,\delta)$ and $m$-a.e. $x \in X_{\frac{\delta}{5},6}$, 
\begin{equation*}
	|\nabla h^{\pm}_{d_\epsilon^2}| \leq 1 + cd_{\epsilon}^2.
\end{equation*}
\end{lem}
\begin{proof}
By Bakry-Ledoux estimate \ref{Bakry-Ledoux}, for any $t > 0$,
\begin{equation}\label{grad h ae bound eq1}
	|\nabla h^{\pm}_t| \leq e^{2(N-1)t}H_t(|\nabla h^{\pm}_0|) \;\text{ $m$-a.e..}
\end{equation}
By definition of $h^{\pm}_0 = \psi d^{\pm}$, we have the following a.e. bounds on $|\nabla h^{\pm}_0|$.
\begin{enumerate}
	\item $|\nabla h^{\pm}_0| = 1$ in $X_{\frac{\delta}{8},8}$
	\item $|\nabla h^{\pm}_0| = 0$ in $X \backslash X_{\frac{\delta}{16},16}$
	\item In $X_{\frac{\delta}{16},16} \backslash X_{\frac{\delta}{8},8},$ 
	\begin{equation}\label{grad h ae bound eq2}
 		|\nabla h^{\pm}_0|\; =\; |\nabla \psi||d^{\pm}|+|\psi||\nabla d^{\pm}| \; \leq \; \frac{c(N)}{\delta d_{p,q}}|d^{\pm}|+1 \; \leq \; c(N,\delta).
	\end{equation}
\end{enumerate}
	Finally, for any $x \in X_{\frac{\delta}{2},4}$,
	\begin{align}\label{grad h ae bound eq3}
	\begin{split}
	H_t(|\nabla h^{\pm}_0|)(x) &= \int\limits_{X} H_t(x,y)|\nabla h^{\pm}_0|(y)\,dm(y)\\
					&= \int_{X_{\frac{\delta}{16},16} \backslash X_{\frac{\delta}{8},8}} H_t(x,y)|\nabla h^{\pm}_0|(y)\,dm(y) + \int_{X_{\frac{\delta}{4},8}}	H_t(x,y)|\nabla h^{\pm}_0|(y)\,dm(y)\\
				&\leq c(N,\delta) \int_{X_{\frac{\delta}{16},16} \backslash X_{\frac{\delta}{8},8}} \nabla H_t(x,y)\,dm(y) + 1 \; \; \text{, by \eqref{grad h ae bound eq2}}\\
				&\leq c(N, \delta) (\frac{\delta}{8})^{-2}t + 1,
	\end{split}
	\end{align}
where the last lines uses statement 2 of the heat kernel bounds \ref{heat kernel bounds}.

	The lemma follows by combining \eqref{grad h ae bound eq1} and \eqref{grad h ae bound eq3}, with $t = d_{\epsilon}^2$ for small $\epsilon$.
\end{proof}


We will now establish some integral bounds on $|\nabla h^{\pm}_t|$. Roughly, we want to apply the $L^1$-Harnack inequality \ref{MV L1H inequality} to $|\nabla h^{\pm}|$. To this effect, we give some regularity of the heat flow in the time parameter.

\begin{lem}\label{H_t time regularity}
Let $(X,d,m)$ be an $\RCD(K,N)$ space for some $K \in \mathbb{R}$, $N \in [1, \infty)$. Let $f \in L^2(m)$.
\begin{enumerate}
	\item $H_t(f) \in C^{1}\big((0,\infty), L^2(m)\big)$ and 
	\begin{equation}\label{h_t time regularity eq1}
		\frac{d}{dt} H_t(f) = \Delta H_t(f) \; \; \forall t > 0;
	\end{equation}
	\item $H_t(f) \in C^0\big((0,\infty), W^{1,2}(X)\big)$;
	\item $H_t(f) \in C^{1}\big((0,\infty), W^{1,2}(X)\big)$ and in particular $|\nabla H_t(f)|^2 \in C^{1}((0,\infty), L^1(m)\big)$ with
		\begin{equation}\label{h_t time regularity eq2}
		\frac{d}{dt} |\nabla H_t(f)|^2 =2\inner{\nabla H_t(f)}{\nabla \Delta H_t(f)} \; \; \forall t > 0.
		\end{equation}
	If $f$ or $|\nabla f|$ is in $L^{\infty}(X, m)$, then $|\nabla H_t(f)|^2 \in C^{1}((0,\infty), L^2(m)\big)$ and the same formula holds. 
\end{enumerate}
\end{lem}

\begin{proof}
Since $H_t(f) = H_{t'}(f) + \int_{t'}^{t} \Delta H_s(f) ds$ for $t > t' > 0$ and $\Delta H_s(f) \in C^0\big((0,\infty), L^2(m)\big)$, statement 1 follows by fundamental theorem of calculus.

For $t > 0$,
\begin{align*}
	\lim\limits_{s \to 0} \int\limits_{X}|\nabla H_{t+s}(f)-\nabla H_{t}(f)|^2 &= \lim\limits_{s \to 0} \int\limits_{X} (H_{t+s}(f)\Delta H_{t+s}(f)) + 2(H_{t+s}(f)\Delta H_{t}(f)) - (H_{t}(f)\Delta H_{t}(f))\\
	&= 0 \text{ , since all terms involved are continuous from $s \to L^2(m)$.}
\end{align*}
We already know $H_t(f) \in C^0\big((0,\infty), L^{2}(X)\big)$, so statement 2 follows. 

Applying statement 2 to $\Delta H_{\epsilon}(f)$ for arbitrarily small positive $\epsilon$, we see that $\Delta H_t(f) \in C^0((0,\infty),$ $W^{1,2}(X))$. The first part of statement 3 then follows by applying fundamental theorem of calculus to $H_t(f) = H_{t'}(f) + \int_{t'}^{t} \Delta H_s(f) ds$ viewed as a $W^{1,2}(X)$-valued Bochner integral. The second part follows by a direct computation. Notice that if $f$ or $|\nabla f|$ is in $L^{\infty}(X, m)$, then $|\nabla h_t(f)| \in L^{\infty}$ by $L^{\infty}$-to-Lipschitz regularization \eqref{Linf-to-Lip} or by Bakry-Ledoux estimate \ref{Bakry-Ledoux}.
\end{proof}

\begin{lem}\label{grad h weak solution}
Let $\phi \in D(\Delta)$ be nonngative, compactly supported, time independent with $|\phi|, |\nabla \phi|, |\Delta \phi| \leq K_1$. If $h$ is the heat flow of some $h_0 \in L^2(m) \cap L^{\infty}(m)$ and $|\nabla h| \leq K_2$ on $\braces{\phi > 0}$, then $(\frac{\partial}{\partial t}-\Delta) [\phi^2 |\nabla h|^2] \leq c(N, K_1, K_2)$ weakly in $(0, \infty) \times X$ as in Definition \ref{weak solutions}.
\end{lem}
\begin{proof}
	Let $t>0$, $h_t \in \TestF(X)$ by the $L^\infty$-to-Lipschitz regularization property of $H_t$ \eqref{Linf-to-Lip}. Let $f \in \TestF(X)$. By \cite[Proposition 3.3.22]{G18}, $\inner{\nabla h_t}{\nabla h_t} \in W^{1,2}(X)$ and 
\begin{equation*}
	\inner {\nabla \inner{\nabla h_t}{\nabla h_t}}{\nabla f} = 2\Hess(h_t)(\nabla h_t, \nabla f) \; \; m\text{-a.e..}
\end{equation*}
Therefore, $|\nabla \inner{\nabla h_t}{\nabla h_t}{\nabla f}| \leq 2|\Hess(h_t)|_{\HS}|\nabla h_t||\nabla f| \, m$-a.e.. 

We then have, for any $\epsilon > 0$ and $m$-a.e.,
\begin{align*}
	4\phi|\inner{\nabla|\nabla h_t|^2}{\nabla \phi}| &\leq 8\phi|\Hess(h_t)|_{\HS}|\nabla h_t||\nabla \phi| \; \; \text{ , by above and the density of $\TestF(X)$ in $W^{1,2}(X)$.}\\
	&\leq 4\epsilon\phi^2|\Hess(h_t)|_{\HS}^2|\nabla h_t|^2 + \frac{4}{\epsilon}|\nabla \phi|^2.
\end{align*}
Choosing $\epsilon>0$ small so that $4\epsilon K_2^2 < 2$, 
\begin{align*}
	\boldsymbol{\Delta}(\phi^2|\nabla h_t|^2) &= \phi^2 \boldsymbol{\Delta}|\nabla h_t|^2 + (2\inner{\nabla|\nabla h_t|^2}{\nabla \phi^2}+|\nabla h_t|^2\Delta\phi^2)m\\
	&\geq( 2 \phi^2|\Hess(h_t)|_{\HS}^2 + 2\phi^2\inner{\nabla \Delta h_t}{\nabla h_t} - 2(N-1)\phi^2|\nabla h_t|^2 + 4\phi\inner{\nabla|\nabla h_t|^2}{\nabla \phi}+|\nabla h_t|^2\Delta \phi^2)m\\
	&\geq (2 \phi^2\inner{\nabla \Delta h_t}{\nabla h_t}-c)m,
\end{align*}
where the improved Bochner inequality \ref{improved Bochner inequality} is used for line 2 and the previous estimate with $\epsilon$ was used for line 3. 

Finally, by Lemma \ref{H_t time regularity}, $\frac{d}{dt} \phi^2|\nabla h_t|^2 = 2\phi^2\inner{\nabla \Delta h_t}{\nabla h}$. This lets us conclude. 
\end{proof}

\begin{thm}\label{grad h integral bound}
There exists a constant $c(N,\delta)$ such that for all $\epsilon \leq \bar{\epsilon}(N,\delta)$,
\begin{enumerate}
	\item if $x \in X_{\frac{\delta}{2},3}$ with $e(x) \leq \epsilon^2d_{p,q}$ then $\fint_{B_{10d_\epsilon}(x)} ||\nabla h^{\pm}_{d_{\epsilon}^2}|-1|\leq c\epsilon;$
	\item if $\sigma$ is an $\epsilon$-geodesic connecting $p$ and $q$, then $\int\limits_{\frac{\delta}{2} d_{p,q}}^{(1-\frac{\delta}{2}) d_{p,q}} \fint_{B_{10d_\epsilon}(\sigma(s))} ||\nabla h^{\pm}_{d_{\epsilon}^2}|-1|\leq c\epsilon^2d_{p,q}.$
\end{enumerate}
\end{thm}
\begin{proof}
We prove the theorem in the case of $h_t^-$. The $h_t^+$ case is similar. We will take the local Lipschitz constant representatitves for $|\nabla h^\pm_t|$. All statements made will be for $\epsilon$ sufficiently small depending on $N$ and $\delta$ so we will forgo repeating this. 

By Lemma \ref{grad h pw bound}, we choose $c'(N,\delta)$ so that $|\nabla h_t^{-}| \leq 1+c't^2$ for all $x \in X_{\frac{\delta}{5},6}$ and $t \leq \epsilon'(N,\delta)^2d_{p,q}^2$. This means there exists $c''(N,\delta)$ so that 
\begin{equation*}
	w_t := 1 + c''t - |\nabla h_t^-|^2 \geq 0 \; \text{on} \;X_{\frac{\delta}{5},6}.
 \end{equation*}
Let $\phi = \phi^+\phi^-$, where $\phi^{\pm}$ are annular good cutoff functions (\ref{good cut off function}) around $p$ and $q$ respectively so that $\phi = 1$ on $X_{\frac{\delta}{4},5}$ and $\phi = 0$ on $X\backslash X_{\frac{\delta}{5},6}$. By Lemma \ref{grad h weak solution},
\begin{equation*}
	(\partial_t - \Delta)|\phi^2w_t| \geq -c \; \text{ weakly in $(0,d_{\epsilon'}^2) \times X$.}
\end{equation*}
Applying the $L^1$-Harnack inequality \ref{MV L1H inequality}, we have, for $x \in X_{\frac{\delta}{3},4}$,
\begin{equation}\label{grad h integral bound eq1}
	\fint_{B_{10d_\epsilon}(x)} w_{d_\epsilon^2} \leq c[\essinf_{B_{10d_\epsilon}(x)} w_{2d_{\epsilon}^2}+d_{\epsilon}^2].
\end{equation}
We will show that the right side is sufficiently small. Let $e(x) \leq \epsilon^2d_{p,q}$ and let $\gamma(t)$ be a unit speed geodesic from $x$ to $p$. By Corollary \ref{h pw bound cor}, 
\begin{align}\label{grad h integral bound eq2}
\begin{split}
	&\hspace{0.5cm} |h_{2d_\epsilon^2}^-(x) - h_{2d_\epsilon^2}^-(\gamma(10d_\epsilon)) - 10d_\epsilon|\\
	&\leq |h_{2d_\epsilon^2}^-(x) - d^-(x)| +|h_{2d_\epsilon^2}^-(\gamma(10d_\epsilon)) - d^-(\gamma(10d_\epsilon))| + |d^-(x) - d^-(\gamma(10d_\epsilon)) - 10d_\epsilon|\\
	&\leq c\epsilon^2d_{p,q}.
\end{split}
\end{align}
The local Lipschitz constant $|\nabla h_{2d_\epsilon^2}^-|$ is an upper gradient, \cite[Remark 2.7]{AGS14a}. Therefore,
\begin{equation}\label{grad h integral bound eq3}
	|h_{2d_\epsilon^2}^-(\sigma) - h_{2d_\epsilon^2}^-(\gamma(10d_\epsilon))| \leq \int_{0}^{10d_\epsilon} |\nabla h_{2d_\epsilon^2}^-|(\gamma(s)) \, ds.
\end{equation}
We have 
\begin{align}\label{grad h integral bound eq4}
\begin{split}
	\int_{0}^{10d_\epsilon} w_{2d_\epsilon^2}(\gamma(s))\, ds &= \int_{0}^{10d_\epsilon} \big( 1+cd_\epsilon^2-|\nabla h^-_{2d_\epsilon^2}|^2(\gamma(s))\big) \, ds\\
		& \leq 10d_{\epsilon}+cd_\epsilon^3-\frac{1}{10d_\epsilon}(\int_{0}^{10d_\epsilon} |\nabla h^-_{2d_\epsilon^2}|(\gamma(s)) \, ds)^2 \; \; \text{ , by Cauchy-Schwarz}\\
		&\leq  10d_{\epsilon}+cd_\epsilon^3-\frac{1}{10d_\epsilon}(h_{2d_\epsilon^2}^-(x) - h_{2d_\epsilon^2}^-(\gamma(10d_\epsilon)))^2 \; \; \text{ , by \eqref{grad h integral bound eq3}}\\
		&\leq 10d_{\epsilon}+cd_\epsilon^3-\frac{1}{10d_\epsilon}(10d_\epsilon - c\epsilon^2d_{p,q})^2 \; \; \text{ , by \eqref{grad h integral bound eq2}}\\
		&\leq c\epsilon \; \; \text{ , if $\epsilon < 1$}
\end{split}
\end{align}
In particular, there exists $s\in[0,10d_\epsilon]$ so that $w_{2d_\epsilon^2}(\gamma(s)) \leq c\epsilon$. Applying Lemma \ref{lip constant nice} to $|\nabla h_t^{\pm}|$ and $\gamma(s) \in \overline{B_{10d_{\epsilon}}(x)}$, we conclude $\essinf\limits_{B_{10d_\epsilon}(x)} w_{2d_{\epsilon}^2} \leq c\epsilon$ and so statement 1 is proved by \eqref{grad h integral bound eq1}. 

By Corollary \ref{h pw bound cor}, arguing as in \eqref{grad h integral bound eq2} ,
\begin{equation}\label{grad h integral bound eq5}
	|h_{2d_\epsilon^2}^-\bigg(\sigma\big((1-\frac{\delta}{2}) d_{p,q}\big)\bigg) - h_{2d_\epsilon^2}^-\bigg(\sigma\big(\frac{\delta}{2} d_{p,q}\big)\bigg) - (1-\delta) d_{p,q}| \; \leq \; c\epsilon^2d_{p,q}.
\end{equation}
Arguing as in \eqref{grad h integral bound eq3} and \eqref{grad h integral bound eq4},
\begin{equation}\label{grad h integral bound eq6}
	\int_{\frac{\delta}{2} d_{p,q}}^{(1-\frac{\delta}{2})d_{p,q}} w_{2d_\epsilon^2}(\sigma(s))  \, ds \leq c\epsilon^2d_{p,q}.
\end{equation}
Finally,
\begin{align}\label{grad h integral bound eq7}
\begin{split}
	\int_{\frac{\delta}{2} d_{p,q}}^{(1-\frac{\delta}{2}) d_{p,q}} \bigg( \fint_{B_{10d_{\epsilon}}(\sigma(s))} ||\nabla h^-_{d_\epsilon^2}|^2 - 1| \bigg) \, ds  &\leq \int_{\frac{\delta}{2} d_{p,q}}^{(1-\frac{\delta}{2})d_{p,q}} \bigg( \fint_{B_{10d_{\epsilon}}(\sigma(s))} w_{d_\epsilon^2}+c\epsilon^2d_{p,q} \bigg)\, ds\\
	&\leq c \int_{\frac{\delta}{2} d_{p,q}}^{(1-\frac{\delta}{2}) d_{p,q}} \bigg(\essinf_{B_{10d_\epsilon^2}(\sigma(s))} w_{2d_\epsilon^2} + c\epsilon^2d_{p,q}\bigg) \, ds \; \; \text{, by \eqref{grad h integral bound eq1}}\\
	&\leq c \int_{\frac{\delta}{2} d_{p,q}}^{(1-\frac{\delta}{2})d_{p,q}} \bigg( w_{2d_\epsilon^2}(\sigma(s)) + c\epsilon^2d_{p,q}\bigg) \, ds \; \; \text{, by Lemma \ref{lip constant nice}}\\
	&\leq c\epsilon^2d_{p,q}.
\end{split}
\end{align}
\end{proof}

We now prove the main Hessian estimate for $h_t^{\pm}$. 
\begin{thm}\label{hess h integral bound}
There exists a constant $c(N,\delta)$ such that for any $0 < \epsilon \leq \bar{\epsilon}(N,\delta)$, any $x \in X_{\frac{\delta}{2},3}$ with $e(x) \leq \epsilon^2d_{p,q}$, or any $\epsilon$-geodesic $\sigma$ connecting $p$ and $q$, there exists $r \in [\frac{1}{2}, 2]$ with 
\begin{enumerate}
	\item $|h^\pm_{rd_\epsilon^2} - d^\pm| \leq c\epsilon^2d_{p,q};$
	\item $\fint_{B_{d_\epsilon}(x)}||\nabla h^\pm_{rd_\epsilon^2}|^2-1|\leq c\epsilon;$
	\item $\int_{\frac{\delta}{2} d_{p,q}}^{(1-\frac{\delta}{2}) d_{p,q}} \bigg( \fint_{B_{d_\epsilon}(\sigma(s))}||\nabla h^\pm_{rd_\epsilon^2}|^2-1| \bigg) \, ds \leq c\epsilon^2d_{p,q};$
	\item $\int_{\frac{\delta}{2} d_{p,q}}^{(1-\frac{\delta}{2}) d_{p,q}} \bigg( \fint_{B_{d_\epsilon}(\sigma(s))} |\Hess h^{\pm}_{rd_\epsilon^2}|^2\bigg) \, ds \leq \frac{c}{d_{p,q}^2}$. 
\end{enumerate}
\end{thm}

\begin{proof}
Statement 1 follows from Lemma \ref{h pw bound} and statements 2 and 3 follow from Theorem \ref{grad h integral bound} with Bishop-Gromov. Note that any $r \in [\frac{1}{2}, 2]$ works in the first 3 statements. 

Using \ref{good ball cut off function}, we fix, for each $s \in (\frac{\delta}{2} d_{p,q}, (1-\frac{\delta}{2})d_{p,q})$, good cut off function $\phi$ with $\phi \equiv 1$ on $B_{d_\epsilon}(\sigma(s))$, vanishing outside of $B_{3d_\epsilon}(\sigma(s))$, and $d_\epsilon |\nabla \phi|$, $d_\epsilon^2 |\Delta \phi| \leq c(N)$. Similarly, fix $\alpha(t)$ a smooth function in time so that $0 \leq \alpha(t) \leq 1$, $\alpha(t) \equiv 1$ for $t \in [\frac{1}{2}d_\epsilon^2, 2d_\epsilon^2]$, vanishing for $t$ outside of $[\frac{1}{4}d_\epsilon^2, 4d_\epsilon^2]$, and satisfying $|\alpha'| \leq 10d_\epsilon^{-2}$. 

Applying the improved Bochner inequality \ref{improved Bochner inequality} to $h_t^{\pm}$, we obtain, for each $s, t$
\begin{align}\label{hess h integral bound eq1} \begin{split}
	& \int \alpha(t)\phi|\Hess h^{\pm}_t|^2\, dm \leq \int \alpha(t) \phi d(\boldsymbol{\Delta}|\nabla h^\pm_t|^2)+ 2\int\alpha(t)\phi\bigg((N-1)|\nabla h^{\pm}_t|^2  - \inner{\nabla h^\pm_t}{\nabla \Delta h^\pm_t}\bigg) \, dm\\
	&= \int \alpha(t) (|\nabla h^\pm_t|^2-1)\Delta(\phi) \, dm + 2(N-1)\int\alpha(t)\phi|\nabla h^{\pm}_t|^2 \, dm - \int \alpha(t)\phi \partial_t(|\nabla h^\pm_t|^2)\, dm.
\end{split} \end{align}
In the last line, we used the definition of the Laplacians along with the fact that $\int \Delta\phi dm = 0$ for the first term and Lemma \ref{H_t time regularity} for the third term. 
Integrating in time using integration by parts and $\int_{0}^{\infty} \alpha'(t)dt = 0$ on the third term of the previous line, 
\begin{align}\label{hess h integral bound eq2} 
\begin{split}
	&\hspace{0.5cm} \int_0^\infty \int \alpha(t)\phi|\Hess h^{\pm}_t|^2\, dm \, dt\\
	&\leq \int_0^\infty \bigg(\int \alpha(t) (|\nabla h^\pm_t|^2-1)\Delta(\phi) \, dm + 2(N-1)\int\alpha(t)\phi|\nabla h^{\pm}_t|^2 \, dm + \int \alpha'(t)\phi (|\nabla h^\pm_t|^2-1)\, dm \bigg)\, dt.
\end{split} 
\end{align}
Using what we know about $\phi$ and $\alpha$ and using Bishop-Gromov in line 2 of the following, we obtain
\begin{align}\label{hess h integral bound eq3}
\begin{split}
	&\hspace{0.5cm} \int_{\frac{1}{2}d_\epsilon^2}^{2d_\epsilon^2} \fint_{B_{d_\epsilon}(\sigma(s))} |\Hess h^{\pm}_t|^2\, dm \, dt\\
	&\leq \int_{\frac{1}{4}d_\epsilon^2}^{4d_\epsilon^2} \bigg(\fint_{B_{3d_\epsilon}(\sigma(s))} \bigg((|\nabla h^\pm_t|^2-1)\Delta(\phi) + 2(N-1) |\nabla h^{\pm}_t|^2  + \alpha'(t)(|\nabla h^\pm_t|^2-1)\bigg)\, dm \bigg)\, dt\\
	&\leq \int_{\frac{1}{4}d_\epsilon^2}^{4d_\epsilon^2} \bigg(\fint_{B_{3d_\epsilon}(\sigma(s))}2(N-1)+ cd_\epsilon^{-2} ||\nabla h_t^{\pm}|^2-1| \,dm \bigg) \, dt.
\end{split}
\end{align}

Integrating across $\sigma$ for $s \in [\frac{\delta}{2} d_{p,q}, (1-\frac{\delta}{2})d_{p,q}]$, 
\begin{align}\label{hess h integral bound eq4}
\begin{split}
	&\hspace{0.5cm} \int_{\frac{1}{2}d_\epsilon^2}^{2d_\epsilon^2} \bigg( \int_{\frac{\delta}{2} d_{p,q}}^{(1-\frac{\delta}{2}) d_{p,q}}\fint_{B_{d_\epsilon}(\sigma(s))} |\Hess h^{\pm}_t|^2\, dm \, ds \bigg) \,dt\\ 
	&\leq cd_\epsilon^{-2}\int_{\frac{1}{4}d_\epsilon^2}^{4d_\epsilon^2} \bigg(\int_{\frac{\delta}{2} d_{p,q}}^{(1-\frac{\delta}{2}) d_{p,q}} \fint_{B_{3d_\epsilon}(\sigma(s))}cd_\epsilon^2+ ||\nabla h_t^{\pm}|^2-1| \,dm \, ds\bigg) \, dt\\
	&\leq c\epsilon^2d_{p,q} \; \; \text{, by statement 2 of Theorem \ref{grad h integral bound}.}
\end{split}
\end{align}
Therefore, statement 4 holds for some $r \in [\frac{1}{2},2]$ and $t = rd_\epsilon^2$.
\end{proof}

\begin{lem}\label{grad h geodesic bound}
Let $\epsilon \leq \bar{\epsilon}(N,\delta)$. Let $\gamma_{x,p}$ be any unit speed geodesic from $x \in X$ to $p$. Then for $m$-a.e. $x \in X_{\frac{\delta}{2},3}$ and any $0 \leq t_1 < t_2 \leq d_{x,p}-\frac{\delta}{2}$, the following estimates hold:
\begin{enumerate}
	\item $\int_0^{d_{x,p}-\frac{\delta}{2}}||\nabla h_{d_\epsilon^2}^-|^2-1|(\gamma_{x,p}(s)) \, ds \leq \frac{c(N,\delta)}{d_{p,q}}(e(x)+d_\epsilon^2);$
	\item $\int_0^{d_{x,p}-\frac{\delta}{2}}|\inner{\nabla h_{d_\epsilon^2}^-}{\nabla d^-}-1|(\gamma_{x,p}(s)) \, ds\leq \frac{c(N,\delta)}{d_{p,q}}(e(x)+d_\epsilon^2);$
	\item $\int_{t_1}^{t_2}|\nabla h_{d_\epsilon^2}^- - \nabla d^-|(\gamma_{x,p}(s))\, ds \leq  \frac{c(N,\delta)\sqrt{t_2-t_1}}{\sqrt{d_{p,q}}}(\sqrt{e(x)}+d_\epsilon).$
\end{enumerate}
\end{lem}

\begin{proof}
The bounds on $(|\nabla h_{d_\epsilon^2}^-|^2-1)_+$ and $(\inner{\nabla h_{d_\epsilon^2}^-}{\nabla d^-}-1)_+$ for statements 1 and 2 come from Lemma \ref{grad h pw bound}, Fubini's theorem and statement 2 of Theorem \ref{-dp grad flow properties}. The bound on the negative part comes from an estimate like \eqref{grad h integral bound eq5} combined with Corollary \ref{W2 geodesic cont equation cor} for statement 2, and then an additional application Cauchy-Schwarz for statement 1. We note that if one traces the proof of \eqref{grad h integral bound eq5} back to Lemma \ref{h pw bound}, it is clear that one can obtain bounds where the excess is not related to the heat flow time as they have been for the past several claims.

For statement 3, 
\begin{equation*}
	|\nabla h_{d_\epsilon^2}^- - \nabla d^-|^2 = |\nabla h_{d_\epsilon^2}^-|^2 +1 - 2\inner{\nabla h_{d_\epsilon^2}^-}{\nabla d^-} \leq ||\nabla h_{d_\epsilon^2}^-|^2-1|+2|\inner{\nabla h_{d_\epsilon^2}^-}{\nabla d^-} - 1| \; \text{$m$-a.e..}
\end{equation*}
Therefore, statement 1, 2, Cauchy-Schwarz and an argument by Fubini's theorem using statement 2 of Theorem \ref{-dp grad flow properties} gives statement 3. 
\end{proof}

\section{Gromov-Hausdorff approximation}\label{section 5}

This section will be divided into three subsections. The main lemma proved in the first subsection gives a way of overcoming the lack of start of induction in the arguments of \cite{CN12} generalized to the $\RCD$ setting. In the second subsection we use the main lemma to construct geodesics with nice properties in its interior. Finally, we prove the main theorem in the third subsection. To be precise, we prove a slightly weaker version of the main theorem analagous to the main result of \cite{CN12}, which will be used to prove non-branching in Section \ref{section 6} and, subsequently, the main theorem. 

Fix $(X,d,m)$, an $\RCD(-(N-1),N)$ metric measure space for $N \in (1,\infty)$ and $p, q \in X$ with d(p,q)=1. Fix $0<\delta<0.1$. For any $x_1, x_2 \in X$, we fix a constant speed geodesic from $x_1$ to $x_2$ parameterized on $[0,1]$ and denote it $\tilde{\gamma}_{x_1,x_2}$. By Remark \ref{a.e. unique geodesic}, we may assume the map $X \times X \times [0,1] \ni (x_1,x_2,t) \mapsto \tilde{\gamma}_{x_1,x_2}(t)$ is Borel. The unit speed reparameterizations of $\tilde{\gamma}_{x_1,x_2}$ to the interval $[0,d(x_1,x_2)]$ will be denoted $\gamma_{x_1,x_2}$. $\gamma$ will denote $\gamma_{p,q}$. For each $x \in X$, define $\Psi: X \times [0,\infty) \to X$ by 
\begin{equation}\label{Psidef}
	(x,s) \mapsto \Psi_{s}(x) = \begin{cases} 
       \gamma_{x,p}(s) &  \text {if } d(x,p) \geq s,\\
	p & \text{if } d(x,p) < s.\\
       \end{cases}
\end{equation}
Similarly, define $\Phi: X \times [0,\infty) \to X$ by 
\begin{equation}\label{Phidef}
	(x,s) \mapsto \Phi_{s}(x) = \begin{cases} 
       \gamma_{x,q}(s) &  \text {if } d(x,q) \geq s,\\
	q & \text{if } d(x,q) < s.\\
       \end{cases}
\end{equation}

By integral Abresch-Gromoll inequality \ref{Abresch-Gromoll}, for any sufficiently small $r \leq \bar{r}(N,\delta)$ and any $\delta\leq t_0 \leq 1-\delta$, 
\begin{equation*}
	\fint_{B_{r}(\gamma(t_0))} e \leq c_0(N, \delta)r^2.
\end{equation*}
Therefore, there exists a subset $S \subseteq B_r(\gamma(t_0))$ so that 
\begin{enumerate}
	\item $\frac{m(S)}{m(B_r(\gamma(t_0))} \geq 1-\frac{V(1,10)}{3}$ (see \ref{BG volume comparison} for the definiton of $V:=V_{-(N-1),N}$)
	\item $\forall z \in S, e(z) \leq c_1(N,\delta)^2 r^2.$
\end{enumerate} 
We fix such a ${c_1}$\label{c1} for the rest of this section and assume in addition $c_1 > 100$. 

In all subsections the letter $c$ will be used to represent different constants which only depend on $N$ and $\delta$. Any constant which will be used repeatedly will be given a subscript. We will continue using the notations of Section \ref{section 4}.

\subsection{Proof of main lemma}

\begin{lem}\label{main lemma} (Main lemma)
There exists $\epsilon_1(N,\delta)>0$ and $\bar{r}_1(N,\delta) > 0$ so that for all $r \leq \bar{r}_1$ and $\delta \leq t_0 \leq 1-\delta$, there exists $z \in B_r(\gamma(t_0))$ so that
\begin{enumerate}
	\item $V(1, 100) \leq \frac{m(B_r(\Psi_s(z)))}{m(B_r(z))} \leq \frac{1}{V(1,100)}$ for any $s \leq \epsilon_1$.
	\item There exists $A \subseteq B_{r}(z)$ with $m(A) \geq (1-V(1,10))m(B_r(z))$ and $\Psi_{s}(A) \subseteq B_{2r}(\Psi_{s}(z))$ for any $s \leq \epsilon_1$. 
	\item $e(z) \leq c_1^2r^2$.
\end{enumerate}
\end{lem}

\begin{proof}
Fix $\delta \leq t_0 \leq 1-\delta$ and a scale $r \leq \bar{r}_1(N,\delta)$. $\bar{r}_1$ need only be chosen smaller than the radius bounds required for the application of various theorems in the proof; most notably Theorem \ref{Abresch-Gromoll} and the estimates of Section \ref{section 4}. It will be clear that all the required radius bounds only depend on $N$ and $\delta$ so we will not address this each time for the sake of brevity. In addition, we assume $\bar{r}_1 \leq \frac{\delta}{10}$.

By Bishop-Gromov volume comparison \ref{BG volume comparison}, integral Abresch-Gromoll inequality \ref{Abresch-Gromoll}, and the fact that $\Psi$ is defined using unit speed geodesics, it is clear there exist $\epsilon$ depending on $N$, $\delta$ and $r$, and $z \in B_r(\gamma(t_0))$ satisfying
\begin{enumerate}
	\item \label{main lemma 1}  $V(1, 100) \leq \frac{m(B_r(\Psi_s(z)))}{m(B_r(z))} \leq \frac{1}{V(1,100)}$ for any $s \leq \epsilon$.
	\item \label{main lemma 2} There exists $A \subseteq B_{r}(z)$ with $m(A) \geq (1-V(1,10))m(B_r(z))$ and $\Psi_{s}(A) \subseteq B_{2r}(\Psi_{s}(z))$ for any $s \leq \epsilon$. 
	\item \label{main lemma 3} $e(z) \leq c_1^2r^2$.
\end{enumerate} We will remove the dependence of $\epsilon$ on $r$. 

To this effect, we will show that if \ref{main lemma 1}, \ref{main lemma 2} and \ref{main lemma 3} hold for some $z \in B_r(\gamma(t_0))$ and all $s \leq \epsilon$ less than or equal to some $\epsilon_1(N,\delta)$ to be fixed later, then in fact we can choose $z' \in B_r({\gamma(t_0)})$ satisfying \ref{main lemma 3} which significantly improves the estimates in \ref{main lemma 1} and \ref{main lemma 2} for $s \leq \epsilon$. To be precise, we find $z' \in B_r(\gamma(t_0))$ satisfying
\begin{enumerate}[label=\arabic*'$_{.}$,ref=\arabic*']
	\item \label{main lemma 1'}$2V(1,100) \leq \frac{m(B_r(\Psi_s(z')))}{m(B_r(z'))} \leq \frac{1}{2V(1,100)}$ for any $s \leq \epsilon$.
	\item \label{main lemma 2'}There exists $A' \subseteq B_{r}(z')$ with $m(A') \geq (1-V(1,10))m(B_r(z'))$ and $\Psi_{s}(A') \subseteq B_{\frac{3r}{2}}(\Psi_{s}(z'))$ for any $s \leq \epsilon$. 
	\item \label{main lemma 3'}$e(z') \leq c_1^2r^2$.
\end{enumerate}
We a priori assume $\epsilon_1 \leq \frac{\delta}{10}$ and impose more bounds on $\epsilon_1$ as the proof continues. Let $z$ satisfy \ref{main lemma 1}, \ref{main lemma 2} and \ref{main lemma 3} for some $\epsilon \leq \epsilon_1$. 

Let $w$ be the midpoint of $z$ and $\gamma(t_0)$. We know $B_{\frac{r}{2}}(w) \subseteq B_r(\gamma(t_0)) \cap B_r(z)$. By Bishop-Gromov and since $r \leq \bar{r}$ which was assumed to be less than $0.01$,
\begin{equation*}
	\frac{m(B_r(\gamma(t_0)) \cap B_r(z))}{m(B_r(z))} \geq \frac{m(B_{\frac{r}{2}}(w))}{m(B_r(z))}
		\geq \frac{m(B_{\frac{r}{2}}(w))}{m(B_{\frac{3r}{2}}(w))}  \geq V(\frac{r}{2}, \frac{3r}{2}) > V(\frac{1}{2},  \frac{3}{2}).
\end{equation*} 
Using $m(A) \geq (1-V(1,10))m(B_r(z)) > (1-\frac{V(\frac{1}{2}, \frac{3}{2})}{3})m(B_r(z))$ and the previous estimate,
\begin{equation*}
	\frac{m(A \cap B_r(\gamma(t_0)))}{m(B_r(z))} > \frac{2}{3}V(\frac{1}{2}, \frac{3}{2}).
\end{equation*}
Therefore, 
\begin{align}\label{main lemma eq1}
\begin{split}
	\frac{m(A \cap B_r(\gamma(t_0)))}{m(B_r(\gamma(t_0)))} &= \frac{m(A \cap B_r(\gamma(t_0)))}{m(B_r(z))}\frac{m(B_r(z))}{m(B_r(\gamma(t_0)))}\\
	& > \frac{2}{3}V(\frac{1}{2}, \frac{3}{2})V(r, 2r) > \frac{2}{3}V(\frac{1}{2}, \frac{3}{2})V(\frac{3}{2}, 3)> \frac{2}{3}V(1,10).
\end{split}
\end{align} 

Define the set 
\begin{equation}\label{main lemma eq2}
D_1 := A \cap B_{r}(\gamma(t_0)) \cap \left\{e(x) \leq c_1^2r^2 \right\},
\end{equation}
where \hyperref[c1]{$c_1$} is as fixed earlier. We will choose a $z'$ satisfyings propereties \ref{main lemma 1'} - \ref{main lemma 3'} from $D_1$. From \eqref{main lemma eq1} and the definition of $c_1$,
\begin{equation*}
	\frac{m(D_1)}{m(B_r(\gamma(t_0)))} > \frac{1}{3}V(1,10).
\end{equation*}
Therefore, by Bishop-Gromov,
\begin{equation}\label{main lemma eq3}
	\frac{m(D_1)}{m(B_r(z))} \geq c(N).
\end{equation}

Since $e(z)\leq c_1^2r^2$ by property \ref{main lemma 3} of $z$, the curve traversing $\gamma_{z,p}$ in reverse and then $\gamma_{z,q}$ is a $c_1r$-geodesic from $p$ to $q$. Fix $h^- \equiv h^-_{\rho(c_1r)^2}$ satisfying statement 4 of Theorem \ref{hess h integral bound} for the balls of radius $c_1r$ along this curve, where $\rho \in [\frac{1}{2},2]$. 

Since $z$ has low excess, by integral Abresch-Gromoll, there exists $B_{2r}(z)' \subseteq B_{2r}(z)$ so that
\begin{equation}\label{main lemma eq4}
 e(x) \leq c(N,\delta)r^2 \; \forall x \in B_{2r}(z)' \; \; \text{and}\; \; \frac{m(B_{2r}(z)')}{m(B_{2r}(z))} \geq 1-\frac{1}{2}V(1,10)^2.
\end{equation}

For all $s \in [0,\epsilon]$ and $(x,y)\in  X \times X$, define
\begin{equation}\label{main lemma eq5}
dt_1(s)(x,y) := \min\left\{r, \max\limits_{0 \leq \tau \leq s} |d(x,y) - d(\Psi_\tau(x), \Psi_\tau(y))|\right\}
\end{equation}
 and 
\begin{equation}\label{main lemma eq6}
U_1^s := \{(x,y) \in D_1 \times B_{2r}(z)'|dt_1(s)(x,y) < r\}. 
\end{equation}
Consider $\int_{D_1 \times B_{2r}(z)'} dt_1(s)(x,y) \, d(m \times m)(x,y)$ for $0 \leq s \leq \epsilon$. Since $r \leq \bar{r}_1 \leq \frac{\delta}{10}$, $\epsilon \leq \epsilon_1 \leq \frac{\delta}{10}$, and $t_0 \geq \delta$, $(\Psi_s)_{s \in [0,\epsilon]}$ is a local flow of $-\nabla d_p$ from both $D_1$ and $B_{2r}(z)'$. Therefore, $s \mapsto \int_{D_1 \times B_{2r}(z)'} dt_1(s)(x,y) \, d(m \times m)(x,y)$ is Lipschitz and
\begin{align}\label{main lemma eq7}
\begin{split}
	&\hspace{0.5cm}\frac{d}{ds} \int\limits_{D_1 \times B_{2r}(z)'} dt_1(s)(x,y) \, d(m \times m)(x,y)\\
	 &\leq \int\limits_{U_1^s} \bigg(|\nabla h^- - \nabla d_p|(\Psi_s(x))+ |\nabla h^- - \nabla d_p|(\Psi_s(y))\bigg)\, d(m \times m)(x,y) \\ 
&\hspace{2cm}+ \int\limits_{0}^{1}  \int\limits_{U_1^s} d(\Psi_s(x), \Psi_s(y))|\Hess h^-|_{\HS}(\tilde{\gamma}_{\Psi_s(x),\Psi_s(y)}(\tau)) \,  d(m \times m)(x,y) \, d\tau,
\end{split} 
\end{align}
for a.e. $s \in [0,\epsilon]$ by Proposition \ref{distortion bound}.

For any $s \in [0,\epsilon]$ and $(x,y) \in U_1^s$, 
\begin{enumerate}
	\item $d(x,y) < 3r$ since $D_1 \subseteq B_{r}(z)$ and $B_{2r}(z)' \subseteq B_{2r}(z)$;
	\item $d(\Psi_{s}(x), \Psi_{s}(z)) < 2r$ since $D_1 \subseteq A$ by definition \eqref{main lemma eq2};
	\item $dt_1(s)(x,y) < r$ by definition of $U_1^s$ \eqref{main lemma eq6}.
\end{enumerate}
Therefore, $\Psi_{s}(y) \in B_{6r}(\Psi_{s}(z))$ by triangle inequality and so $(\Psi_{s}, \Psi_{s})(U_1^s) \subseteq B_{\frac{c_1}{2}r}(\Psi_{s}(z)) \times B_{\frac{c_1}{2}r}(\Psi_{s}(z))$ since we assumed \hyperref[c1]{$c_1$}$> 100$. Therefore,
\begin{align}\label{main lemma eq8}
\begin{split}
	&\hspace{0.5cm}\int\limits_{0}^{1}  \int\limits_{U_1^s} d(\Psi_s(x), \Psi_s(y))|\Hess h^-|_{\HS}(\tilde{\gamma}_{\Psi_s(x),\Psi_s(y)}(\tau)) \,  d(m \times m)(x,y) \, d\tau\\
	&\leq c(N,\delta)\int\limits_{0}^{1}  \int\limits_{(\Psi_s, \Psi_s)(U_1^s)} d(x, y)|\Hess h^-|_{\HS}(\tilde{\gamma}_{x,y}(\tau)) \,  d(m \times m)(x,y) \, d\tau \; \; \text{, by Theorem \ref{-dp grad flow properties}, 2}\\
	&\leq c(N,\delta)rm(B_{\frac{c_1}{2}r}(\Psi_s(z)))\int\limits_{B_{c_1r}(\Psi_s(z))} |\Hess h^-|_{\HS}\,  dm \; \; \text{, by segment inequality \ref{segment inequality}}\\
	&\leq c(N,\delta)rm(B_{r}(z))^2 \fint\limits_{B_{c_1r}(\Psi_s(z))} |\Hess h^-|_{\HS} \, dm \; \; \text{, by Bishop-Gromov and property \ref{main lemma 1} of $z$}.
\end{split}
\end{align}
Integrating in $s \in [0,\epsilon]$, 
\begin{align}\label{main lemma eq9}
\begin{split}
	&\hspace{0.5cm} \int\limits_{0}^{\epsilon} \bigg( \int\limits_{0}^{1} \int\limits_{U_1^s} d(\Psi_s(x), \Psi_s(y))|\Hess h^-|_{\HS}(\tilde{\gamma}_{\Psi_s(x),\Psi_s(y)}(\tau)) \,  d(m \times m)(x,y)\, d\tau \bigg)\, ds\\
	&\leq crm(B_{r}(z))^2 \int\limits_{0}^{\epsilon} \fint\limits_{B_{c_1r}(\Psi_s(z))} |\Hess h^-|_{\HS} \, dm \, ds\\
	&\leq c(N,\delta)rm(B_{r}(z))^2 \sqrt{\epsilon},
\end{split}
\end{align}
where the last line follows from the definition of $h^-$, statement 4 of Theorem \ref{hess h integral bound}, and Cauchy-Schwarz.

By statement 3 of \ref{grad h geodesic bound}, the excess bound on the elements of $D_1$ \eqref{main lemma eq2} (and therefore also on the elements $\Psi_{s}(D_1)$), and Bishop-Gromov,
\begin{align}\label{main lemma eq10}
\begin{split}
	\int\limits_{0}^{\epsilon} \int\limits_{U^s_1} |\nabla h^- - \nabla d_p|(\Psi_s(x)) \, d(m \times m)(x,y)  \, ds \leq c(N,\delta)rm(B_r(z))^2\sqrt{\epsilon}.
\end{split}
\end{align}
Similarly by the excess bounds on the elements of $B_{2r}(z)'$ \eqref{main lemma eq4}, 
\begin{equation}\label{main lemma eq11}
\int\limits_{0}^{\epsilon} \int\limits_{U^s_1} |\nabla h^- - \nabla d_p|(\Psi_s(y))\, d(m \times m)(x,y) \, ds \leq c(N,\delta)rm(B_r(z))^2\sqrt{\epsilon}.
\end{equation}

Combining  \eqref{main lemma eq9} - \eqref{main lemma eq11} with the bound \eqref{main lemma eq7} on $\frac{d}{ds} \int\limits_{D_1 \times B_{2r}(z)'} dt_1(s)(x,y)$, we obtain
\begin{align}\label{main lemma eq12}
\begin{split}
\int\limits_{D_1 \times B_{2r}(z)'} dt_1(\epsilon)(x,y) \, d(m \times m)(x,y) & = \int_{0}^{\epsilon} [\frac{d}{ds} \int\limits_{D_1 \times B_{2r}(z)'} dt_1(s)(x,y) \, d(m \times m)(x,y)] \, ds\\ &\leq c(N,\delta)rm(B_r(z))^2\sqrt{\epsilon}.
\end{split}
\end{align}
Since $D_1$ takes a non-trivial portion of the measure of $B_r(z)$ by \eqref{main lemma eq3},
\begin{equation*}
	\fint\limits_{D_1} \int\limits_{B_{2r}(z)'} dt_1(\epsilon) (x,y) \, dm(y) \leq c(N,\delta)rm(B_{r}(z))\sqrt{\epsilon}.
\end{equation*}
In particular, there exists $z' \in D_1$ so that
\begin{equation*}
	\int\limits_{B_{2r}(z)'} dt_1(\epsilon) (z',y) \, dm(y) \leq crm(B_{r}(z))\sqrt{\epsilon}.
\end{equation*}
By definition of $D_1$, property \ref{main lemma 3'} of $z'$ is satisfied. 

We next check property \ref{main lemma 2'} is satisfied for the chosen $z'$ as well if $\epsilon$ is sufficiently small. Define $B_{r}(z')' := B_r(z') \cap B_{2r}(z)'$. By the previous estimate and Bishop-Gromov,
\begin{equation}\label{main lemma eq13}
	\int\limits_{B_r(z')'} dt_1(\epsilon)(z',y) \, dm(y) \leq crm(B_{r}(z))\sqrt{\epsilon} \leq c(N,\delta)rm(B_r(z')) \sqrt{\epsilon}.
\end{equation}
Using this, we bound $\epsilon_1$ sufficiently small depending on $N$ and $\delta$ so that for $\epsilon \leq \epsilon_1$, 
\begin{equation}\label{main lemma eq14}
	\int\limits_{B_r(z')'} dt_1(\epsilon)(z',y) \, dm(y) \leq \frac{1}{4}rm(B_r(z'))V(1,10).
\end{equation}
For example, $\epsilon_1 \leq  (\frac{V(1,10)}{4c})^2$ suffices, where $c$ is the last one from \eqref{main lemma eq13}. Moreover, $B_{2r}(z)'$ takes significant mass in $B_{2r}(z)$ from \eqref{main lemma eq4} and so by Bishop-Gromov,
\begin{equation}\label{main lemma eq15}
\frac{m(B_{r}(z')')}{m(B_{r}(z'))}  \geq  1 - \bigg(\frac{m(B_{2r}(z)')}{m(B_{2r}(z))}\frac{m(B_{2r}(z))}{m(B_{r}(z'))}\bigg) \geq 1 - \bigg(\frac{1}{2}\frac{(V(1,10)^2)}{V(1,3)}\bigg) \geq 1 - \frac{1}{2}V(1,10).
\end{equation}
Combining \eqref{main lemma eq14} and \eqref{main lemma eq15}, we conclude there exists $A' \subseteq B_r(z')'$ so that
\begin{equation}\label{main lemma eq16}
\frac{m(A')}{m(B_r(z'))} \geq 1 - V(1,10) \; \; \text{and} \; \; dt_1(\epsilon)(z',y) \leq \frac{1}{2}r \; \forall y \in A'.
\end{equation}
The latter implies $\Psi_{s}(A') \subseteq B_{\frac{3r}{2}}(\Psi_s(z'))$ for any $s \leq \epsilon$ and so property \ref{main lemma 2'} of $z'$ is satisfied. 

This also gives one direction of the bound in property \ref{main lemma 1'} for $z'$. For each $s \leq \epsilon$, 
\begin{align*}
	\frac{m(B_{r}(\Psi_{s}(z')))}{m(B_{r}(z'))} &\geq V(1, \frac{3}{2})\frac{m(B_{\frac{3r}{2}}(\Psi_{s}(z')))}{m(B_{r}(z'))} \; \;  \text{, by Bishop-Gromov}\\
								&\geq V(1, \frac{3}{2})\frac{m(\Psi_{s}(A'))}{m(B_{r}(z'))}\\
								&\geq V(1, \frac{3}{2})(1+c(N,\delta)s)^{-N}\frac{m(A')}{m(B_{r}(z'))} \; \;\text{, by Theorem \ref{-dp grad flow properties}, 2}\\
								&\geq V(1, \frac{3}{2})(1+c(N,\delta)s)^{-N}(1-V(1,10)).
\end{align*}
We bound $\epsilon_1$ sufficiently small depending on $N$ and $\delta$ so that for $s \leq \epsilon \leq \epsilon_1$, the last line is greater than $2V(1,100)$.

The other direction of the bound in property \ref{main lemma 1'} of z' will be proved similarly by sending a sufficiently large portion of $B_{r}(\Psi_{s}(z'))$ close to $z'$ (in fact $z$) using a flow which does not decrease measure significantly and then using Bishop-Gromov. To do this, we first use the RLF associated to $-\nabla h^-$ to send a portion of $B_{r}(z')$ close to $\Psi_{s}(z')$. We then use the inverse flow  (i.e. the RLF associated to $\nabla h^-$) on the image of that portion to make sure a large enough portion of $B_{r}(\Psi_{s}(z'))$ indeed ends up close to $z'$ under the inverse flow. 

$|\nabla h_0^-| \in L^{\infty}(m)$ by \eqref{grad h ae bound eq2} and so $|\nabla h^-|, \Delta h^- \in L^{\infty}(m)$ by Bakry-Ledoux estimate \ref{Bakry-Ledoux} and $h^- \in W^{2,2}(X)$ by Corollary \ref{improved Bochner inequality cor} of the improved Bochner inequality. Therefore, the time-independent vector fields $-\nabla h^-$ and $\nabla h^-$ are bounded and satisfy the conditions of the existence and uniqueness of RLFs Theorem \ref{RLF existence}. Let $(\tilde{\Psi}_t)_{t \in [0,1]}$ and $(\tilde{\Psi}_{-t})_{t \in [0,1]}$ be the associated RLFs of $-\nabla h^-$ and $\nabla h^-$ for $t \in [0,1]$ respectively. The choice of notation is due to Proposition \ref{RLF inverse}, which says $\tilde{\Psi}_{-t}$ and $\tilde{\Psi}_{t}$ are $m$-a.e. inverses of each other. 

Since $e(z') \leq c_1^2r^2$ and $m(A') \geq (1 - V(1,10))m(B_r(z'))$, integral Abresch-Gromoll gives $A'' \subseteq A'$ so that
\begin{equation}\label{main lemma eq2.1}
	e(x) \leq c(N,\delta)r^2 \; \forall x \in A'' \; \; {and} \; \; \frac{m(A'')}{m(B_{r}(z'))} \geq 1 - 2V(1,10).
\end{equation} 

For all $s \in [0,\epsilon]$ and $(x,y)\in X \times X$, define
\begin{equation}\label{main lemma eq2.2}
dt_2(s)(x,y) := \min\left\{r, \max\limits_{0 \leq \tau \leq s} |d(x,y) - d(\Psi_\tau(x), \tilde{\Psi}_\tau(y))|\right\}
\end{equation}
and 
\begin{equation}\label{main lemma eq2.3}
U_2^s := \{(x,y) \in A'' \times B_{r}(z')|dt_2(s)(x,y) < r\}. 
\end{equation}
Consider $\int_{A'' \times B_{r}(z')} dt_2(s)(x,y) \, d(m \times m)(x,y)$ for $0 \leq s \leq \epsilon$. By Proposition \ref{distortion bound}, for a.e. $s \in [0,\epsilon]$,
\begin{align}\label{main lemma eq2.4}
\begin{split}
	&\hspace{0.5cm}\frac{d}{ds} \int\limits_{A'' \times B_{r}(z')} dt_2(s)(x,y) \, d(m \times m)(x,y)\\
	 &\leq \int\limits_{U_2^s} |\nabla h^- - \nabla d_p|(\Psi_s(x)) \, d(m \times m)(x,y)\\
&\hspace{2cm} + \int\limits_{0}^{1} \int\limits_{U_2^s} d(\Psi_s(x), \tilde{\Psi}_s(y))|\Hess h^-|_{\HS}(\tilde{\gamma}_{\Psi_s(x),\tilde{\Psi}_s(y)}(\tau)) \,  d(m \times m)(x,y) \, d\tau.
\end{split} 
\end{align}

For any $s \in [0, \epsilon]$ and $(x,y) \in U_2^s$, 
\begin{enumerate}
	\item $d(x,y) < 2r$ since $A'' \subseteq B_{r}(z')$;
	\item $d(\Psi_{s}(x), \Psi_{s}(z)) \leq d(\Psi_{s}(x), \Psi_{s}(z'))+d(\Psi_{s}(z'), \Psi_{s}(z))< \frac{7}{2}r$ by definition of $A'$ \eqref{main lemma eq16} and $z' \in D_1 \subseteq A$ \eqref{main lemma eq2};
	\item $dt_2(s)(x,y) < r$ by definition of $U_2^s$ \eqref{main lemma eq2.4}.
\end{enumerate}
Therefore, $\tilde{\Psi}_{s}(y) \in B_{\frac{13r}{2}}(\Psi_{s}(z))$ by triangle inequality and so $(\Psi_{s}, \tilde{\Psi}_{s})(U_2^s) \subseteq B_{\frac{c_1}{2}r}(\Psi_{s}(z)) \times B_{\frac{c_1}{2}r}(\Psi_{s}(z))$ since $c_1 > 100$. Therefore,
\begin{align}\label{main lemma eq2.5}
\begin{split}
	&\hspace{0.5cm}\int\limits_{0}^{1} \int\limits_{U_2^s} d(\Psi_s(x), \tilde{\Psi}_s(y))|\Hess h^-|_{\HS}(\tilde{\gamma}_{\Psi_s(x),\tilde{\Psi}_s(y)}(\tau)) \,  d(m \times m)(x,y) \, d\tau\\
	&\leq c(N,\delta)\int\limits_{0}^{1} \int\limits_{(\Psi_s, \tilde{\Psi}_s)(U_2^s)} d(x, y)|\Hess h^-|_{\HS}(\tilde{\gamma}_{x,y}(\tau)) \,  d(m \times m)(x,y) \, d\tau \text{ , by \ref{-dp grad flow properties} 2, \ref{lap cutoffed bound} and \ref{RLF existence} \eqref{RLF existence eq1}}\\
	&\leq c(N,\delta)rm(B_{\frac{c_1}{2}r}(\Psi_s(z)))\int\limits_{B_{c_1r}(\Psi_s(z))} |\Hess h^-|_{\HS}\,  dm \; \; \text{, by segment inequality \ref{segment inequality}}\\
	&\leq c(N,\delta)rm(B_{r}(z))^2 \fint\limits_{B_{c_1r}(\Psi_s(z))} |\Hess h^-|_{\HS} \, dm \; \; \text{, by Bishop-Gromov and property \ref{main lemma 1} of $z$}.
\end{split}
\end{align}
Integrating in $s \in [0,\epsilon]$, 
\begin{align}\label{main lemma eq2.6}
\begin{split}
	&\hspace{0.5cm} \int\limits_{0}^{\epsilon} \bigg(\int\limits_{0}^{1}  \int\limits_{U_2^s} d(\Psi_s(x), \tilde{\Psi}_s(y))|\Hess h^-|_{\HS}(\tilde{\gamma}_{\Psi_s(x),\tilde{\Psi}_s(y)}(\tau)) \,  d(m \times m)(x,y) \, d\tau \bigg)\, ds\\
	& \leq crm(B_{r}(z))^2 \int\limits_{0}^{\epsilon} \fint\limits_{B_{c_1r}(\Psi_s(z))} |\Hess h^-|_{\HS} \, dm \, ds\\
	&\leq c(N,\delta)rm(B_{r}(z))^2 \sqrt{\epsilon},
\end{split}
\end{align}
where the last line follows from the definition of $h^-$, statement 4 of Theorem \ref{hess h integral bound}, and Cauchy-Schwarz.

By statement 3 of \ref{grad h geodesic bound}, the excess bound on the elements of $A''$ \eqref{main lemma eq2.1}, and Bishop-Gromov,
\begin{align}\label{main lemma eq2.7}
\begin{split}
	\int\limits_{0}^{\epsilon} \int\limits_{U^s_2} |\nabla h^- - \nabla d_p|(\Psi_s(x)) \, d(m \times m)(x,y)  \, ds \leq c(N,\delta)rm(B_r(z))^2\sqrt{\epsilon}.
\end{split}
\end{align}

Combining  \eqref{main lemma eq2.6}, \eqref{main lemma eq2.7} with the bound \eqref{main lemma eq2.4} we obtain,
\begin{align}\label{main lemma eq2.8}
\begin{split}
\int\limits_{A'' \times B_{r}(z')} dt_2(\epsilon)(x,y) \, d(m \times m)(x,y) & = \int_{0}^{\epsilon} [\frac{d}{ds} \int\limits_{A'' \times B_{r}(z')} dt_2(s)(x,y) \, d(m \times m)(x,y)] \, ds\\ &\leq c(N,\delta)rm(B_r(z))^2\sqrt{\epsilon}.
\end{split}
\end{align}
$A''$ is comparable in measure to $B_r(z')$ by \eqref{main lemma eq2.1} and hence also to $B_r(z)$ by Bishop-Gromov. Therefore, there exists $z_1 \in A''$ so that
\begin{equation*}
	\int\limits_{B_{r}(z')} dt_2(\epsilon) (z_1,y) \, dm(y) \leq c(N, \delta)rm(B_{r}(z))\sqrt{\epsilon}.
\end{equation*}
By Bishop-Gromov, $\frac{m(B_{r}(z))}{m(B_{r}(z'))} \leq c(N)$ and so
\begin{equation*}
	\fint\limits_{B_{r}(z')} dt_2(\epsilon) (z_1,y) \, dm(y) \leq c(N, \delta)r\sqrt{\epsilon}.
\end{equation*}
Using this, we bound $\epsilon_1$ sufficiently small depending on $N$ and $\delta$ so that there exists $D_2 \subseteq B_{r}(z')$ with
\begin{equation}\label{main lemma eq2.9}
\frac{m(D_2)}{m(B_r(z'))} \geq 1 - V(1,10) \; \; \text{and} \; \; dt_2(\epsilon)(z_1,y) \leq \frac{r}{2} \; \forall y \in D_2.
\end{equation}
For each $y \in D_2$ and $s \in [0,\epsilon]$,
\begin{enumerate}
	\item $d(z_1,y) < 2r$;
	\item $d(\Psi_{s}(z_1), \Psi_{s}(z)) \leq d(\Psi_{s}(z_1), \Psi_{s}(z'))+d(\Psi_{s}(z'), \Psi_{s}(z))< \frac{7}{2}r$;
	\item $dt_2(\epsilon)(z_1,y) \leq \frac{r}{2}$,
\end{enumerate}
and so 
\begin{equation}\label{main lemma eq2.9.5}
	\tilde{\Psi}_{s}(D_2) \subseteq B_{4r}(\Psi_s(z')) \subseteq B_{6r}(\Psi_s(z)).
\end{equation}
Moreover, $\tilde{\Psi}_{s}(D_2)$ is non-trivial in measure compared to $B_{r}(z)$.
\begin{align}\label{main lemma eq2.10}
\begin{split}
	\frac{m(\tilde{\Psi}_{s}(D_2))}{m(B_{r}(z))} &\geq e^{-c(N,\delta)\frac{\delta}{10}}\frac{m(D_2)}{m(B_{r}(z))} \; \; \text{ , by \ref{lap cutoffed bound}, \ref{RLF existence} \eqref{RLF existence eq1}, and $\epsilon_1 \leq \frac{\delta}{10}$}\\
	&\geq c(N,\delta) \; \; \text{ , by definition of $D_2$ and Bishop-Gromov.}
\end{split}
\end{align}

We will now flow $\tilde{\Psi}_{s}(D_2)$ back by $\tilde{\Psi}_{-t}$ and use that to control the flow of $B_{r}(\Psi_{s}(z'))$ under $\tilde{\Psi}_{-t}$. Fix $s \in [0,\epsilon]$. By Proposition \ref{RLF inverse}, we may assume, up to choosing a full measure subset, that $D_2$ satisfies
\begin{equation}\label{main lemma eq3.0}
\tPsi_{-t}(\tPsi_{s}(x)) = \tPsi_{s-t}(x) \; \forall t \in [0,s] \; \text{and} \; \forall x \in D_2.
\end{equation}

For all $t \in [0,s]$ and $(x,y)\in X \times X$, define
\begin{equation}\label{main lemma eq3.1}
dt_{3}(t)(x,y) := \min\left\{r, \max\limits_{0 \leq \tau \leq t} |d(x,y) - d(\tilde{\Psi}_{-\tau}(x), \tilde{\Psi}_{-\tau}(y))|\right\}
\end{equation}
and 
\begin{equation}\label{main lemma eq3.2}
U_3^t := \{(x,y) \in \tilde{\Psi}_{s}(D_2) \times B_{r}(\Psi_{s}(z'))| dt_{3}(t)(x,y) < r \}. 
\end{equation}
We note that $U_3^t$ implicitly depends on $s$.
Consider $\int_{\tilde{\Psi}_{s}(D_2) \times B_{r}(\Psi_{s}(z'))} dt_3(t)(x,y) \, d(m \times m)(x,y)$ for $0 \leq t \leq s$. By proposition \ref{distortion bound}, for a.e. $t \in [0,s]$,
\begin{align}\label{main lemma eq3.3}
\begin{split}
	&\hspace{0.5cm}\frac{d}{dt} \int\limits_{\tilde{\Psi}_{s}(D_2) \times B_{r}(\Psi_{s}(z'))} dt_3(t)(x,y) \, d(m \times m)(x,y)\\
&\leq \int\limits_{0}^{1}  \int\limits_{U_3^t} d(\tPsi_{-t}(x), \tPsi_{-t}(y))|\Hess h^-|_{\HS}(\tilde{\gamma}_{\tPsi_{-t}(x),\tPsi_{-t}(y)}(\tau)) \,  d(m \times m)(x,y) \, d\tau,
\end{split} 
\end{align}

For any $t \in [0,s]$, $\omega \in [0,t]$ and $(x,y) \in U_3^t$, 
\begin{enumerate}
	\item $d(x,y) < 5r$ since $\tilde{\Psi}_{s}(D_2) \subseteq B_{4r}(\Psi_s(z'))$ by \eqref{main lemma eq2.9.5};
	\item $d(\tilde{\Psi}_{-\omega}(x), \Psi_{s-\omega}(z)) = d(\tilde{\Psi}_{s-\omega}(x'), \Psi_{s-\omega}(z)) < 6r$ for some $x' \in D_2$ by \eqref{main lemma eq2.9.5} and \eqref{main lemma eq3.0};
	\item $dt_3(t)(x,y) < r$ by definition of $U_3^t$  \eqref{main lemma eq3.2}.
\end{enumerate}
Hence,
\begin{equation}\label{main lemma eq3.3.3}
\tilde{\Psi}_{-\omega}(y) \in B_{12r}(\Psi_{s-\omega}(z))
\end{equation}
by triangle inequality. Therefore, $(\tilde{\Psi}_{-\omega}, \tilde{\Psi}_{-\omega})(U_3^t) \subseteq B_{\frac{c_1}{2}r}(\Psi_{s-\omega}(z)) \times B_{\frac{c_1}{2}r}(\Psi_{s-\omega}(z))$ for all $\omega \in [0, t]$ since \hyperref[c1]{$c_1$} $>100$. For any $(x,y) \in U_3^t$,
\begin{align}\label{main lemma eq3.3.5}
\begin{split}
	\Delta h^-(\tilde{\Psi}_{-\omega}(x)) &= \Delta h^+(\tilde{\Psi}_{-\omega}(x)) + \Delta \hat{e}(\tilde{\Psi}_{-\omega}(x))\\
	&\geq -c(N,\delta) \; \; \text{ , by Lemma \ref{lap cutoffed bound} and Lemma \ref{excess pw bound} \ref{excess pw bound 3}, using $e(z) \leq c_1^2r^2$},
\end{split}
\end{align}
where $h^+$, $\hat{e}$ are heat flow approximations of $h^+_0$ and $e_0$ respectively up to the same time as $h^-$. We have the same bound for $\Delta h^-(\tilde{\Psi}_{-\omega}(y))$. Therefore,
\begin{align}\label{main lemma eq3.4}
\begin{split}
	&\hspace{0.5cm}\int\limits_{0}^{1} \int\limits_{U_3^t} d(\tilde{\Psi}_{-t}(x), \tilde{\Psi}_{-t}(y))|\Hess h^-|_{\HS}(\tilde{\gamma}_{\tilde{\Psi}_{-t}(x), \tilde{\Psi}_{-t}(y)}(\tau)) \,  d(m \times m)(x,y) \, d\tau\\
	&\leq c(N,\delta)\int\limits_{0}^{1} \int\limits_{(\tilde{\Psi}_{-t}, \tilde{\Psi}_{-t})(U_3^t)} d(x, y)|\Hess h^-|_{\HS}(\tilde{\gamma}_{x,y}(\tau)) \,  d(m \times m)(x,y) \, d\tau \text{, by \eqref{main lemma eq3.3.5} and Remark \ref{RLF local volume bound}}\\
	&\leq c(N,\delta)rm(B_{\frac{c_1}{2}r}(\Psi_{s-t}(z)))\int\limits_{B_{c_1r}(\Psi_{s-t}(z))} |\Hess h^-|_{\HS}\,  dm \; \; \text{, by segment inequality \ref{segment inequality}}\\
	&\leq c(N,\delta)rm(B_{r}(z))^2 \fint\limits_{B_{c_1r}(\Psi_{s-t}(z))} |\Hess h^-|_{\HS} \, dm \; \; \text{, by Bishop-Gromov and property \ref{main lemma 1} of $z$}.
\end{split}
\end{align}
Integrating in $t \in [0,s]$, 
\begin{align}\label{main lemma eq3.5}
\begin{split}
	&\hspace{0.5cm} \int\limits_{0}^{s} \bigg(\int\limits_{0}^{1}  \int\limits_{U_3^t} d(\tilde{\Psi}_{-t}(x), \tilde{\Psi}_{-t}(y))|\Hess h^-|_{\HS}(\tilde{\gamma}_{\tilde{\Psi}_{-t}(x), \tilde{\Psi}_{-t}(y)}(\tau)) \,  d(m \times m)(x,y) \, d\tau \bigg)\, dt\\
	& \leq crm(B_{r}(z))^2 \int\limits_{0}^{s} \fint\limits_{B_{c_1r}(\Psi_{s-t}(z))} |\Hess h^-|_{\HS} \, dm \, ds\\
	&\leq c(N,\delta)rm(B_{r}(z))^2 \sqrt{s},
\end{split}
\end{align}
where the last line follows from the definition of $h^-$, statement 4 of Theorem \ref{hess h integral bound}, and Cauchy-Schwarz.
Therefore, 
\begin{align}\label{main lemma eq3.6}
\begin{split}
\int\limits_{\tilde{\Psi}_{s}(D_2) \times B_{r}(\Psi_{s}(z'))} dt_3(s)(x,y) \, d(m \times m)(x,y) & = \int_{0}^{s} [\frac{d}{dt} \int\limits_{\tilde{\Psi}_{s}(D_2) \times B_{r}(\Psi_{s}(z'))} dt_3(t)(x,y) \, d(m \times m)(x,y)] \, dt\\ &\leq c(N,\delta)rm(B_r(z))^2\sqrt{s}.
\end{split}
\end{align}

We previously computed that $\tilde{\Psi}_{s}(D_2)$ is non-trivial in measure compared to $B_{r}(z)$ in \eqref{main lemma eq2.10} and so there exists $z_2 \in \tilde{\Psi}_s(D_2)$ with
\begin{equation*}
	\int\limits_{B_{r}(\Psi_{s}(z'))} dt_3(s) (z_2,y) \, dm(y) \leq c(N, \delta)rm(B_{r}(z))\sqrt{s}.
\end{equation*}
By Bishop-Gromov, $\frac{m(B_{r}(\Psi_{s}(z')))}{m(B_{r}(\Psi_{s}(z)))} \geq c(N)$ and so by property \ref{main lemma 1} of $z$, 
\begin{equation*}
	\fint\limits_{B_{r}(\Psi_{s}(z'))} dt_3(s) (z_2,y) \, dm(y) \leq c(N, \delta)r\sqrt{s}.
\end{equation*}
Using this, we bound $\epsilon_1$ sufficiently small depending on $N$ and $\delta$ so there exists $D_3 \subseteq B_{r}(\Psi_{s}(z'))$ with
\begin{equation}\label{main lemma eq3.8}
\frac{m(D_3)}{m(B_r(\Psi_s(z')))} \geq 1 - V(1,10) \; \; \text{and} \; \; dt_3(s)(z_2,y) \leq \frac{1}{2}r \; \forall y \in D_3.
\end{equation}
For each $y \in D_3$, 
\begin{enumerate}
	\item $d(z_2,y) < 5r$ by \eqref{main lemma eq2.9.5};
	\item $d(\tilde{\Psi}_{-s}(z_2), z') = d(z_2', z') < r$, where $z_2' \in D_2 \subseteq B_{r}(z')$ is so that $\tPsi_s(z_2') = z_2$;
	\item $dt_3(s)(z_2,y) \leq \frac{1}{2}r$,
\end{enumerate}
and so $\tilde{\Psi}_{-s}(D_3) \subseteq B_{7r}(z')$. Notice also for the next calculation that $\tilde{\Psi}_{-t}(D_3) \subseteq B_{12r}(z)$ for any $t \in [0,s]$ by the caulations of \eqref{main lemma eq3.3.3}. Therefore, one has the same lower bound for the Laplacian of $h^-$ on $\tilde{\Psi}_{-t}(D_3)$ as in \eqref{main lemma eq3.3.5}.

We estimate
\begin{align*}
	&\hspace{0.5cm}\frac{m(B_{r}(\Psi_s(z')))}{m(B_{r}(z'))}\leq \frac{1}{V(1,7)}\frac{m(B_{r}(\Psi_{s}(z')))}{m(B_{7r}(z'))} \; \; \text{, by Bishop-Gromov}\\
								&\leq  \frac{1}{V(1,7)}\frac{1}{1-V(1,10)}\frac{m(D_3)}{m(B_{7r}(z'))} \text{, by property \eqref{main lemma eq3.8} of $D_3$}\\
								&\leq  \frac{1}{V(1,7)}\frac{1}{1-V(1,10)}e^{c(N,\delta)s}\frac{m(\tilde{\Psi}_{-s}(D_3))}{m(B_{7r}(z'))}\; \;\text{, by Remark \ref{RLF local volume bound}}
\\
	&\leq \frac{1}{V(1,7)}\frac{1}{1-V(1,10)}e^{cs}.
\end{align*}
We bound $\epsilon_1$ sufficiently small depending on $N$ and $\delta$ so that for $s \leq \epsilon \leq \epsilon_1$, the last line is less than $\frac{1}{2V(1,100)}$.

All this imply that if \ref{main lemma 1}, \ref{main lemma 2} and \ref{main lemma 3} hold for $z \in B_r(\gamma(t_0))$ and $\epsilon \leq \epsilon_1(N, \delta)$, then there exists $z' \in B_r(\gamma(t_0))$ satisfying \ref{main lemma 1'},  \ref{main lemma 2'} and \ref{main lemma 3'} for the same $\epsilon$. By Bishop-Gromov volume comparison and the fact that $\Psi$ is defined using unit speed geodesics, there is some $T(N,\delta,r)>0$ so that \ref{main lemma 1} and \ref{main lemma 2} hold for $z'$, $A'$, and $0 \leq s \leq \epsilon+T$. Combining this with the existence of some $z$ and $\epsilon$ depending on $N, \delta$, and $r$ satisfying \ref{main lemma 1}, \ref{main lemma 2} and \ref{main lemma 3} mentioned at the beginning of the proof, we conclude there exists $z \in B_{r}(\gamma(t_0))$ so that \ref{main lemma 1}, \ref{main lemma 2} and \ref{main lemma 3} hold for $\epsilon = \epsilon_1$. 
\end{proof}

\subsection{Construction of limit geodesics}\label{subsection 5.2}

In this section we construct a geodesics between $p$ and $q$ which has properties 1 and 2 of Lemma \ref{main lemma} on the geodesic itself. Roughly, this means we construct $\bar{\gamma}$ so that small balls centered on $\bar{\gamma}$ between $\delta$ and $1-\delta$ stay close to the geodesic itself for a short amount time under the flows $\Psi$ and $\Phi$.

We start by showing that for any fixed scale $r$ we can find points $z$ arbitrarily close to $\gamma(t_0)$ which have the properties 1 - 3 as in Lemma \ref{main lemma}. In order to do this we will prove the following lemma which will form our induction step.

\begin{lem}\label{induction lemma}
There exists $\epsilon_2(N,\delta)>0$ and $\bar{r}_2(N,\delta) > 0$ so that for any $r \leq \bar{r}_2$, $\delta \leq t_0 \leq 1- \delta$, if there exists $z \in B_{r}(\gamma(t_0))$ and $\epsilon \leq \epsilon_2$ so that
\begin{enumerate}
	\item \label{induction lemma 1} $V(1, 100) \leq \frac{m(B_r(\Psi_s(z)))}{m(B_r(z))} \leq \frac{1}{V(1,100)}$ for any $s \leq \epsilon$;
	\item \label{induction lemma 2}There exists $A_r \subseteq B_{r}(z)$ with $m(A_r) \geq (1-V(1,10))m(B_r(z))$ and $\Psi_{s}(A_r) \subseteq B_{2r}(\Psi_{s}(z))$ for any $s \leq \epsilon$; 
	\item \label{induction lemma 3}$e(z) \leq c_1^2r^2$,
\end{enumerate}
then for the same $z$ and any $r' \in [4r, 16r]$,
\begin{enumerate}[label=\roman*., ref=\roman*]
	\item  \label{induction lemma i} $V(1, 100) \leq \frac{m(B_{r'}(\Psi_s(z)))}{m(B_{r'}(z))} \leq \frac{1}{V(1,100)}$ for any $s \leq \epsilon$;
	\item  \label{induction lemma ii} There exists $A_{r'} \subseteq B_{r'}(z)$ with $m(A_{r'}) \geq (1-V(1,10))m(B_{r'}(z))$ and $\Psi_{s}(A_{r'}) \subseteq B_{2r'}(\Psi_{s}(z))$ for any $s \leq \epsilon$;
	\item  \label{induction lemma iii} $e(z) \leq c_1^2r^2 < c_1^2(r')^2$.
\end{enumerate}
\end{lem}
\begin{proof}
Fix $\delta \leq t_0 \leq 1-\delta$. We assume $\bar{r}_2 \leq \frac{\delta}{1000}$ and $\epsilon_2 \leq \frac{\delta}{1000}$ to begin with but will impose more bounds on both depending on $N$ and $\delta$ as the proof continues. We will not keep track of $\bar{r}_2$ for the sake of brevity. Fix a scale $r \leq \bar{r}_2$ and $r' \in [4r,16r]$. Fix $z$, $A_{r}$, and $\epsilon \leq \epsilon_2$ so that \ref{induction lemma 1}, \ref{induction lemma 2} and \ref{induction lemma 3} hold. 

Since $e(z) \leq c_1^2r^2$, by integral Abresch-Gromoll, there exists $B_{r'}(z)' \subseteq B_{r'}(z)$ so that
\begin{equation}\label{induction lemma eq1}
 e(x) \leq c(N,\delta)r^2 \; \forall x \in B_{r'}(z)' \; \; \text{and}\; \; \frac{m(B_{r'}(z)')}{m(B_{r'}(z))} \geq 1-\frac{1}{2}V(1,10).
\end{equation}
Similarly, there exists $A' \subseteq A_r$ so that
\begin{equation}\label{induction lemma eq2}
 e(x) \leq c(N,\delta)r^2 \; \forall x \in A' \; \; \text{and}\; \; \frac{m(A')}{m(B_{r}(z))} \geq 1-2V(1,10).
\end{equation}

As in the previous lemma, the curve traversing $\gamma_{z,p}$ in reverse and then $\gamma_{z,q}$ is a $c_1r$-geodesic from $p$ to $q$. Fix $h^- \equiv h^-_{\rho(c_1r)^2}$ satisfying statement 4 of Theorem \ref{hess h integral bound} for the balls of radius $c_1r$ along this curve, where $\rho \in [\frac{1}{2},2]$. 

For all $s \in [0, \epsilon]$ and $(x,y) \in X\times X$, define 
\begin{equation}\label{induction lemma eq3}
dt_1(s)(x,y) := \min\left\{r, \max\limits_{0 \leq \tau \leq s} |d(x,y) - d(\Psi_\tau(x), \Psi_\tau(y))|\right\}
\end{equation}
 and 
\begin{equation}\label{induction lemma eq4}
U_1^s := \{(x,y) \in A' \times B_{r'}(z)'|dt_1(s)(x,y) < r\}. 
\end{equation}
Consider $\int_{ A' \times B_{r'}(z)'} dt_1(s)(x,y) \, d(m \times m)(x,y)$ for $0 \leq s \leq \epsilon$. For any $s \in [0,\epsilon]$ and $(x,y) \in U_1^s$,
\begin{enumerate}
	\item $d(x,y) < r'+r$;
	\item $d(\Psi_{s}(x), \Psi_{s}(z)) < 2r$ by definition of $A' \subseteq A_r$;
	\item $dt_1(s)(x,y) < r$.
\end{enumerate}
Therefore, $\Psi_{s}(y) \in B_{r'+4r}(\Psi_{s}(z)) \subseteq B_{\frac{c_1}{2}r}(\Psi_{s}(z))$ since \hyperref[c1]{$c_1$}$>100$. 

Using exactly the same type of computation as the first part of the proof of the main lemma, by interpolating between the two local flows of $\Psi_{s}$ from $A'$ and $B_{r'}(z)'$ with $\nabla h^-$, we obtain
\begin{align}\label{induction lemma eq5}
\begin{split}
\int\limits_{A' \times B_{r'}(z)'} dt_1(\epsilon)(x,y) \, d(m \times m)(x,y) \leq c(N,\delta)rm(B_r(z))^2\sqrt{\epsilon}.
\end{split}
\end{align}
Since $A'$ takes a significant portion of the measure of $B_r(z)$ by \eqref{induction lemma eq2},
\begin{equation*}
	\fint\limits_{A'} \int\limits_{B_{r'}(z)'} dt_1(\epsilon) (x,y) \, dm(y) \leq c(N,\delta)rm(B_{r}(z))\sqrt{\epsilon},
\end{equation*}
and so there exists $z' \in A'$ so that
\begin{equation*}
	\int\limits_{B_{r'}(z)'} dt_1(\epsilon) (z',y) \, dm(y) \leq crm(B_{r}(z))\sqrt{\epsilon}.
\end{equation*}
Therefore, by the fact that $\frac{m(B_{r'}(z)')}{m(B_{r'}(z))} \geq 1-\frac{1}{2}V(1,10)$ and Bishop-Gromov, 
\begin{equation*}
	\fint\limits_{B_{r'}(z)'} dt_1(\epsilon) (z',y) \, dm(y) \leq c(N,\delta)r\sqrt{\epsilon}.
\end{equation*}
Using this, we bound $\epsilon_2$ sufficiently small depending on $N$ and $\delta$ so that there exists $A_{r'} \subseteq B_{r'}(z)'$ with
\begin{equation}\label{induction lemma eq6}
	\frac{m(A_{r'})}{m(B_{r'}(z))} \geq 1 - V(1,10) \; \; \text{and} \; \; dt_1(\epsilon)(z',y) \leq \frac{r}{2} \; \forall y \in A_{r'}.
\end{equation}
For each $y \in A_{r'}$ and $s \in [0,\epsilon]$,
\begin{enumerate}
	\item $d(z',y) < r'+r$;
	\item $d(\Psi_{s}(z'), \Psi_{s}(z)) < 2r$;
	\item $dt_1(\epsilon)(z',y) \leq \frac{r}{2}$,
\end{enumerate}
and so 
\begin{equation}\label{induction lemma eq7}
	\Psi_{s}(A_{r'}) \subseteq B_{r'+\frac{7r}{2}}(\Psi_s(z)) \subseteq B_{2r'}(\Psi_s(z)).
\end{equation}
This proves property \ref{induction lemma ii}.

This also gives one direction of the bound in property \ref{induction lemma i}. For each $s \leq \epsilon$, 
\begin{align*}
	\frac{m(B_{r'}(\Psi_{s}(z)))}{m(B_{r'}(z))} &\geq V(1, 2)\frac{m(B_{2r'}(\Psi_{s}(z)))}{m(B_{r'}(z'))} \; \;  \text{, by Bishop-Gromov}\\
								&\geq V(1,2)\frac{m(\Psi_{s}(A_{r'}))}{m(B_{r'}(z'))}\\
								&\geq V(1,2)(1+c(N,\delta)s)^{-N}\frac{m(A_{r'})}{m(B_{r'}(z'))} \; \;\text{, by Theorem \ref{-dp grad flow properties}, 2}\\
								&\geq V(1,2)(1+c(N,\delta)s)^{-N}(1-V(1,10)).
\end{align*}
We bound $\epsilon_2$ sufficiently small depending on $N$ and $\delta$ so that for $s \leq \epsilon \leq \epsilon_2$, the last line is greater than $V(1,100)$.

To obtain the other direction of the bound in property \ref{induction lemma i}, we employ the same strategy as the proof of the main lemma as well. Let $(\tPsi_t)_{t \in [0,1]}$ and $(\tPsi_{-t})_{t \in [0,1]}$ be the RLFs of the time-independent vector fields $-\nabla h^-$ and $\nabla h^-$ respectively as before. 

For all $s \in [0, \epsilon]$ and $(x,y) \in X \times X$, define 
\begin{equation}\label{induction lemma eq2.1}
dt_2(s)(x,y) := \min\left\{r, \max\limits_{0 \leq \tau \leq s} |d(x,y) - d(\Psi_\tau(x), \tPsi_\tau(y))|\right\}
\end{equation}
and 
\begin{equation}\label{indctuion lemma eq2.2}
U_2^s := \{(x,y) \in A' \times B_{r}(z)|dt_2(s)(x,y) < r\}. 
\end{equation}
Consider $\int_{A' \times B_{r}(z)} dt_2(s)(x,y) \, d(m \times m)(x,y)$ for $0 \leq s \leq \epsilon$. For any $s \in [0,\epsilon]$ and $(x,y) \in U_2^s$,
\begin{enumerate}
	\item $d(x,y) < 2r$;
	\item $d(\Psi_{s}(x), \Psi_{s}(z)) < 2r$ by definition of $A' \subseteq A_r$;
	\item $dt_2(s)(x,y) < r$.
\end{enumerate}
Therefore, $\tPsi_{s}(y) \in B_{5r}(\Psi_{s}(z)) \subseteq B_{\frac{c_1}{2}r}(\Psi_{s}(z))$. 

Using exactly the same type of computation as the second part of the proof of the main lemma, 
\begin{align}\label{induction lemma eq2.3}
\begin{split}
\int\limits_{A' \times B_{r}(z)} dt_2(\epsilon)(x,y) \, d(m \times m)(x,y) \leq c(N,\delta)rm(B_r(z))^2\sqrt{\epsilon}.
\end{split}
\end{align}
$A'$ takes a significant portion of the measure of $B_r(z)$ by \eqref{induction lemma eq2}. The same considerations as before gives the existence of $D_1 \subseteq B_r(z)$ so that
\begin{equation}\label{induction lemma eq2.4}
\frac{m(D_1)}{m(B_r(z))} \geq 1 - V(1,10) \; \; \text{and} \; \; \tPsi_s(D_1) \subseteq B_{5r}(\Psi_s(z)),
\end{equation}
for any $s \in [0,\epsilon]$ after we bound $\epsilon_2$ sufficiently small depending only on $N$ and $\delta$. 
Moreover, $\tPsi_{s}(D_1)$ is non-trivial in measure compared to $B_{r}(z)$.
\begin{align}\label{induction lemma eq2.5}
\begin{split}
	\frac{m(\tilde{\Psi}_{s}(D_1))}{m(B_{r}(z))} &\geq e^{-c(N,\delta)\frac{\delta}{1000}}\frac{m(D_1)}{m(B_{r}(z))} \; \; \text{ , by \ref{lap cutoffed bound}, \ref{RLF existence} \eqref{RLF existence eq1}, and $\epsilon_2 \leq \frac{\delta}{1000}$}\\
	&\geq c(N,\delta) \; \; \text{ , by definition of $D_1$.}
\end{split}
\end{align}

Fix $s \in [0,\epsilon]$. By Proposition \ref{RLF inverse}, we may assume, up to choosing a full measure subset, that $D_1$ satisfies
\begin{equation}\label{induction lemma eq3.0}
\tPsi_{-t}(\tPsi_{s}(x)) = \tPsi_{s-t}(x) \; \forall t \in [0,s] \; \text{and} \; \forall x \in D_1.
\end{equation}
For all $t \in [0,s]$ and $(x,y) \in X \times X$, define
\begin{equation}\label{induction lemma eq3.1}
dt_{3}(t)(x,y) := \min\left\{r, \max\limits_{0 \leq \tau \leq t} |d(x,y) - d(\tilde{\Psi}_{-\tau}(x), \tilde{\Psi}_{-\tau}(y))|\right\}
\end{equation}
and 
\begin{equation}\label{induction lemma eq3.2}
U_3^t := \{(x,y) \in \tilde{\Psi}_{s}(D_1) \times B_{r'}(\Psi_{s}(z))| dt_{3}(t)(x,y) < r \}. 
\end{equation}

Consider $\int_{\tilde{\Psi}_{s}(D_1) \times B_{r'}(\Psi_{s}(z))} dt_3(t)(x,y) \, d(m \times m)(x,y)$ for $0 \leq t \leq s$. For any $t \in [0,s]$, $\omega \in [0,t]$, and $(x,y) \in U_3^t$, 
\begin{enumerate}
	\item $d(x,y) < r'+5r$ by \eqref{induction lemma eq2.4};
	\item $d(\tilde{\Psi}_{-\omega}(x), \Psi_{s-\omega}(z)) = d(\tilde{\Psi}_{s-\omega}(x'), \Psi_{s-\omega}(z)) < 5r$ for some $x' \in D_1$ by \eqref{induction lemma eq2.4} and \eqref{induction lemma eq3.0};
	\item $dt_3(t)(x,y) < r$.
\end{enumerate}
Hence, 
\begin{equation}\label{induction lemma eq3.2.5}
	\tilde{\Psi}_{-\omega}(y) \in B_{r'+11r}(\Psi_{s-\omega}(z))
\end{equation}
by triangle inequality. Therefore, $(\tilde{\Psi}_{-\omega}, \tilde{\Psi}_{-\omega})(U_3^t) \subseteq B_{\frac{c_1}{2}r}(\Psi_{s-\omega}(z)) \times B_{\frac{c_1}{2}r}(\Psi_{s-\omega}(z))$ for any $\omega \in [0,t]$ since \hyperref[c1]{$c_1$} $>100$. 

Using exactly the same type of computation as the third part of the proof of the main lemma, 
\begin{align}\label{induction lemma eq3.3}
\begin{split}
\int\limits_{\tilde{\Psi}_{s}(D_1) \times B_{r'}(\Psi_{s}(z))} dt_3(\epsilon)(x,y) \, d(m \times m)(x,y) \leq c(N,\delta)rm(B_r(z))^2\sqrt{\epsilon}.
\end{split}
\end{align}
By \eqref{induction lemma eq2.5}, $\tilde{\Psi}_{s}(D_1)$ is non-trivial in measure compared to $B_{r}(z)$. By Bishop-Gromov and property \ref{induction lemma 1} of $z$, $B_{r'}(\Psi_{s}(z))$ is also non-trivial in measure compared to $B_{r}(z)$. The same considerations as before gives the existence of $D_2 \subseteq B_{r'}(\Psi_{s}(z))$ so that
\begin{equation}\label{induction lemma eq3.4}
\frac{m(D_2)}{m(B_{r'}(\Psi_{s}(z)))} \geq 1 - V(1,10) \; \; \text{and} \; \; \tPsi_{-s}(D_2) \subseteq B_{r'+7r}(z) \subseteq B_{3r'}(z),
\end{equation}
for any $s \in [0,\epsilon]$ after we bound $\epsilon_2$ sufficiently small depending only on $N$ and $\delta$. 

We estimate
\begin{align*}
	&\hspace{0.5cm}\frac{m(B_{r'}(\Psi_s(z)))}{m(B_{r'}(z))}\leq \frac{1}{V(1,3)}\frac{m(B_{r'}(\Psi_{s}(z)))}{m(B_{3r'}(z))} \; \; \text{, by Bishop-Gromov}\\
								&\leq  \frac{1}{V(1,3)}\frac{1}{1-V(1,10)}\frac{m(D_2)}{m(B_{3r'}(z))} \text{, by property \eqref{induction lemma eq3.4} of $D_2$}\\
								&\leq  \frac{1}{V(1,3)}\frac{1}{1-V(1,10)}e^{c(N,\delta)s}\frac{m(\tilde{\Psi}_{-s}(D_2))}{m(B_{3r'}(z))}\; \;\text{, by Remark \ref{RLF local volume bound}}
\\
	&\leq \frac{1}{V(1,7)}\frac{1}{1-V(1,10)}e^{cs},
\end{align*}

where for the third inequality we used the fact that $\tPsi_{-t}(D_2)\subseteq B_{27r}(\Psi_{s-t}(z))$ for $0 \leq t\leq s$ by the calcluations of \eqref{induction lemma eq3.2.5}; On these sets $\Delta h^-$ is bounded below by $-c(N,\delta)$ by the same argument as \eqref{main lemma eq3.3.5}. Using this, we bound $\epsilon_2$ sufficiently small depending on $N$ and $\delta$ so that for $s \leq \epsilon \leq \epsilon_2$, the last line is less than $V(1,100)$. This shows the other half of the bound in property \ref{induction lemma i}. Property \ref{induction lemma iii} is obvious and so we conclude. 
\end{proof}

Combined with the main lemma, this gives
\begin{lem}\label{arbitrarily close}
There exists $\epsilon_3(N,\delta)>0$ and $\bar{r}_3(N,\delta)>0$ so that for any $r \leq \bar{r}_3$ and $\delta \leq t_0 \leq 1-\delta$, there exists $z \in B_{r}(\gamma(t_0))$ so that for any $r \leq r' \leq \bar{r}_3$:
\begin{enumerate}
	\item $V(1, 100) \leq \frac{m(B_{r'}(\Psi_s(z)))}{m(B_{r'}(z))} \leq \frac{1}{V(1,100)}$ for any $s \leq \epsilon_{3}$;
	\item There exists $A \subseteq B_{r'}(z)$ with $m(A) \geq (1-V(1,10))m(B_{r'}(z))$ and $\Psi_{s}(A) \subseteq B_{2r'}(\Psi_{s}(z))$ for any $s \leq \epsilon_{3}$;
	\item $e(z) \leq c_1^2r^2$.
\end{enumerate}
\end{lem}
\begin{proof}
Choose $\epsilon_{3} = \min \left\{\epsilon_1, \epsilon_2 \right\}$ and $\bar{r}_{3}:=\min\braces{\bar{r}_1, \bar{r}_2}$. Apply \ref{main lemma} to find some $z \in B_{\frac{r}{4}}(\gamma(t_0))$ which satisfies properties 1 - 3  on the scale of $\frac{r}{4}$ and then use \ref{induction lemma} repeatedly to conclude. 
\end{proof}

The following corollary immediately follows from Lemma \ref{arbitrarily close} by a limiting argument. 
\begin{cor}\label{at unique geodesic}
There exists $\epsilon_3(N,\delta)>0$ and $\bar{r}_3(N,\delta) > 0$ so that if $\gamma$ is the unique geodesic between $p$ and $q$, then for all $r \leq \bar{r}_3$  and $\delta \leq t_0 \leq 1-\delta$, 
\begin{enumerate}
	\item $V(1, 100) \leq \frac{m(B_r(\gamma(t_0-s)))}{m(B_r(\gamma(t_0)))} \leq \frac{1}{V(1,100)}$ for any $0 \leq s \leq \epsilon_{3}$.
	\item There exists $A \subseteq B_{r}(\gamma(t_0))$ with $m(A) \geq (1-V(1,10))m(B_r(\gamma(t_0)))$ and $\Psi_{s}(A) \subseteq B_{2r}(\gamma(t_0-s))$ for any $0 \leq s \leq \epsilon_{3}$. 
\end{enumerate}
\end{cor}

We can use the corollary directly in the proof of Theorem \ref{section main theorem} to prove the main result for all unique geodesics. Taking Theorem \ref{directional BG} and Remark \ref{a.e. unique geodesic} into account, this is already enough for several applications. Nevertheless, we will next prove the existence of a geodesic between any $p, q \in X$ with H\"older continuity on the geometry of small radius balls in its interior, which is the full result of \cite{CN12}. The desired geodesic will be constructed using multiple limiting and gluing arguments. 

\begin{lem} \label{at geodesic one direction}
There exists $\epsilon_4(N,\delta) > 0$ and $\bar{r}_4(N,\delta) > 0$ so that for any unit speed geodesic $\gamma$ from $p$ to $q$, there exists a unit speed geodesic $\gamd$ from $p$ to $q$ with $\gamd \equiv \gamma$ on $[1-\delta,1]$ so that for all $r \leq \bar{r}_4$  and $\delta \leq t_0 \leq t_1 \leq 1-\delta$, if $t_1-t_0 \leq \epsilon_4$ then,
\begin{enumerate}
	\item \label{agod 1}$V(1, 100)^4 \leq \frac{m(B_r(\gamd(t_1)))}{m(B_r(\gamd(t_0)))} \leq \frac{1}{V(1,100)^4}$;
	\item \label{agod 2}There exists $A \subseteq B_{r}(\gamd(t_1))$ so that $m(A) \geq (1-V(1,10))m(B_r(\gamd(t_1)))$ and $\Psi_{s}(A) \subseteq B_{2r}(\gamd(t_1-s))$ for all $s \in [0,t_1-t_0]$. 
\end{enumerate}
\end{lem}

\begin{proof}
Let $\epsilon_3$, $\bar{r}_3$ be from Lemma \ref{arbitrarily close}. We begin by assuming $\bar{r}_4 \leq \bar{r}_3$ and $\epsilon_4 \leq \frac{\epsilon_3}{2}$ but will impose more bounds on both as the proof continues. We will not keep track of $\bar{r}_4$ for the sake of brevity.

Partition $[\delta, 1-\delta]$ by $\braces{\sigma_0 = \delta, \sigma_1, \; ... \;, \sigma_k = 1-\delta}$ for some $k \in \mathbb{N}$ so that all subintervals have the same width equal to some $\omega \in [\frac{\epsilon_3}{2}, \epsilon_3]$. This is always possible since the width of the original interval is $1-2\delta > 0.8$ and $\epsilon_3$ is much smaller than 0.4. We construct $\gamd$ inductively as follows
\begin{itemize}
	\item Define $\gamd \equiv \gamma$ for $t \in [1-\delta, 1]$.
	\item Let $1  \leq i \leq k$. Assume $\gamd$ has been constructed on the interval $[\sigma_i, 1]$. Fix a sequence of $r_j \to 0$. For each $j$, choose $z_j \in B_{r_j}(\gamd(\sigma_i))$ as in the statement of Lemma \ref{arbitrarily close}. Use Arzel\`a-Ascoli Theorem to take a limit, after passing to a subsequence, of the unit speed geodesics from $z_j$ to $p$. This limit $\gamma^{\delta,i}$ is a unit speed geodesic from $\gamd(\sigma_i)$ to $p$. For $t\in[\sigma_{i-1}, \sigma_{i}]$, define $\gamd(t) = \gamma^{\delta,i}(\sigma_i-t)$.
	\item For $t\in[0,\delta]$, define $\gamd$ to be any geodesic from $p$ to $\gamd(\sigma_0)$.
\end{itemize}

For any $1\leq i\leq k$, $\sigma_i-\sigma_{i-1} \leq \epsilon_3$ so it follows from the construction that $\gamd$ satisfies, for any $r \leq \bar{r}_4$,
\begin{enumerate}[label=\roman*., ref=\roman*]
	\item \label{agod i} $V(1, 100) \leq \frac{m(B_r(\gamd(\sigma_i-s)))}{m(B_r(\gamd(\sigma_i)))} \leq \frac{1}{V(1,100)}$ for any $0 \leq s \leq \omega$;
	\item  \label{agod ii} There exists $A \subseteq B_{r}(\gamd(\sigma_i))$ with $m(A) \geq (1-V(1,10))m(B_{r}(\gamd(\sigma_i)))$ and $\Psi_{s}(A) \subseteq B_{2r}(\gamd(\sigma_i - s))$ for any $0 \leq s \leq \omega$. 
\end{enumerate}

Fix $\delta\leq t_0 \leq t_1\leq1-\delta$ with $t_1 - t_0\leq\epsilon_4$ and a scale $r \leq \bar{r}_4$. Since $\epsilon_4 \leq \frac{\epsilon_3}{2}$ is no greater than the widths of the subintervals of the partition $\omega \geq \frac{\epsilon_3}{2}$, $t_0$ and $t_1$ must be contained in $[\sigma_{i-2}, \sigma_{i}]$ for some $2\leq i \leq k$. Statement \ref{agod 1} of the lemma then follows trivially from property \ref{agod i} of $\gamd$.

$t_0$ and $t_1$ must then be either contained in a single subinterval or two neighbouring subintervals of the partition. We will assume the second case; the first case follows from a similar and simpler argument. Let $t_1 \in (\sigma_{i-1}, \sigma_{i}]$ and $t_0 \in (\sigma_{i-2}, \sigma_{i-1}]$ for some $i$. By property \ref{agod ii} of $\gamd$ and Abresch-Gromoll, there exists $A_1 \subseteq B_{\frac{r}{16}}(\gamd(\sigma_i))$ so that
\begin{enumerate}
	\item $\frac{m(A_1)}{m(B_{\frac{r}{16}}(\gamd(\sigma_i)))} \geq 1 - 2V(1,10)$;
	\item $\Psi_{s}(A_1) \subseteq B_{\frac{r}{8}}(\gamd(\sigma_{i}-s)) \; \forall s \in [0, \omega]$;
	\item $e(x) \leq c(N,\delta)r^2 \; \forall x \in A_1$.
\end{enumerate}
Simiarly, there exists $A_2 \subseteq B_{\frac{r}{16}}(\gamd(\sigma_{i-1}))$ so that
\begin{enumerate}
	\item $\frac{m(A_2)}{m(B_{\frac{r}{16}}(\gamd(\sigma_{i-1})))} \geq 1 - 2V(1,10)$;
	\item $\Psi_{s}(A_2) \subseteq B_{\frac{r}{8}}(\gamd(\sigma_{i-1}-s)) \; \forall s \in [0, \omega]$;
	\item $e(x) \leq c(N,\delta)r^2 \; \forall x \in A_2$.
\end{enumerate}

The plan is as follows: first we show that a significant portion of $A_1$ can be flowed by $\Psi$ a non-trivial amount of time past $\gamd(\sigma_{i-1})$ while staying close $\gamd$, then we use the flow of $A_1$ under $\Psi$ to control the flow of $B_{r}(\gamd(t_1))$ under $\Psi$. 

Fix $h^- \equiv h^-_{\rho(8r)^2}$ satisfying statement 4 of Theorem \ref{hess h integral bound} for the balls of radius $8r$ along $\gamd$, where $\rho \in [\frac{1}{2},2]$. For all $s \in [0,\omega]$ and $(x,y) \in X \times X$, define 
\begin{equation}\label{agod eq1}
dt(s)(x,y) := \min\left\{r, \max\limits_{0 \leq \tau \leq s} |d(x,y) - d(\Psi_\tau(x), \Psi_\tau(y))|\right\}
\end{equation}
and 
\begin{equation}\label{agod eq2}
U_1^s := \{(x,y) \in A_2 \times A_1|dt(s)(x,\Psi_{\omega}(y)) < r\}. 
\end{equation}
Consider $\int_{A_2 \times A_1} dt(s)(x,\Psi_{\omega}(y)) \, d(m \times m)(x,y)$ for $0 \leq s \leq \omega$. 

For any $s \in [0,\omega]$ and $(x,y) \in  U_1^s$,
\begin{enumerate}
	\item $d(x,\Psi_{\omega}(y))<\frac{3r}{16}$;
	\item $d(\Psi_{s}(x), \gamd(\sigma_{i-1}-s)) < \frac{r}{8}$;
	\item $dt(s)(x, \Psi_{\omega}(y))<r$.
\end{enumerate}
Therefore, $\Psi_{s}(\Psi_{\omega}(y)) \in B_{r+\frac{5r}{16}}(\gamd(\sigma_{i-1}-s))$ and so $(\Psi_s, \Psi_s \circ \Psi_{\omega})(U_1^s) \subseteq B_{4r}(\gamd(\sigma_{i-1}-s)) \times B_{4r}(\gamd(\sigma_{i-1}-s))$. By Remark \ref{a.e. unique geodesic}, we have $\Psi_{s} \circ \Psi_{\omega} = \Psi_{s+\omega}$ $m$-a.e.. Since we may always choose subsets of full measure where the equality is satisfied, we will replace the former with the latter freely.

We have,
\begin{align}\label{agod eq3}
\begin{split}
	&\hspace{0.5cm}\int\limits_{0}^{1}  \int\limits_{U_1^s} d(\Psi_s(x), \Psi_{s+\omega}(y)))|\Hess h^-|_{\HS}(\tilde{\gamma}_{\Psi_s(x),\Psi_{s+\omega}(y)}(\tau)) \,  d(m \times m)(x,y) \, d\tau\\
	&\leq c(N,\delta)\int\limits_{0}^{1}  \int\limits_{(\Psi_s, \Psi_{s+\omega})(U_1^s)} d(x, y)|\Hess h^-|_{\HS}(\tilde{\gamma}_{x,y}(\tau)) \,  d(m \times m)(x,y) \, d\tau \; \; \text{, by Theorem \ref{-dp grad flow properties}, 2}\\
	&\leq c(N,\delta)rm(B_{4r}(\gamd(\sigma_{i-1}-s)))\int\limits_{B_{8r}(\gamd(\sigma_{i-1}-s))} |\Hess h^-|_{\HS}\,  dm \; \; \text{, by segment inequality \ref{segment inequality}}\\
	&\leq c(N,\delta)rm(B_{r}(\gamd(\sigma_{i})))^2 \fint\limits_{B_{8r}(\gamd(\sigma_{i-1}-s))} |\Hess h^-|_{\HS} \, dm \; \; \text{, by Bishop-Gromov and property \ref{agod i} of $\gamd$}.
\end{split}
\end{align}
Integrating in $s \in [0,\omega']$ for some $\omega' \in (0, \omega]$ to be fixed later, we have
\begin{align}\label{agod eq4}
\begin{split}
	&\hspace{0.5cm} \int\limits_{0}^{\omega'} \bigg( \int\limits_{0}^{1} \int\limits_{U_1^s} d(\Psi_s(x), \Psi_{s+\omega}(y))|\Hess h^-|_{\HS}(\tilde{\gamma}_{\Psi_s(x),\Psi_{s+\omega}(y)}(\tau)) \,  d(m \times m)(x,y)\, d\tau \bigg)\, ds\\
	&\leq crm(B_{r}(\gamd(\sigma_{i})))^2 \int\limits_{0}^{\omega'} \fint\limits_{B_{8r}(\gamd(\sigma_{i-1}-s))} |\Hess h^-|_{\HS} \, dm \, ds\\
	&\leq c(N,\delta)rm(B_{r}(\gamd(\sigma_{i})))^2 \sqrt{\omega'},
\end{split}
\end{align}
where the last line follows from the definition of $h^-$, statement 4 of Theorem \ref{hess h integral bound}, and Cauchy-Schwarz.

By statement 3 of \ref{grad h geodesic bound}, the excess bound on the elements of $A_2$ and property \ref{agod i} of $\gamd$,
\begin{align}\label{agod eq5}
\begin{split}
	\int\limits_{0}^{\omega'} \int\limits_{U^s_1} |\nabla h^- - \nabla d_p|(\Psi_s(x)) \, d(m \times m)(x,y)  \, ds \leq c(N,\delta)rm(B_r(\gamd(\sigma_i)))^2\sqrt{\omega'}.
\end{split}
\end{align}
Similarly by the excess bounds on the elements of $A_1$, 
\begin{equation}\label{agod eq6}
\int\limits_{0}^{\omega'} \int\limits_{U^s_1} |\nabla h^- - \nabla d_p|(\Psi_{s+\omega}(y))\, d(m \times m)(x,y) \, ds \leq c(N,\delta)m(B_r(\gamd(\sigma_i)))^2r\sqrt{\omega'}.
\end{equation}

By Proposition \ref{distortion bound}, 
\begin{equation}\label{agod eq6.5}
 	\int_{A_2 \times A_1} dt(\omega')(x,\Psi_{\tau}(y)) \leq c(N,\delta)m(B_r(\gamd(\sigma_i)))^2r\sqrt{\omega'}.
\end{equation}
Arguing as in the proof of Lemma \ref{main lemma} and using property \ref{agod i} of $\gamd$, we can then fix $\omega'$ sufficiently small depending only on $N$ and $\delta$ so that there exists $z \in A_2$ and $A_1' \subseteq A_1$ with
\begin{equation}\label{agod eq7}
\frac{m(A_1')}{m(B_{\frac{r}{16}}(\gamd(\sigma_i)))} \geq 1 - 3V(1,10) \; \; \text{and} \; \; dt(\omega')(z,\Psi_{\omega}(y)) \leq \frac{r}{16} \; \forall y \in A_1'.
\end{equation}
The latter implies 
\begin{equation}\label{agod eq8}
\Psi_{s+\omega}(A_1') \subseteq B_{\frac{3r}{8}}(\gamd(\sigma_{i-1}-s)) \; \forall s \in [0,\omega'].
\end{equation} 
Notice by definiton of $A_1$, $\Psi_{s}(A_1') \subseteq B_{\frac{r}{8}}(\gamd(\sigma_{i}-s))$ for any $s \in [0,\omega]$. 

We now compare the flow of $B_{r}(\gamd(t_1))$ to that of $\Psi_{\sigma_{i}-t_1}(A_1')$ under $\Psi_s$ for $s \in [0,\omega']$. By integral Abresch-Gromoll there exists $B_{r}(\gamd(t_1))' \subseteq B_{r}(\gamd(t_1))$ so that 
\begin{equation}\label{agod eq2.1}
 e(x) \leq c(N,\delta)r^2 \; \forall x \in B_{r}(\gamd(t_1))' \; \; \text{and}\; \; \frac{m(B_{r}(\gamd(t_1))')}{m(B_{r}(\gamd(t_1)))} \geq 1-\frac{1}{2}V(1,10).
\end{equation}

For all $s \in [0, \omega']$, define
\begin{equation}\label{agod eq2.2}
U_2^s := \{(x,y) \in A_1' \times B_{r}(\gamd(t_1))'|dt(s)(\Psi_{\sigma_{i}-t_1}(x),y) < r\}. 
\end{equation}
Consider $\int_{A_1' \times B_{r}(\gamd(t_1))'} dt(s)(\Psi_{\sigma_i-t_1}(x),y) \, d(m \times m)(x,y)$ for $0 \leq s \leq \omega'$. For any $s \in [0,\omega']$ and $(x,y) \in  U_2^s$,
\begin{enumerate}
	\item $d(\Psi_{\sigma_i - t_1}(x),y)<r+\frac{r}{8}$ by definition of $A_1$;
	\item $d(\Psi_{s}(\Psi_{\sigma_i - t_1}(x)), \gamd(t_1-s)) < \frac{3r}{8}$ by \eqref{agod eq8} and the line below it;
	\item $dt(s)(\Psi_{\sigma_i - t_1}(x),y)<r$.
\end{enumerate}
Therefore, $\Psi_{s}(y) \in B_{\frac{5r}{2}}(\gamd(t_1-s))$ and so $(\Psi_s \circ \Psi_{\sigma_i - t_1}, \Psi_s)(U_2^s) \subseteq B_{4r}(\gamd(t_1-s)) \times B_{4r}(\gamd(t_1-s))$. 

By the same type of computations as the first part of this proof, for some $\omega'' \in (0,\omega']$ to be fixed later,
\begin{equation}\label{agod eq2.3}
 	\int_{A_1' \times B_{r}(\gamma(t_1))'} dt(\omega'')(\Psi_{\sigma_{i}-t_1}(x), y) \leq c(N,\delta)m(B_r(\gamd(\sigma_i)))^2r\sqrt{\omega''}.
\end{equation} 
Arguing as in the proof of Lemma \ref{main lemma} and using property \ref{agod i} of $\gamd$, we can then fix $\omega''$ sufficiently small depending only on $N$ and $\delta$ so that there exists $z' \in A_1'$ and $A \subseteq B_{r}(\gamma(t_1))'$ with
\begin{equation}\label{agod eq2.4}
\frac{m(A)}{m(B_{r}(\gamd(t_1)))} \geq 1 - V(1,10) \; \; \text{and} \; \; dt(\omega'')(\Psi_{\sigma_i-t_1}(z'),y) \leq \frac{r}{2} \; \forall y \in A.
\end{equation}
The latter implies 
\begin{equation}\label{agod eq2.5}
\Psi_{s}(A) \subseteq B_{2r}(\gamd(t_1-s)) \; \forall s \in [0,\omega''].
\end{equation} 
We bound $\epsilon_4 \leq \omega''$ and so statement \ref{agod 1} of the lemma is proved.
\end{proof}

We may also apply Lemma \ref{at geodesic one direction} in the other direction of $\gamma$ towards $q$. However, there is no guarantee the two geodesics we end up with in the two applications of the lemma are the same geodesics. Therefore, we will show that a geodesic which has the properties of \ref{at geodesic one direction} necessarily has the same properties going in the other direction. The reason for this is Lemma \ref{excess pw bound}, which roughly implies the local flows of $h^+$ and $h^-$ are close to each other near a geodesic on the scale of $r$. 

\begin{lem} \label{at geodesic other direction}
There exists $\epsilon_5(N,\delta)>0$ and $\bar{r}_5(N,\delta) >0$ so that if there exists a unit speed geodesic $\gamma$ from $p$ to $q$, $\epsilon \leq \epsilon_5$ and $\bar{r} \leq \bar{r}_5$ which satisfy, for all $r \leq \bar{r}$ and $\delta \leq t_0 \leq t_1 \leq 1-\delta$ with $t_1 - t_0 \leq \epsilon$, 
\begin{enumerate}
	\item \label{agotd 1}$V(1, 100)^4 \leq \frac{m(B_r(\gamma(t_1)))}{m(B_r(\gamma(t_0)))} \leq \frac{1}{V(1,100)^4}$;
	\item \label{agotd 2}There exists $A_1 \subseteq B_{r}(\gamma(t_1))$ so that $m(A_1) \geq (1-V(1,10))m(B_r(\gamma(t_1)))$ and $\Psi_{s}(A_1) \subseteq B_{2r}(\gamma(t_1-s))$ for all $s \in [0, t_1 - t_0]$,
\end{enumerate}
then for the same geodesic $\gamma$, $\epsilon$ and $\bar{r}$, for all $r \leq \bar{r}$ and $\delta \leq t_0 \leq t_1 \leq 1-\delta$ with $t_1 - t_0 \leq \epsilon$, there exists $A_2 \subseteq B_{r}(\gamma(t_0))$ so that 
\begin{equation}\label{agotd eq0}
	\frac{m(A_2)}{m(B_r(\gamma(t_0)))} \geq 1 - V(1,10) \; \; \text{and} \; \; \Phi_{s}(A_2) \subseteq B_{2r}(\gamma(t_0+s)) \; \forall s \in [0, t_1-t_0].
\end{equation}
\end{lem}
\begin{proof}
As a reminder, $\Phi$ is defined by \eqref{Phidef} and is the local flow of $-\nabla d_q$, at least from the sets and for the time interval we are concerned with.

We assume $\bar{r}_5 \leq \frac{\delta}{10}$ and $\epsilon_5 \leq \frac{\delta}{10}$ to begin with but will impose more bounds on both as the proof continues. We will not keep track of $\bar{r}_5$ for the sake of brevity. Fix $\gamma$, $\epsilon \leq \epsilon_5$ and $\bar{r} \leq \bar{r}_5$ which satisfy conditions \ref{agotd 1} and \ref{agotd 2}. Fix $r \leq \bar{r}$ and $\delta \leq t_0 \leq t_1 \leq 1-\delta$ with $t_1 - t_0 \leq \epsilon$. 

Fix $h^- \equiv h^-_{\rho(8r)^2}$ satisfying statement 4 of Theorem \ref{hess h integral bound} for the balls of radius $8r$ along $\gamma$, where $\rho \in [\frac{1}{2},2]$. Let $(\tPsi_t)_{t \in [0,1]}$ and $(\tPsi_{-t})_{t \in [0,1]}$ be the RLFs of the time-independent vector fields $-\nabla h^-$ and $\nabla h^-$ respectively as before. 

The plan is as follows: first we use property \ref{agotd 2} to make sure a significant portion of $B_{\frac{r}{16}}(\gamma(t_1))$ stays close to $\gamma$ from $t_1$ to $t_0$ under the flow of $\tPsi_{s}$, then we reverse flow the image of this portion under $\tPsi_{-s}$ and use it to make sure a significant portion of $B_r(\gamma(t_0))$ stays close to $\gamma$ from $t_0$ to $t_1$ under $\Phi_{s}$. 

By condition \ref{agotd 2} and integral Abresch-Gromoll, there exists $A_1 \subseteq B_{\frac{r}{16}}(\gamma(t_1))$ so that
\begin{enumerate}
	\item $\frac{m(A_1)}{m(B_{\frac{r}{16}}(\gamma(t_1)))} \geq 1 - 2V(1,10)$;
	\item $\Psi_{s}(A_1) \subseteq B_{\frac{r}{8}}(\gamma(t_1-s)) \; \forall s \in [0, t_1-t_0]$;
	\item $e(x) \leq c(N,\delta)r^2 \; \forall x \in A_1$.
\end{enumerate}

For all $s \in [0,t_1-t_0]$ and $(x,y)\in X \times X$, define
\begin{equation}\label{agotd eq1}
dt_1(s)(x,y) := \min\left\{r, \max\limits_{0 \leq \tau \leq s} |d(x,y) - d(\Psi_\tau(x), \tilde{\Psi}_\tau(y))|\right\}
\end{equation}
and 
\begin{equation}\label{agotd  eq2}
U_1^s := \{(x,y) \in A_1 \times B_{\frac{r}{16}}(\gamma(t_1))|dt_1(s)(x,y) < r\}. 
\end{equation}
Consider $\int_{A_1 \times B_{\frac{r}{16}}(\gamma(t_1))} dt_1(s)(x,y) \, d(m \times m)(x,y)$ for $0 \leq s \leq t_1-t_0$. For any $s \in [0, t_1 - t_0]$ and $(x,y) \in U_1^s$, 
\begin{enumerate}
	\item $d(x,y) < \frac{r}{8}$;
	\item $d(\Psi_{s}(x), \gamma(t_1-s)) < \frac{r}{8}$ by definition of $A_1$;
	\item $dt_1(s)(x,y) < r$.
\end{enumerate}
Therefore, $\tilde{\Psi}_{s}(y) \in B_{\frac{5r}{4}}\gamma(t_1-s)$ by triangle inequality and so $(\Psi_{s}, \tilde{\Psi}_{s})(U_1^s) \subseteq B_{4r}(\Psi_{s}(z)) \times B_{4r}(\Psi_{s}(z))$.

Using exactly the same type of computation as the second part of the proof of the main lemma, 
\begin{align}\label{agotd eq3}
\begin{split}
\int\limits_{A_1 \times B_{\frac{r}{16}}(\gamma(t_1))} dt_1(t_1-t_0)(x,y) \, d(m \times m)(x,y) \leq c(N,\delta)rm(B_r(\gamma(t_1)))^2\sqrt{t_1-t_0}.
\end{split}
\end{align}
$A_1$ takes a significant portion of the measure of $B_r(\gamma(t_1))$ by definition. The same considerations as in the proof of Lemma \ref{main lemma} gives the existence of $z \in A_1$ and $D_1 \subseteq B_r(\gamma(t_1))$ so that
\begin{equation}\label{agotd eq4}
\frac{m(D_1)}{m(B_{\frac{r}{16}}(\gamma(t_1)))} \geq 1 - V(1,10) \; \; \text{and} \; \; dt_1(t_1-t_0)(z,y)\leq \frac{r}{16} \; \forall y \in D_1,
\end{equation}
after we bound $\epsilon_5$ sufficiently small depending only on $N$ and $\delta$. The latter implies 
\begin{equation}\label{agotd eq5}
\tPsi_{s}(D_1) \subseteq B_{\frac{5r}{16}}(\gamma(t_1-s)) \; \forall s \in [0,t_1-t_0].
\end{equation} 
By Proposition \ref{RLF inverse}, we may assume in addition that $D_1$ satisfies
\begin{equation}\label{agotd eq5.5}
\tPsi_{-s}(\tPsi_{t_1-t_0}(x)) = \tPsi_{t_1-t_0-s}(x) \; \forall s \in [0,t_1-t_0] \; \text{and} \; \forall x \in D_1.
\end{equation}

Furthermore, $\tPsi_{s}(D_1)$ is non-trivial in measure compared to $B_{r}(\gamma(t_1))$ for all $s \in [0,t_1-t_0]$.
\begin{align}\label{agotd eq6}
\begin{split}
	\frac{m(\tilde{\Psi}_{s}(D_1))}{m(B_{r}(\gamma(t_1)))} &\geq e^{-c(N,\delta)\frac{\delta}{10}}\frac{m(D_1)}{m(B_{r}(z))} \; \; \text{ , by \ref{lap cutoffed bound}, \ref{RLF existence} \eqref{RLF existence eq1}, and $\epsilon_5 \leq \frac{\delta}{10}$}\\
	&\geq c(N,\delta) \; \; \text{ , by definition of $D_1$ and Bishop-Gromov.}
\end{split}
\end{align}

By integral Abresch-Gromoll, there exists $B_{r}(\gamma(t_0))' \subseteq B_{r}(\gamma(t_0))$ so that
\begin{equation}\label{agotd eq6.5}
 e(x) \leq c(N,\delta)r^2 \; \forall x \in B_{r}(\gamma(t_0))' \; \; \text{and}\; \; \frac{m(B_{r}(\gamma(t_0))')}{m(B_{r}(\gamma(t_0)))} \geq 1-\frac{1}{2}V(1,10).
\end{equation}

For all $s \in [0,t_1-t_0]$ and $(x,y)\in X \times X$, define
\begin{equation}\label{agotd eq7}
dt_2(s)(x,y) := \min\left\{r, \max\limits_{0 \leq \tau \leq s} |d(x,y) - d(\tPsi_{-\tau}(x), \Phi_\tau(y))|\right\}
\end{equation}
and 
\begin{equation}\label{agotd  eq8}
U_2^s := \{(x,y) \in \tPsi_{t_1-t_0}(D_1) \times B_{r}(\gamma(t_0))'|dt_2(s)(x,y) < r\}. 
\end{equation}
Consider $\int_{\tPsi_{t_1-t_0}(D_1) \times B_{r}(\gamma(t_0))'} dt_2(s)(x,y) \, d(m \times m)(x,y)$ for $0 \leq s \leq t_1-t_0$. By proposition \ref{distortion bound}, for a.e. $s \in [0,t_1-t_0]$,
\begin{align}\label{agotd eq9}
\begin{split}
	&\hspace{0.5cm}\frac{d}{ds} \int\limits_{\tPsi_{t_1-t_0}(D_1) \times B_{r}(\gamma(t_0))'} dt_2(s)(x,y) \, d(m \times m)(x,y)\\
	 &\leq \int\limits_{U_2^s} |\nabla h^- + \nabla d_q|(\Phi_{s}(y)) \, d(m \times m)(x,y)\\
&\hspace{2cm} + \int\limits_{0}^{1} \int\limits_{U_2^s} d(\tPsi_s(x), \Phi_s(y))|\Hess h^-|_{\HS}(\tilde{\gamma}_{\tPsi_s(x),\Phi_s(y)}(\tau)) \,  d(m \times m)(x,y) \, d\tau.
\end{split} 
\end{align}

For any $s \in [0,t_1-t_0]$, $s' \in [0,s]$, and $(x,y) \in U_2^s$, 
\begin{enumerate}
	\item $d(x,y) < r+\frac{5r}{16}$ by \eqref{agotd eq5};
	\item $d(\tPsi_{-s'}(x), \gamma(t_0+s')) = d(\tilde{\Psi}_{t_1-t_0-s'}(x'), \gamma(t_0+s')) < \frac{5r}{16}$ for some $x' \in D_1$ by \eqref{agotd eq5} and \eqref{agotd eq5.5};
	\item $dt_2(s)(x,y) < r$ by definition of $U_2^s$ \eqref{agotd eq8}.
\end{enumerate}
Hence,
\begin{equation}\label{agotd eq10}
\Phi_{s'}(y) \in B_{2r+\frac{5r}{8}}(\gamma(t_0+s'))
\end{equation}
by triangle inequality. Therefore, $(\tPsi_{-s'}, \Phi_{s'})(U_2^s) \in B_{4r}(\gamma(t_0+s')) \times B_{4r}(\gamma(t_0+s'))$ for all $s' \in [0,s]$. For any $(x,y) \in U_2^s$,
\begin{align}\label{agotd eq11}
\begin{split}
	\Delta h^-(\tilde{\Psi}_{-s'}(x)) &= \Delta h^+(\tilde{\Psi}_{-s'}(x)) + \Delta \hat{e}(\tilde{\Psi}_{-s'}(x))\\
	&\geq -c(N,\delta) \; \; \text{ , by Lemma \ref{lap cutoffed bound} and Lemma \ref{excess pw bound} \ref{excess pw bound 3}},
\end{split}
\end{align}
where $h^+$, $\hat{e}$ are heat flow approximations of $h^+_0$ and $e_0$ respectively up to the same time as $h^-$. 
Therefore,
\begin{align}\label{agotd eq12}
\begin{split}
	&\hspace{0.5cm}\int\limits_{0}^{1} \int\limits_{U_2^s} d(\tilde{\Psi}_{-s}(x), \Phi_s(y))|\Hess h^-|_{\HS}(\tilde{\gamma}_{\tilde{\Psi}_{-s}(x), \Phi_s(y)}(\tau)) \,  d(m \times m)(x,y) \, d\tau\\
	&\leq c(N,\delta)\int\limits_{0}^{1} \int\limits_{(\tilde{\Psi}_{-s},\Phi_s)(U_2^s)} d(x, y)|\Hess h^-|_{\HS}(\tilde{\gamma}_{x,y}(\tau)) \,  d(m \times m)(x,y) \, d\tau \text{, by \ref{-dp grad flow properties} 2, \eqref{agotd eq11} and \ref{RLF local volume bound}}\\
	&\leq c(N,\delta)rm(B_{4r}(\gamma(t_0+s)))\int\limits_{B_{8r}(\gamma(t_0+s))} |\Hess h^-|_{\HS}\,  dm \; \; \text{, by segment inequality \ref{segment inequality}}\\
	&\leq c(N,\delta)rm(B_{r}(\gamma(t_1)))^2 \fint\limits_{B_{8r}(\gamma(t_0+s))} |\Hess h^-|_{\HS} \, dm \; \; \text{, by Bishop-Gromov and property \ref{agotd 1} of $\gamma$}.
\end{split}
\end{align}
Integrating in $s \in [0,t_1-t_0]$, 
\begin{align}\label{agotd eq13}
\begin{split}
	&\hspace{0.5cm} \int\limits_{0}^{t_1-t_0} \bigg(\int\limits_{0}^{1}  \int\limits_{U_2^s} d(\tilde{\Psi}_{-s}(x),\Phi_{s}(y))|\Hess h^-|_{\HS}(\tilde{\gamma}_{\tilde{\Psi}_{-s}(x), \Phi_s(y)}(\tau)) \,  d(m \times m)(x,y) \, d\tau \bigg)\, dt\\
	& \leq crm(B_{r}(\gamma(t_1)))^2 \int\limits_{0}^{t_1-t_0} \fint\limits_{B_{8r}(\gamma(t_0+s))} |\Hess h^-|_{\HS} \, dm \, ds\\
	&\leq c(N,\delta)rm(B_{r}(\gamma(t_1)))^2 \sqrt{t_1-t_0},
\end{split}
\end{align}
where the last line follows from the definition of $h^-$, statement 4 of Theorem \ref{hess h integral bound}, and Cauchy-Schwarz.

We have
\begin{align}\label{agotd eq14}
\begin{split}
	\int\limits_{0}^{t_1-t_0} \int\limits_{U^s_2} |\nabla h^+ + \nabla d_q|(\Phi_s(y)) \, d(m \times m)(x,y)  \, ds \leq c(N,\delta)rm(B_r(\gamma(t_1)))^2\sqrt{t_1-t_0},
\end{split}
\end{align}
where the first inequality is from statement 3 of Lemma \ref{grad h geodesic bound} and the excess bound on the elements of $B_{r}(\gamma(t_0))'$ \eqref{agotd eq6.5}, and the second inequality is from property \ref{agotd 1} of $\gamma$ and Bishop-Gromov. 
Similarly, using statement 3 of Lemma \ref{excess pw bound} with \eqref{agotd eq10}, 
\begin{align}\label{agotd eq15}
\begin{split}
	\int\limits_{0}^{t_1-t_0} \int\limits_{U^s_2} |\nabla h^- - \nabla h^+|(\Phi_s(y)) \, d(m \times m)(x,y)  \, ds &\leq c(N,\delta)rm(B_r(\gamma(t_1)))^2(t_1-t_0)\\
	&\leq crm(B_r(\gamma(t_1)))^2\sqrt{t_1-t_0}.
\end{split}
\end{align}

Combining \eqref{agotd eq12} - \eqref{agotd eq14} with the bound \eqref{agotd eq9} on $\frac{d}{ds} \int\limits_{\tPsi_{t_1-t_0}(D_1) \times B_r(\gamma(t_0))'} dt_2(s)(x,y)$, we obtain
\begin{align}\label{agotd eq16}
\begin{split}
&\hspace{0.5cm}\int\limits_{\tPsi_{t_1-t_0}(D_1) \times B_r(\gamma(t_0))'} dt_2(t_1-t_0)(x,y) \, d(m \times m)(x,y)\\
& = \int_{0}^{t_1-t_0} [\frac{d}{ds} \int\limits_{\tPsi_{t_1-t_0}(D_1) \times B_r(\gamma(t_0))'} dt_2(s)(x,y) \, d(m \times m)(x,y)] \, ds\\ 
&\leq c(N,\delta)rm(B_r(\gamma(t_1)))^2\sqrt{t_1-t_0}.
\end{split}
\end{align}

Both $\tPsi_{t_1-t_0}(D_1)$ and $B_{r}(\gamma(t_0))'$ are non-trivial in measure compared to $B_{r}(\gamma(t_1))$ by \eqref{agotd eq6} and property \ref{agotd 1} respectively. The same considerations as in the proof of Lemma \ref{main lemma} gives the existence of $z' \in \tPsi_{t_1-t_0}(D_1)$ and $A_2 \subseteq B_r(\gamma(t_0))'$ so that
\begin{equation}\label{agotd eq17}
\frac{m(A_2)}{m(B_{r}(\gamma(t_0)))} \geq 1 - V(1,10) \; \; \text{and} \; \; dt_2(t_1-t_0)(z',y)\leq \frac{3r}{8} \; \forall y \in A_2,
\end{equation}
after we bound $\epsilon_5$ sufficiently small depending only on $N$ and $\delta$. The latter implies 
\begin{equation}\label{agotd eq18}
\Phi_{s}(A_2) \subseteq B_{2r}(\gamma(t_0+s)) \; \forall s \in [0,t_1-t_0].
\end{equation} 
This finishes the proof of the lemma. 
\end{proof}

Lemmas \ref{at geodesic one direction} and \ref{at geodesic other direction} give
\begin{cor}\label{at geodesic both directions}
For any $\delta \in (0,0.1)$, there exists $\epsilon_6(N,\delta) > 0$ and $\bar{r}_6(N,\delta) > 0$ so that for any unit speed geodesic $\gamma$ from $p$ to $q$, there exists a unit speed geodesic $\gamd$ from $p$ to $q$ with with $\gamd \equiv \gamma$ on $[1-\delta,1]$ so that for all $r \leq \bar{r}_6$  and $\delta \leq t_0 \leq t_1 \leq 1-\delta$, if $t_1-t_0 \leq \epsilon_6$, then
\begin{enumerate}
	\item \label{agbd 1}$V(1, 100)^4 \leq \frac{m(B_r(\gamd(t_1)))}{m(B_r(\gamd(t_0)))} \leq \frac{1}{V(1,100)^4}$;
	\item \label{agbd 2}There exists $A_1 \subseteq B_{r}(\gamd(t_1))$ so that $m(A_1) \geq (1-V(1,10))m(B_r(\gamd(t_1)))$ and $\Psi_{s}(A_1) \subseteq B_{2r}(\gamd(t_1-s))$ for all $s \in [0,t_1-t_0]$;
	\item \label{agbd 3}There exists $A_2 \subseteq B_{r}(\gamd(t_0))$ so that $m(A_2) \geq (1-V(1,10))m(B_r(\gamd(t_0)))$ and $\Phi_{s}(A_2) \subseteq B_{2r}(\gamd(t_0+s))$ for all $s \in [0,t_1-t_0]$. 
\end{enumerate}
\end{cor}

We would like to now construct a geodesic that has this behaviour for all $\delta$ by taking a limit of $\gamma^{\delta}$ as $\delta \to 0$. To have the properties pass over to the limit, we need to make sure each $\gamma^\delta$ also satisfies the above properties for any $\delta' > \delta$. For all $\delta \in (0,0.1)$, we fix the constants $\epsilon_6(N,\delta)$ and $\bar{r}_6(N,\delta)$ which come from taking the minimimum of their respective counterparts in lemmas \ref{at geodesic one direction} and \ref{at geodesic other direction}.
\begin{lem}\label{at geodesic larger delta}
Let $\delta \in (0,0.1)$ and assume some unit speed geodesic $\gamd$ from $p$ to $q$ satisfies properties \ref{agbd 1} - \ref{agbd 3} for $\delta$, $\epsilon_6(N,\delta)$ and $\bar{r}_6(N,\delta)$. Then for any $\delta' \in (\delta,0.1)$, $\gamd$ also satisfies properties \ref{agbd 1} - \ref{agbd 3} for $\delta'$, $\epsilon_6(N,\delta')$ and $\bar{r}_6(N,\delta')$
\end{lem}

\begin{proof}
	Fix $\delta' \in (\delta,0.1)$. Use $\gamd$ to construct some $\gamdp$ with the construction of Lemma \ref{at geodesic one direction}. $\gamdp$ satisfies properties \ref{agbd 1} - \ref{agbd 3} for $\delta'$, $\epsilon_6(N,\delta')$ and $\bar{r}_6(N,\delta')$ by lemmas \ref{at geodesic one direction} and \ref{at geodesic other direction}. Furthermore, $\gamd(t) = \gamdp(t)$ for all $t \in [1-\delta',1]$ by construction. We will show $\gamd(t) = \gamdp(t)$ for all $t \in [\delta',1]$ which will allow us to conclude.

	Assume this is not the case. Define $s_0 := \min\{s \in [\delta',1-\delta']\, : \, \gamd(t) = \gamdp(t)\, \forall \, t \in [s, 1]\}$. By assumption, $\delta' < s_0 \leq 1-\delta'$. Therefore, there exists $t_0 \in [\delta',s_0)$ so that $t_1 - t_0 \leq \min\{\epsilon_6(N, \delta), \epsilon_6(N, \delta')\}$ and $\gamd(t_0) \neq \gamdp(t_0)$. Choose any $r < \min\{\bar{r}_6(N,\delta), \bar{r}_6(N,\delta'), \frac{d(\gamd(t_0), \gamdp(t_0))}{4}\}$ and consider $B_r(\gamd(t_1))$. On one hand, most of $B_r(\gamd(t_1))$ needs to end up in $B_{2r}(\gamd(t_0))$ under $\Psi_{t_1-t_0}$ by property \ref{agbd 2} for $\gamd$. On the other hand, most of $B_r(\gamd(t_1))$ needs to end up in $B_{2r}(\gamdp(t_0))$ under $\Psi_{t_1-t_0}$ by property \ref{agbd 2} for $\gamdp$. These two balls are disjoint and so we have a contradiction. 
\end{proof}

This immediately gives the existence of a geodesic with the desired properties for any $\delta \in (0,0.1)$.

\begin{thm}\label{limit geodesic}
There exists a unit speed geodesic $\lgam$ from $p$ to $q$ so that for any $\delta \in (0,0.1)$, there exists $\epsilon_6(N,\delta)>0$ and $\bar{r}_6(N,\delta)>0$ so that for all $r \leq \bar{r}_6$ and $\delta \leq t_0 \leq t_1 \leq 1-\delta$, if $t_1-t_0 \leq \epsilon_6$, then
\begin{enumerate}
	\item \label{lg 1}$V(1, 100)^4 \leq \frac{m(B_r(\lgam(t_1)))}{m(B_r(\lgam(t_0)))} \leq \frac{1}{V(1,100)^4}$;
	\item \label{lg 2}There exists $A_1 \subseteq B_{r}(\lgam(t_1))$ so that $m(A_1) \geq (1-V(1,10))m(B_r(\lgam(t_1)))$ and $\Psi_{s}(A_1) \subseteq B_{2r}(\lgam(t_1-s))$ for all $s \in [0,t_1-t_0]$;
	\item \label{lg 3}There exists $A_2 \subseteq B_{r}(\lgam(t_0))$ so that $m(A_2) \geq (1-V(1,10))m(B_r(\lgam(t_0)))$ and $\Phi_{s}(A_2) \subseteq B_{2r}(\lgam(t_0+s))$ for all $s \in [0,t_1-t_0]$. 
\end{enumerate}
\end{thm}

\begin{proof}
Fix any unit speed geodesic $\gamma$ from $p$ to $q$ and then use the construction of Lemma \ref{at geodesic one direction} to obtain a $\gamma^\delta$ for each $\delta \in (0,0.1)$. By Arzel\`a-Ascoli theorem, we can take a limit $\lgam$ of $\gamma^\delta$ as $\delta \to 0$ after passing to a subsequence. By Lemma \ref{at geodesic larger delta} and Bishop-Gromov, $\lgam$ will have the desired properties.
\end{proof}

\subsection{Proof of main theorem}
We now prove the H{\"o}lder continuity in pointed Gromov-Hausdorff distance of small balls along the interior of any geodesic between $p$ and $q$ constructed in Theorem \ref{limit geodesic} using essentially the same argument as in \cite{CN12}. 

\begin{thm}\label{section main theorem}
There exists a unit speed geodesic $\gamma$ between $p$ and $q$ so that for any $\delta \in (0,0.1)$, there exists $\epsilon(N,\delta) > 0$, $\bar{r}(N,\delta) >  0$ and $C(N,\delta)$ so that for any $r \leq \bar{r}$ and $t_0, t_1 \in [\delta, 1-\delta]$, if $|t_1-t_0| \leq \epsilon$ then 
\begin{equation*}
	d_{pGH}\bigg((B_{r}(\gamma(t_0)), \gamma(t_0)), (B_{r}(\gamma(t_1)), \gamma(t_1))\bigg) \leq Cr|t_1-t_0|^{\frac{1}{2N(1+2N)}} \footnote{We note that the H\"older exponent is slightly different to that of \cite{CN12} due to a minor error in the first line of equation (3.38).}.
\end{equation*}
\end{thm}

\begin{proof}
Let $\lgam$, $\epsilon_6$ and $\bar{r}_6$ be as in Theorem \ref{limit geodesic} and let $\gamma = \lgam$. Fix $\delta \in (0,0.1)$. We begin by assuming $\epsilon \leq \epsilon_6$ and $\bar{r} \leq \bar{r}_6$ but will impose more bounds as the proof continues. Fix $r \leq \bar{r}$ and $\delta \leq t_0 \leq t_1 \leq 1-\delta$ with $\omega := t_1 - t_0 \leq \epsilon$. 

We begin by showing that a large portion of $B_{r}(\gamma(t_1))$ is maped by $\Psi_{\omega}$ close (on the scale of $r$) to $\gamma(t_0)$, where the closeness and the relative size of the portion are both H\"older dependent on $\omega$. This also shows that the measure of $B_{r}(\gamma)$ is H\"older along $\gamma \rvert_{[\delta,1-\delta]}$ as a consequence. 

Define $\eta:=\omega^{\frac{N}{2(1+2N)}}$ and $\mu := \eta^{\frac{1}{N}} = \omega^{\frac{1}{2(1+2N)}}$. By property \ref{lg 2} of $\gamma$ from Theorem \ref{limit geodesic} and integral Abresch-Gromoll, there exists $B_{\mu r}(\gamma(t_1))' \subseteq B_{\mu r}(\gamma(t_1))$ so that
\begin{enumerate}
	\item $\frac{m(B_{\mu r}(\gamma(t_1))')}{m(B_{\mu r}(\gamma(t_1)))} \leq 1-2V(1,10)$;
	\item $\Psi_{s}(B_{\mu r}(\gamma(t_1))') \subseteq B_{2\mu r}(\gamma(t_1-s)) \; \forall s \in [0, \omega]$; 
	\item $e(x) \leq c(N,\delta)\mu^2r^2 \leq cr^2 \; \forall x \in B_{\mu r}(\gamma(t_1))'$.
\end{enumerate}

By integral Abresch-Gromoll, there exists $B_{r}(\gamma(t_1))' \subseteq B_{r}(\gamma(t_1))$ so that
\begin{equation}\label{smt eq1}
 e(x) \leq c(N,\delta)\frac{1}{\eta}r^2 \; \forall x \in B_{r}(\gamma(t_1))' \; \; \text{and}\; \; \frac{m(B_{r}(\gamma(t_1))')}{m(B_{r}(\gamma(t_1)))} \geq 1-\eta.
\end{equation}

Fix $h^- \equiv h^-_{\rho(4r)^2}$ satisfying statement 4 of Theorem \ref{hess h integral bound} for the balls of radius $4r$ along $\gamma$, where $\rho \in [\frac{1}{2},2]$. For all $s \in [0,\omega]$ and $(x,y)\in  X\times X$, define
\begin{equation}\label{smt eq2}
dt(s)(x,y) := \min\left\{4r, \max\limits_{0 \leq \tau \leq s} |d(x,y) - d(\Psi_\tau(x), \Psi_\tau(y))|\right\}.
\end{equation}
Note for any $(x,y) \in B_{\mu r}(\gamma(t_1))' \times B_{r}(\gamma(t_1))'$, $\Psi_{s}(x), \Psi_{s}(y) \in B_{2r}(\gamma(t_1-s))$ and so $dt(s)(x,y) < 4r$.

By Proposition \ref{distortion bound}, for a.e. $s \in [0,\omega]$, 
\begin{align}\label{smt eq3}
\begin{split}
	&\hspace{0.5cm}\frac{d}{ds} \int\limits_{B_{\mu r}(\gamma(t_1))' \times B_{r}(\gamma(t_1))'} dt(s)(x,y) \, d(m \times m)(x,y)\\
	 &\leq \int\limits_{B_{\mu r}(\gamma(t_1))' \times B_{r}(\gamma(t_1))'} \bigg(|\nabla h^- - \nabla d_p|(\Psi_s(x))+ |\nabla h^- - \nabla d_p|(\Psi_s(y))\bigg)\, d(m \times m)(x,y) \\ 
&\hspace{2cm}+ \int\limits_{0}^{1}  \int\limits_{B_{\mu r}(\gamma(t_1))' \times B_{r}(\gamma(t_1))'} d(\Psi_s(x), \Psi_s(y))|\Hess h^-|_{\HS}(\tilde{\gamma}_{\Psi_s(x),\Psi_s(y)}(\tau)) \,  d(m \times m)(x,y) \, d\tau.
\end{split} 
\end{align}

We estimate
\begin{align}\label{smt eq4}
\begin{split}
	&\hspace{0.5cm}\int\limits_{0}^{1}  \int\limits_{B_{\mu r}(\gamma(t_1))' \times B_{r}(\gamma(t_1))'} d(\Psi_s(x), \Psi_s(y))|\Hess h^-|_{\HS}(\tilde{\gamma}_{\Psi_s(x),\Psi_s(y)}(\tau)) \,  d(m \times m)(x,y) \, d\tau\\
	&\leq c(N,\delta)\int\limits_{0}^{1}  \int\limits_{(\Psi_s, \Psi_s)(B_{\mu r}(\gamma(t_1))' \times B_{r}(\gamma(t_1))')} d(x, y)|\Hess h^-|_{\HS}(\tilde{\gamma}_{x,y}(\tau)) \,  d(m \times m)(x,y) \, d\tau \; \; \text{, by Theorem \ref{-dp grad flow properties}, 2}\\
	&\leq c(N,\delta)rm(B_{2r}(\gamma(t_1-s)))\int\limits_{B_{4r}(\gamma(t_1-s))} |\Hess h^-|_{\HS}\,  dm \; \; \text{, by segment inequality \ref{segment inequality}}\\
	&\leq c(N,\delta)rm(B_{r}(\gamma(t_1)))^2 \fint\limits_{B_{4r}(\gamma(t_1-s))} |\Hess h^-|_{\HS} \, dm \; \; \text{, by Bishop-Gromov and property \ref{lg 1} of $\gamma$}.
\end{split}
\end{align}
Integrating in $s \in [0,\omega]$, 
\begin{align}\label{smt eq5}
\begin{split}
	&\hspace{0.5cm} \int\limits_{0}^{\omega} \bigg( \int\limits_{0}^{1} \int\limits_{B_{\mu r}(\gamma(t_1))' \times B_{r}(\gamma(t_1))'} d(\Psi_s(x), \Psi_s(y))|\Hess h^-|_{\HS}(\tilde{\gamma}_{\Psi_s(x),\Psi_s(y)}(\tau)) \,  d(m\times m)(x,y)\, d\tau \bigg)\, ds\\
	&\leq crm(B_{r}(\gamma(t_1)))^2 \int\limits_{0}^{\omega} \fint\limits_{B_{4r}(\gamma(t_1-s))} |\Hess h^-|_{\HS} \, dm \, ds\\
	&\leq c(N,\delta)m(B_{r}(\gamma(t_1)))^2 \sqrt{\omega}r,
\end{split}
\end{align}
where the last line follows from the definition of $h^-$, statement 4 of Theorem \ref{hess h integral bound}. and Cauchy-Schwarz.

By statement 3 of \ref{grad h geodesic bound} and the excess bound on the elements of $B_{\mu r}(\gamma(t_1))'$,
\begin{equation}\label{smt eq6}
	\int\limits_{0}^{\omega} \int\limits_{B_{\mu r}(\gamma(t_1))' \times B_{r}(\gamma(t_1))'} |\nabla h^- - \nabla d_p|(\Psi_s(x)) \, d(m \times m)(x,y)  \, ds \leq c(N,\delta)m(B_r(\gamma(t_1)))^2\sqrt{\omega}r.
\end{equation}
Similarly by the excess bound on the elemnts of $B_{r}(\gamma(t_1))'$ \eqref{smt eq1},
\begin{align}\label{smt eq7}
\begin{split}
	&\hspace{0.5cm} \int\limits_{0}^{\omega} \int\limits_{B_{\mu r}(\gamma(t_1))' \times B_{r}(\gamma(t_1))'} |\nabla h^- - \nabla d_p|(\Psi_s(y)) \, d(m \times m)(x,y)  \, ds\\ &\leq c(N,\delta)\frac{1}{\sqrt{\eta}}m(B_r(\gamma(t_1)))m(B_{\mu r}(\gamma(t_1)))\sqrt{\omega}r.
\end{split}
\end{align}

Combining  \eqref{smt eq5} - \eqref{smt eq7} with \eqref{smt eq3} we immediately obtain
\begin{align}\label{smt eq8}
\begin{split}
&\hspace{0.5cm}\int\limits_{B_{\mu r}(\gamma(t_1))' \times B_{r}(\gamma(t_1))'} dt(\omega)(x,y) \, d(m \times m)(x,y)\\ 
&= \int_{0}^{\omega} [\frac{d}{ds} \int\limits_{B_{\mu r}(\gamma(t_1))' \times B_{r}(\gamma(t_1))'} dt(s)(x,y) \, d(m \times m)(x,y)] \, ds\\ &\leq c(N,\delta)m(B_{r}(\gamma(t_1)))\bigg(m(B_{r}(\gamma(t_1)))+\frac{1}{\sqrt{\eta}}m(B_{\mu r}(\gamma(t_1))) \bigg)\sqrt{\omega}r.
\end{split}
\end{align}
By Bishop-Gromov, $\frac{m(B_{\mu r}(\gamma(t_1)))}{m(B_{r}(\gamma(t_1)))} \geq C(N)\mu^N = c\eta$. Therefore, there exists $z \in B_{\mu r}(\gamma(t_1))'$ so that 
\begin{equation*}
	\int\limits_{B_{r}(\gamma(t_1))'} dt(\omega)(z,y) \, dm(y) \leq  c(N,\delta)m(B_{r}(\gamma(t_1)))\bigg(\frac{1}{\eta}+\frac{1}{\sqrt{\eta}}\bigg)\sqrt{\omega}r \leq c(N,\delta)m(B_{r}(\gamma(t_1)))\frac{1}{\eta}\sqrt{\omega}r.
\end{equation*}

This means there exists $D \subseteq B_{r}(\gamma(t_1))'$ so that 
\begin{equation}\label{smt eq9}
	\frac{m(D)}{m(B_{r}(\gamma(t_1)))} \geq 1-2\eta \; \; \text{and} \; \; dt(\omega)(z,y)\leq c(N,\delta)\frac{1}{\eta^2}\sqrt{\omega}r \; \forall y \in D.
\end{equation}
Since $\omega = (\eta^2 \mu)^2$, the latter combined with $z \in B_{\mu r}(\gamma(t_1))'$ implies 
\begin{equation}\label{smt eq9.5}
	\Psi_{\omega}(D) \subseteq B_{(1+c(N,\delta)\mu)r}(\gamma(t_0)).
\end{equation}

We have the following estimate on the volume of $B_{r}(\gamma(t_1))$ compared to volume of $B_{r}(\gamma(t_0))$ after possibly constraining $\epsilon$ further depending on $N$ and $\delta$.
\begin{align}\label{smt eq10}
\begin{split}
	&\frac{m(B_{r}(\gamma(t_1)))}{m(B_{r}(\gamma(t_0)))} \leq (1+c\eta)\frac{m(D)}{m(B_{r}(\gamma(t_0)))} \; \; \text{, by \eqref{smt eq9}}\\
	&\leq (1+c\eta)(1+c(N,\delta)\omega)^N\frac{m(\Psi_{\omega}(D))}{m(B_{r}(\gamma(t_0)))} \; \; \text{, by Theorem \ref{-dp grad flow properties}, 2}\\
	&\leq (1+c\eta)(1+c\omega)^N(1+c(N,\delta)\mu)^N\frac{m(\Psi_{\omega}(D))}{m(B_{(1+c\mu)r}(\gamma(t_0)))} \; \; \text{, by Bishop-Gromov}\\
	&\leq 1+c(N,\delta)\mu = 1+c\omega^{\frac{1}{2(1+2N)}}.
\end{split}
\end{align}
Making the same calculation with $\Phi$ in the other direction as well, we obtain the following H\"older estimate on volume
\begin{equation}\label{smt eq11}
	\bigg| \frac{m(B_{r}(\gamma(t_1)))}{m(B_{r}(\gamma(t_0)))}-1 \bigg| \leq c(N,\delta)|t_1-t_0|^{\frac{1}{2(1+2N)}}.
\end{equation}

We now show the required Bishop-Gromov approximation can be constructed by using $\Psi_{\omega}$ on a $c\mu r$-dense subset of $B_{r}(\gamma(t_1))$. 

Fix representatitves for $|\Hess h^-|_{\HS}$ and $|\nabla h^- - \nabla d_p|$. Using the same calculation as before,
\begin{align}\label{smt eq12}
\begin{split}
	&\hspace{0.5cm} \int\limits_{0}^{\omega} \bigg( \int\limits_{0}^{1} \int\limits_{D \times D} d(\Psi_s(x), \Psi_s(y))|\Hess h^-|_{\HS}(\tilde{\gamma}_{\Psi_s(x),\Psi_s(y)}(\tau)) \,  d(m\times m)(x,y)\, d\tau \bigg)\, ds\\
	&\leq c(N,\delta)m(B_{r}(\gamma(t_1)))^2 \sqrt{\omega}r,
\end{split}
\end{align}
and so by Fubini's theorem, there exists $A \subseteq D$ so that
\begin{enumerate}
	\item $\frac{m(A)}{m(B_{r}(\gamma(t_1)))} \geq 1-3\eta$;
	\item $\int\limits_{0}^{\omega} \bigg( \int\limits_{0}^{1} \int\limits_{D} d(\Psi_s(x), \Psi_s(y))|\Hess h^-|_{\HS}(\tilde{\gamma}_{\Psi_s(x),\Psi_s(y)}(\tau)) \,  dm(y)\, d\tau \bigg)\, ds \leq c\frac{1}{\eta}m(B_{r}(\gamma(t_1)))\sqrt{\omega}r$ for all $x \in A$.
\end{enumerate}

For each $x \in A$, there exists $A_x \subseteq D$ so that
\begin{enumerate}
	\item $\frac{m(A_x)}{m(B_{r}(\gamma(t_1)))} \geq 1-3\eta$;
	\item $\int\limits_{0}^{\omega} \bigg( \int\limits_{0}^{1} d(\Psi_s(x), \Psi_s(y))|\Hess h^-|_{\HS}(\tilde{\gamma}_{\Psi_s(x),\Psi_s(y)}(\tau)) \, d\tau \bigg)\, ds\ \leq c\frac{1}{\eta^2}\sqrt{\omega}r$ for all $y \in A_x$.
\end{enumerate}

Since $A$ and each $A_x$ are contained in $B_{r}(\gamma(t_1))'$, their elements have an excess bound of $c\frac{1}{\eta}r^2$ by \eqref{smt eq1}. By statement 3 of \ref{grad h geodesic bound}, for $m$-a.e. $x \in A$,
\begin{equation}\label{smt eq13}
	\int\limits_{0}^{\omega} |\nabla h^- - \nabla d_p|(\Psi_s(x)) ds \leq c(N,\delta)\frac{1}{\sqrt{\eta}}\sqrt{\omega}r.
\end{equation}
Similarly, for each $x \in A$ and $m$-a.e. $y \in A_x$,
\begin{equation}\label{smt eq14}
	\int\limits_{0}^{\omega} |\nabla h^- - \nabla d_p|(\Psi_s(y)) ds \leq c(N,\delta)\frac{1}{\sqrt{\eta}}\sqrt{\omega}r.
\end{equation}

By a Fubini's theorem argument, it is clear that the inequality in Proposition \ref{distortion bound} holds pointwise for $(m \times m)$-a.e. $(x,y)$. We first replace $A$ with a full measure subset so that in addition the inequality in Proposition \ref{distortion bound} holds for all $x \in A$ and $m$-a.e. $y \in A_x$. We then replace each $A_x$ with a full measure subset so that the same inequality holds for all $x \in A$ and all $y \in A_x$. Therefore, for all $x \in A$, $y \in A_x$, 
\begin{equation*}
	dt(\omega)(x,y) \leq c(N,\delta)\bigg(\frac{1}{\eta^2}+\frac{1}{\sqrt{\eta}}\bigg)\sqrt{\omega}r \leq c(N,\delta)\frac{1}{\eta^2}\sqrt{\omega}r \leq c\mu r.
\end{equation*}

For any $x, y \in A$, $A_x \cap A_y$ is $c(N)\eta^{\frac{1}{n}}r$-dense in $B_{r}(\gamma(t_1))$ and so there exists some $z \in A_x \cap A_y$ where $d(x,z) < c\eta^{\frac{1}{n}}r =  c\mu r$. Therefore, 
\begin{align}\label{smt eq15}
\begin{split}
	&\hspace{0.5cm} |d(\Psi_{\omega}(x), \Psi_{\omega}(y))-d(x,y))|\\ &\leq |d(\Psi_{\omega}(x), \Psi_{\omega}(y))-d(\Psi_{\omega}(z), \Psi_{\omega}(y)|+|d(\Psi_{\omega}(z), \Psi_{\omega}(y))-d(z,y)|+|d(z,y)-d(x,y)|\\ &\leq c(N,\delta)\mu r.
\end{split}
\end{align}
Moreover, we have the following estimate on the volume of $\Psi_{\omega}(A) \subseteq B_{(1+c\mu)r}(\gamma(t_0))$ after possibly constraining $\epsilon$ further depending on $N$ and $\delta$.
\begin{align}\label{smt eq16}
\begin{split}
	&\frac{m(\Psi_{\omega}(A))}{m(B_{1+c\mu r}(\gamma(t_0)))} \geq \frac{1}{(1+c(N,\delta)\omega)^N}\frac{m(A)}{m(B_{(1+c\mu)r}(\gamma(t_0)))} \; \; \text{, by Theorem \ref{-dp grad flow properties}, 2}\\
	&\geq \frac{1}{(1+c\omega)^N} \frac{1}{(1+c(N,\delta)\mu)^N}\frac{m(A)}{m(B_{r}(\gamma(t_0)))} \; \; \text{, by Bishop-Gromov}\\
	&\geq \frac{1}{(1+c\omega)^N} \frac{1}{(1+c\mu)^N}\frac{1}{1+c(N,\delta)\mu}\frac{m(A)}{m(B_{r}(\gamma(t_1)))} \; \; \text{, by \eqref{smt eq10}}\\
	&\geq \frac{1}{(1+c\omega)^N} \frac{1}{(1+c\mu)^N}\frac{1}{1+c\mu}(1-3\eta) \geq 1-c(N,\delta)\mu.
\end{split}
\end{align}

To summarize, $A \subseteq B_{r}(\gamma(t_1))$ is so that
\begin{enumerate}
	\item $\Psi_{\omega}(A) \subseteq B_{(1+c(N,\delta)\mu)r}(\gamma(t_0))$;
	\item $\forall x, y \in A$, $|d(\Psi_\omega(x),\Psi_\omega(y))-d(x,y)| \leq c(N,\delta)\mu r$;
	\item $A$ is $c(N)\mu r$-dense in $B_{r}(\gamma(t_1))$;
	\item $\Psi_{\omega}(A)$ is $c(N,\delta)\mu^{\frac{1}{N}}r$-dense in $B_{(1+c(N,\delta)\mu)r}(\gamma(t_0))$.
\end{enumerate}
Moreover, there exists $c(N)$ so that $m(B_{c(N)\mu r}(\gamma(t_1))) \geq \frac{2\eta}{1-V(1,10)}m(B_{r}(\gamma(t_1)))$ by Bishop-Gromov. By property \ref{lg 2} of $\gamma$, there exists $B_{c\mu r}(\gamma(t_1))' \subseteq B_{c\mu r}(\gamma(t_1))$ so that 
\begin{equation*}
\frac{m(B_{c\mu r}(\gamma(t_1))')}{m(B_{c\mu r}(\gamma(t_1)))} \geq 1-V(1,10) \; \; \text{and} \; \; \Psi_{\omega}(B_{c\mu r}(\gamma(t_1))') \subseteq B_{2c\mu r}(\gamma(t_0)).
\end{equation*}
Therefore, $B_{c(N)\mu r}(\gamma(t_1))' \cap A$ is non-empty by measure considerations. In other words, there is an element in $A$ which is $c(N) \mu r$ close to $\gamma(t_1)$ and is mapped $2c(N) \mu r$ close to $\gamma(t_0)$ under $\Psi_\omega$.

These facts about $A$ allow for the construction of a $c(N,\delta)\mu^{\frac{1}{N}}r$ pointed Gromov-Hausdorff approximation which finishes the proof. 
\end{proof}

Before we prove Theorem \ref{main theorem 1}, we will first prove that $\RCD(K,N)$ spaces are non-branching in the next section using the construction we have developed so far. A corollary then follows which immediately gives Theorem \ref{main theorem 1}.

\section{Applications}\label{section 6}

\subsection{Non-branching} \label{subsection 6.1}
In this subsection, we prove that $\RCD(K,N)$ spaces are non-branching. The use of the essentially non-branching property of $\RCD(K,N)$ spaces in the proof was pointed out to the author by Vitali Kapovitch. 

\begin{proof}[Proof of Theorem \ref{no branch}]
Assume otherwise. By zooming in and cutting off geodesics if necessary, we may assume $(X,d,m)$ is an $\RCD(-(N-1),N)$ space for some $N \in (1,\infty)$ and we have two unit speed geodesics $\gamma_1, \gamma_2:[0,1] \to X$ with
\begin{enumerate}
	\item $\gamma_1(0) = \gamma_2(0) = p$ for $p \in X$;
	\item $\gamma_1(1) = q_1$ and $ \gamma_2(1) = q_2$ for $q_1, q_2 \in X$;
	\item $\max\{t\in[0,1]\, : \, \gamma_1(s) = \gamma_2(s) \; \forall s \in [0,t]\} = 0.5$. 
\end{enumerate}
Let $p' = \gamma_1(0.5)$ and $\Psi_{s}$ be as in \eqref{Psidef} towards $p$. 

Since $(X,2d,m)$ is again an $\RCD(-(N-1),N)$ space and $2d(p,p')=1$, we may apply Theorem \ref{limit geodesic} to obtain a $2d$-unit speed geodesic $\tgam:[0,1] \to X$ between $p$ and $p'$. Reparameterize $\tgam$ to $\gamma: [0,0.5] \to X$ so that $\gamma$ is a $d$-unit speed geodesic. 

Fix any $\delta \in (0,0.1)$, use Corollary \ref{at geodesic both directions} to construct a unit speed geodesic $\gamma_1^{\delta}:[0,1] \to X$ from $p$ to $q_1$ with $\gamma^{\delta}_1(t)=\gamma(t)$ for all $t \in [0,\delta]$. Therefore, the proof of Theorem \ref{section main theorem} passes for $\gamma_1^{\delta}$ for the same $\delta$ and in particular we have the estimates \eqref{smt eq9} - \eqref{smt eq11} for $\gamma_1^{\delta}$. As a reminder, for $\delta \leq s_1 < s_2 \leq 1-\delta$ and sufficiently small $r$,
\begin{itemize}
	\item \eqref{smt eq9} and \eqref{smt eq9.5} imply a portion of $B_{r}(\gamma_1^\delta(s_2))$ is sent to a ball of radius slightly larger than $r$ around $\gamma_1^\delta(s_1)$ by $\Psi_{s_2-s_1}$, where the relative size of the portion and the increase in radius on the scale of $r$ can be both made H\"older dependent on $s_2 - s_1$ and go uniformly to $1$ and $0$ respectively as $s_2 - s_1 \to 0$. 
	\item \eqref{smt eq11} implies the ratio between the measures of $B_{r}(\gamma^{\delta}_1(s_1))$ and $B_{r}(\gamma^{\delta}_1(s_2))$ is H\"older dependent on $s_2 - s_1$ and in particular goes uniformly to $1$ as $s_2 - s_1\to 0$
\end{itemize}

We show that $\gamma_1^{\delta}(t) = \gamma(t)$ for all $t \in [0,0.5]$. Suppose not, let $t_0 := \max\{t\in[0,0.5]\, : \, \gamma_1^{\delta}(s) = \gamma(s) \; \forall s \in [0, t]\}$ and so $t_0 \in [\delta, 0.5)$. 

We claim there exists $t_1 \in (t_0, 0.5)$ and $\bar{r} > 0$ so that for any $r \leq \bar{r}$, there exists $A_1 \subseteq B_{r}(\gamma_1^{\delta}(t_1))$ and $A_2 \subseteq B_{r}(\gamma(t_1))$ so that 
\begin{enumerate}
	\item $\gamma_1^{\delta}(t_1) \neq \gamma(t_1)$.
	\item $\Psi_{t_1-t_0}(A_1) \subseteq B_{r}(\gamma_1^{\delta}(t_0))$ and $\Psi_{t_1-t_0}(A_2) \subseteq B_{r}(\gamma(t_0))$;
	\item $\frac{m(\Psi_{t_1-t_0}(A_1))}{m( B_{r}(\gamma_1^{\delta}(t_0)))} > \frac{1}{2}$ and $\frac{m(\Psi_{t_1-t_0}(A_2))}{m( B_{r}(\gamma(t_0)))} > \frac{1}{2}$.
\end{enumerate}
We can choose $t_1$ arbitrarily close to $t_0$ so that statement 1 holds by definition of $t_0$. Statements 2 and 3 then follow from \eqref{smt eq9} - \eqref{smt eq11} for $\gamma_1^{\delta}$, the same for $\tgam$, Bishop-Gromov and statement 2 of Theorem \ref{-dp grad flow properties} to control the volume distortion of $\Psi$, as soon as $t_1$ is chosen close enough to $t_0$.  Choosing $r \leq \min\{\bar{r}, d(\gamma_1^{\delta}(t_1), \gamma(t_1))\slash2\}$ then leads to a contradiction with Theorem \ref{directional BG}. To be precise, Theorem \ref{directional BG} can be used to show that the subset of points $x \in X$ where a geodesic from $p$ to $x$ can be extended to two branching geodesics must be measure 0, see \cite[Proposition 4.5]{C14}, which gives the contradiction. 

Therefore, $\gamma_1^{\delta}(t) = \gamma(t)$ for all $t \in [0,0.5]$. Since this is true for all $\delta \in (0,0.1)$, taking $\delta \to 0$ and using Arzel\`a-Ascoli Theorem, after possibly passing to a subsequence, we obtain a geodesic $\bar{\gamma}_1$ satisfying Theorem \ref{limit geodesic} with $\bar{\gamma}_1 \equiv \gamma$ on $[0, 0.5]$ and $\bar{\gamma}_1(1) = q_1$. The same construction for $\gamma_2$ gives $\bar{\gamma}_2$ satisfying Theorem \ref{limit geodesic} with $\bar{\gamma}_2 \equiv \gamma$ on $[0, 0.5]$ and $\bar{\gamma}_2(1) = q_2$. Applying the previous argument again for $\bar{\gamma}_1$ and $\bar{\gamma}_2$ shows that they cannot split, which is a contradiction. 
\end{proof}

As a corollary, we have the following improvement of Theorem \ref{section main theorem}.
\begin{cor}\label{section main theorem 2}
Let $(X,d,m)$ be an $\RCD(-(N-1),N)$ space for some $N \in (1,\infty)$ and $p,q \in X$ with $d(p,q) = 1$. For any $\delta \in (0,0.1)$, there exists $\epsilon(N,\delta) > 0$, $\bar{r}(N,\delta) >  0$ and $C(N,\delta)$ so that for any unit speed geodesic $\gamma$ between $p$ and $q$, $r \leq \bar{r}$, and $t_0, t_1 \in [\delta, 1-\delta]$, if $|t_1-t_0| \leq \epsilon$ then 
\begin{equation*}
	d_{pGH}\bigg((B_{r}(\gamma(t_0)), \gamma(t_0)), (B_{r}(\gamma(t_1)), \gamma(t_1)))\bigg) \leq Cr|t_1-t_0|^{\frac{1}{2N(1+2N)}}.
\end{equation*}
\end{cor}

\begin{proof}
Fix any $s \in [0,0.5]$. Since $(X,2d,m)$ is again an $\RCD(-(N-1),N)$ space and $2d(\gamma(s), \gamma(s+0.5))=1$, we may use Theorem \ref{section main theorem} to construct some 2d-unit speed geodesic $\gamma^s$ between $\gamma(s)$ and $\gamma(s+0.5)$. Since $X$ is non-branching, $\gamma^s$ and $\gamma$ must coincide between $\gamma(s)$ and $\gamma(s+0.5)$. Since this is true for all $s\in[0,0.5]$ and all $\gamma^s$ have the H\"older properties of Theorem \ref{section main theorem} on $(X,2d,m)$. The same is true for $\gamma$ on $(X,d,m)$ with slightly worse constants. 
\end{proof}




Theorems \ref{main theorem 1} now follows immediately by rescaling.

\subsection{Dimension and weak convexity of the regular set}\label{subsection 6.2} In this subsection we will extend to the $\RCD(K,N)$ setting the results of \cite{CN12} on regular sets. All proofs translate directly from \cite{CN12}. We mention again that Theorem \ref{unique local dimension 2} has already been established using a new argument involving the Green's function in \cite{BS20}.

\begin{thm}\label{unique local dimension 2} (Constancy of the dimension) Let $(X,d,m)$ be an $\RCD(K,N)$ m.m.s. for some $K \in \mathbb{R}$ and $N \in (1, \infty)$. Assume $X$ is not a point. There exists a unique $n \in \mathbb{N}$, $1 \leq n \leq N$ so that $m(X\backslash \mathcal{R}_{n})=0$.
\end{thm}

\begin{proof}
Let $A^1, A^2 \subseteq X \times X$ be the sets of $(x,y) \in X \times X$ so that geodesics from $x$ to $y$ are extendible past $x$ and $y$ respectively. For each $x \in X$, let $A_{x}$ be the set of $y \in X$ so that geodesics from $x$ to $y$ are extendible past $y$. Using the arguments of \cite[Section 4]{C14},  $A^1$, $A^2$ are $(m \times m)$-measurable and $A_x$ is $m$-measurable for all $x \in X$. $m(X \backslash A_x) = 0$ for any $x \in X$ by a standard argument using Bishop-Gromov. Let $A := A^1 \cap A^2$, Fubini's theorem then gives $(m \times m)((X \times X) \backslash A) = 0$. 

Let $\gamma_{x,y}:[0,1] \to X$ be a Borel selection (\ref{a.e. unique geodesic}) of constant speed geodesics from any $x \in X$ to any $y \in X$. Since $m(\mathcal{S}) = 0$, it follows from applying the segment inequality to the characteristic function of $\mathcal{S}$ that for $(m \times m)$-a.e. $(x,y) \in X \times X$, $\gamma_{x,y} \cap \mathcal{R}_{\reg}$ has full measure, and therefore is also dense, in $[0,1]$. By Theorem \ref{main theorem 1}, for any geodesic $\gamma$ and $k$, $\gamma \cap \mathcal{R}_k$ is closed relative to the interior of $\gamma$. Combining these with the fact that almost every $\gamma_{x,y}$ is extendible, we obtain for $(m \times m)$-a.e. $(x,y) \in X \times X$, there exists $k \in \mathbb{N}$ with $1 \leq k \leq N$ so that $\gamma_{x,y} \subseteq \mathcal{R}_k$. This leads to a contradiction if there are two regular sets of different dimension with positive measure. 
\end{proof}

\begin{de}\label{def a.e. convexity}($m$-a.e. convexity)
	Let $(X,d,m)$ be a m.m.s.. Let $S$ be an $m$-measurable set in $X$. $S$ is \textit{$m$-a.e. convex} iff for $(m \times m)$ almost every pair $(x,y) \in S \times S$, there exists a minimizing geodesic $\gamma \subseteq S$ connecting $x$ and $y$. 
\end{de}

\begin{de}\label{def weak convexity}(weak convexity)
	Let $(X,d)$ be a metric space. $S \subseteq X$ is \textit{weakly convex} iff for all $(x,y) \in S\times S$ and $\epsilon > 0$, and there exists an $\epsilon$-geodesic (see Definition \ref{epsilon geodesic}) $\gamma \subseteq S$ connecting $x$ and $y$.
\end{de}

\begin{thm}\label{a.e. and weak convexity} ($m$-a.e. and weak convexity of the regular set) Let $\mathcal{R}_n$ be as in Theorem \ref{unique local dimension 2}, then
	\begin{enumerate}
		\item $\mathcal{R}_n$ is $m$-a.e. convex;
		\item $\mathcal{R}_n$ is weakly convex.
	\end{enumerate}
In particular, $\mathcal{R}_n$ is connected. 
\end{thm}

\begin{proof}
Statement 1 is contained in the proof of Theorem \ref{unique local dimension 2}. The proof of statement 2 follows verbatim from \cite[Theorem 1.20]{CN12}.
\end{proof}

\bibliographystyle{amsalpha}
\bibliography{new}

\end{document}